\documentclass{article}

\usepackage{microtype}
\usepackage{graphicx}
\usepackage{booktabs} 

\usepackage{hyperref}

\usepackage[accepted]{icml2023}

\usepackage{amsmath}
\usepackage{amssymb}
\usepackage{mathtools}
\usepackage{amsthm}

\usepackage[capitalize,noabbrev]{cleveref}

\theoremstyle{plain}
\newtheorem{theorem}{Theorem}[section]
\newtheorem*{theorem*}{Theorem}

\newtheorem{lemma}[theorem]{Lemma}

\theoremstyle{definition}

\newtheorem{assumption}[theorem]{Assumption}
\theoremstyle{remark}

\usepackage[textsize=tiny]{todonotes}

\usepackage{tabularray}
\usepackage{amsthm}
\usepackage{subcaption}
\usepackage{hyperref}
\usepackage{enumitem}
\newcounter{Ax}

\usepackage{tablefootnote}
\usepackage{amsmath}
\usepackage{amssymb}
\usepackage{mathtools}

\icmltitlerunning{Optimal Shrinkage for Distributed Second-Order Optimization}

\begin{document}

\twocolumn[
\icmltitle{Optimal Shrinkage for Distributed Second-Order Optimization}



\icmlsetsymbol{equal}{*}

\begin{icmlauthorlist}
\icmlauthor{Fangzhao Zhang}{sch}
\icmlauthor{Mert Pilanci}{sch}
\end{icmlauthorlist}

\icmlaffiliation{sch}{Department of Electrical Engineering, Stanford University}

\icmlcorrespondingauthor{Fangzhao Zhang}{zfzhao@stanford.edu}

\icmlkeywords{Machine Learning, ICML}

\vskip 0.3in
]



\printAffiliationsAndNotice{}  

\begin{abstract}
In this work, we address the problem of Hessian inversion bias in distributed second-order optimization algorithms. We introduce a novel shrinkage-based estimator for the resolvent of gram matrices which is asymptotically unbiased, and characterize its non-asymptotic convergence rate in the isotropic case. We apply this estimator to bias correction of Newton steps in distributed second-order optimization algorithms, as well as randomized sketching based methods. We examine the bias present in the naive averaging-based distributed Newton's method using analytical expressions and contrast it with our proposed bias-free approach. Our approach leads to significant improvements in convergence rate compared to standard baselines and recent proposals, as shown through experiments on both real and synthetic datasets.
\end{abstract}

\section{Introduction}\label{intro}

In a distributed setting, where multiple agents have access only to subsets of the entire dataset, accurate estimation of the Hessian inverse and Newton steps is crucial for the effective application of second-order optimization algorithms. A straightforward way for estimating Hessian inverse is to simply collect and average all local Hessian inverses, however, this is usually not accurate due to the existence of inversion bias, i.e., in general
\[\lim_{m\rightarrow \infty} \left\|\frac{1}{m}\sum_{i=1}^m H_i^{-1} -H^{-1}\right\|\neq 0\]
where $m$ represents the total number of agents, such as distributed workers, and $H_i$ is the local Hessian computed by worker $i$ (see Theorem \ref{thm2.6} for more details). Therefore, a naive averaging of local Hessian inverses leads to a biased estimator of the global Hessian inverse. As a result, Newton steps computed by averaging local Newton steps can be far from exact. 


Different ways to reduce the inversion bias mentioned above have already been studied by a line of prior work. A particularly related one is the determinantal averaging method proposed by \citet{dereziński2019distributed}. The authors show that $\sum_{i=1}^m \mbox{det}(H_i)H_i^{-1}/\sum_{i=1}^m \mbox{det}(H_i)$ serves as an unbiased estimator of the global Hessian inverse when the data is uniformly distributed to each agent. However, this method has shortcomings involving the overhead of computing local Hessian determinants, and potential numerical instabilities in computing  determinants of local Hessian when the data dimensions are large. 


In this work, we borrow tools from random matrix theory and study the problem of estimation of the covariance resolvent when the data is randomly distributed. A key observation is that the inverses of positive semidefinite Hessians typically have the form of a covariance resolvent $(\frac{1}{n} X^TD^2X+\lambda I)^{-1}$ where $\frac{1}{n} X^TD^2X$ is the empirical covariance of an appropriately chosen  matrix in which $D$ is diagonal and $X$ is a data matrix. We propose an asymptotically unbiased estimator of the covariance resolvents in the form of a shrinkage formula (Theorem \ref{thm2}) whose informal version is stated below, and we also characterize its non-asymptotic convergence rate (see Section \ref{con_rate}).
Specifically, under some weak assumptions on the data distribution, we have the following result:
\begin{theorem*}{(informal, see Theorem \ref{thm2} for the assumptions)} 

For a random data matrix $A\in \mathbb{R}^{n\times d}$, let $d_\lambda$ denote the effective dimension of the true covariance matrix $\Sigma_n$, and $\hat \Sigma_n$ denote the empirical covariance matrix. Then, we have
\[\lim_{\substack{n,d\rightarrow \infty \\ \frac{d}{n}\rightarrow y\in [0,1)}} \left\|\mathbb{E}\left[\left(\gamma\hat \Sigma_n+\lambda I\right)^{-1}\right]-(\Sigma_n+\lambda I)^{-1}\right\|=0,\]
where $\gamma=\frac{1}{1-\frac{d_\lambda}{n}}$.
\end{theorem*}
This result implies that a simple scaling of the Hessian by $\frac{1}{1-\frac{d_\lambda}{n}}$ removes the inversion bias, where $d_\lambda$ is the effective dimension of the covariance.
Since the Hessian inverse is related to covariance resolvents, this theorem can help reducing Hessian inversion bias in the large data regime. We study its application to distributed second-order optimization algorithms and randomized second-order optimization algorithms,  where
we observe a significant speedup in the convergence rate compared to baseline methods (see Figure \ref{fig1} below,
and more simulation results in Section \ref{simu} and Appendix \ref{append_simu}). 
\begin{figure}[h]
\begin{subfigure}{0.48\linewidth}
\includegraphics[width=\linewidth]{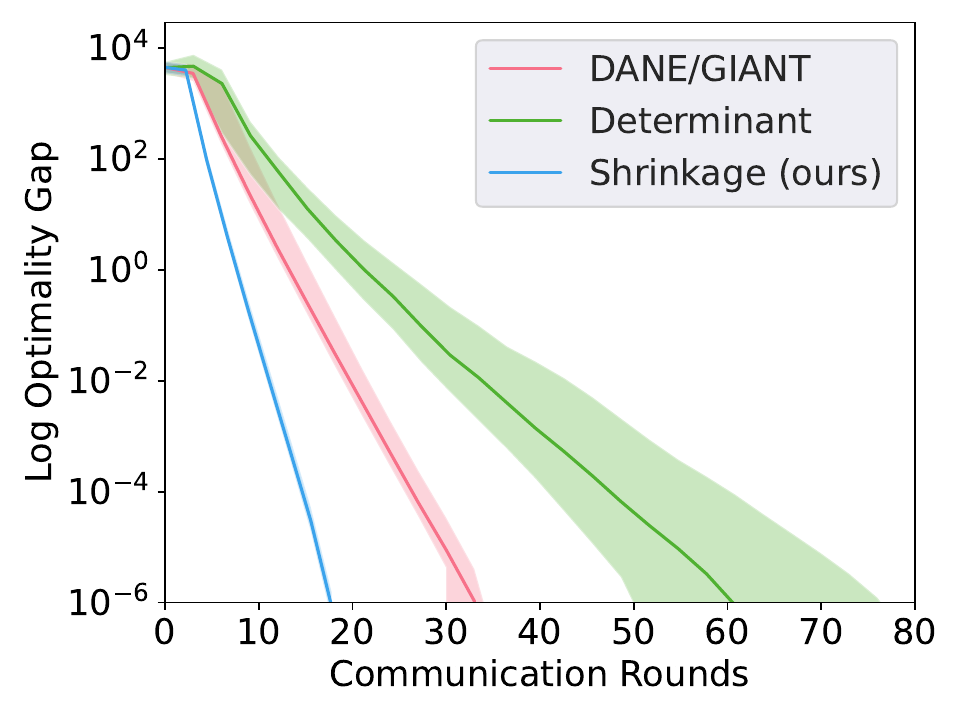}
\end{subfigure}
\hfill
\begin{subfigure}{0.48\linewidth}
\includegraphics[width=\linewidth]{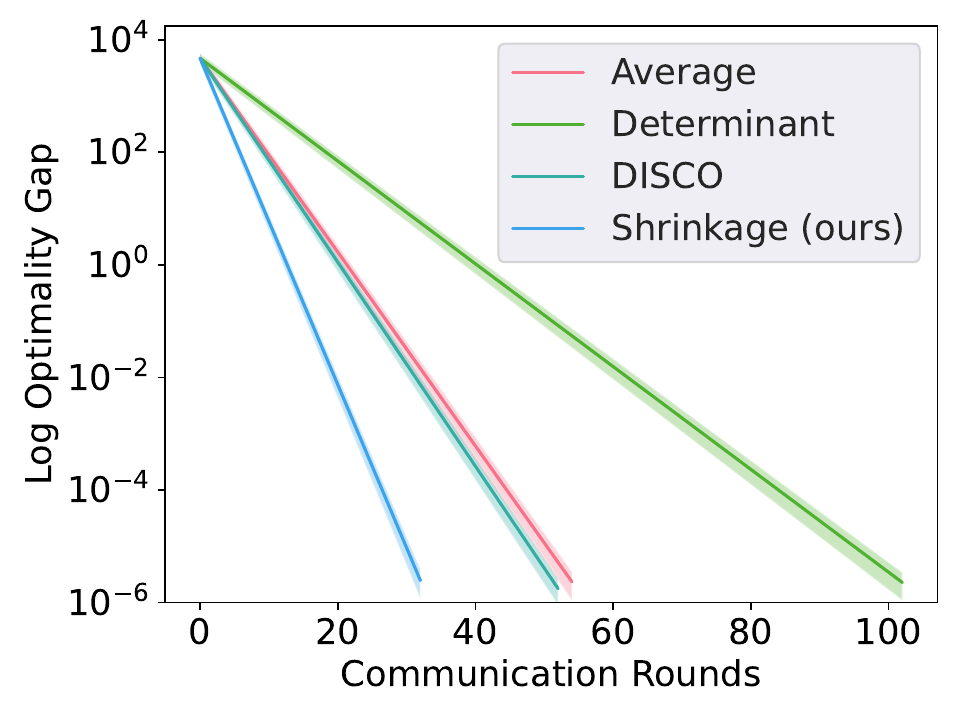}
\end{subfigure}
\caption{Synthetic data experiments on ridge regression. Total number of data $n=30000$, data dimension $d=150$, number of agents $m=200$, regularizer $\lambda=0.01$. The left plot shows the convergence of distributed Newton's method (Algorithm \ref{alg:newton}). The right plot shows the convergence of distributed inexact Newton's method 
 (Algorithm \ref{alg:pcg}). Step sizes are chosen via line search in all methods. See Section \ref{simu} for details.} \label{fig1}
\end{figure}

\vspace{-0.5cm}
\subsection{Prior Work}\label{prior-work}




Early work on distributed Newton-type methods including DANE  studied by \citet{shamir2013communication}, where the approximate Newton step is an average of local mirror descent steps. Later work such as AIDE \cite{reddi2016aide}, standing for accelerated inexact DANE,  solves an Nesterov accelerated version of DANE's local optimization problem inexactly to some degree of accuracy. Other works include COCOA \cite{ma2015adding} which solves a dual problem of logistic regression locally, DiSCO \cite{zhang2015disco} which considers inexact damped Newton's step solved with distributed preconditioned conjugate gradient method using the inverse of first agent's local Hessian as the preconditioning matrix, and GIANT \cite{wang2017giant} which uses the average of local Newton's step as the global one. More recent work by \citet{dobriban2018distributed} analyzes performance loss in one-shot weighted averaging and iterative averaging for linear regression, and \citet{dobriban2019wonder} also study one-shot weighted averaging for distributed ridge regression in high dimensions. 

While most of the work mentioned above does not address the Hessian inversion bias   directly, another line of work considers this issue. For example, the determinantal averaging method \cite{dereziński2019distributed}  states that one can bypass the inversion bias  as long as an unbiased estimator of Hessian matrix exists and computing determinants of local Hessian matrices is feasible. This method is unstable when data is of large dimension since local Hessian determinant computation is usually infeasible there. It also introduces computational overhead for computing local Hessian determinants.  \citet{zhang2012communication} proposed a bootstrap subsampling method to reduce bias in one-shot averaging that improves the approximated optimizer to some finite suboptimality. However, the improved  approximated optimizer can still be much worse compared to the true optimizer (shown in \cite{shamir2013communication}, Section 2).

Our work addresses bias correction in distributed optimization algorithms with analytical tools from random matrix theory. When data is independently and identically distributed, we introduce a shrinkage formula for estimating the resolvent of the covariance matrix of data which can serve as an inversion bias corrector in the large data regime. This bias correction method is more stable compared to the determinantal averaging method and achieves significantly better accuracy without any notable computational overhead. In the field of asymptotic random matrix theory, prior works studying similar shrinkage formulas exist only for statistical estimation of covariance and precision matrices. \citet{Ledoit2004AWE} is one among the early works studying linear shrinkage estimator of  large covariance matrices. Later work by Bodnar et al. \cite{bodnar2014covariance} studies  linear shrinkage estimator of large covariance matrices with almost sure smallest Frobenius loss. In \citet{bodnar2016directSE}, the authors studied linear shrinkage estimator for the precision matrix. \citet{bodnar2022advance} gives a comprehensive review of recent advancements in shrinkage-based high dimensional inference studies. 
These shrinkage-based methods have already been successfully applied to tests for weights for portfolios \cite{bodnar2019tests} and to robust adaptive beamforming \cite{xiao2018robust}. Our work focuses on  shrinkage-based estimation of the resolvent of covariance matrices, which is different from the shrinkage formula for both covariance and precision matrices, and we study its application to optimization algorithms.

Besides distributed second-order optimization, we find the shrinkage formula we studied is also useful for improving sketching methods. We note \citet{Dereziński2020debiasing} studied debiasing randomized  optimization algorithms with surrogate sketching, where a non-standard carefully chosen sketching matrix is used. Also, in \citet{bartan2022distributed}, the authors exploited random matrix theory for bias elimination in distributed randomized ridge regression. They showed when the covariance is isotropic, an asymptotically unbiased estimator can be obtained by tuning local regularizers. Unlike their approaches, our method works for general  covariance matrices and sketching matrices.




\vspace{-0.2cm}
\subsection{Contribution}
In this work, we propose an asymptotically unbiased shrinkage-based estimator of the resolvent of covariance matrices. Unlike most prior studies in the field of  large-dimensional random matrix theory which consider only the asymptotic settings, we also characterize the non-asymptotic convergence rate. Furthermore, we study the application of this shrinkage formula to distributed second-order optimization algorithms and randomized sketching methods. We carry out real data simulations where a significant convergence speedup is obtained compared to standard baselines.

To our best knowledge, there is no existing work on  shrinkage-based estimation of the resolvent of covariance matrices, and we are also the first to derive a closed-form formula for optimal shrinkage and apply it to optimization.

\vspace{-0.2cm}
\section{Main Theorems}\label{section2}


In this section, we establish our main theoretic results.  
A shrinkage-based asymptotically unbiased estimator of the resolvent of covariance matrices is studied with its convergence rate characterized and its variant in the small regularizer regime analyzed in Section \ref{main}. We then show that the commonly used averaging method has non-zero asymptotic bias and summarize different methods for estimating the resolvent of covariance matrices in Section \ref{avg-bias}.
\label{sec2.1}

We first introduce notations we use for stating our main theorems. We follow the classical  Kolmogorov asymptotics. Consider a sequence of problems $\mathcal{B}_n=(\Sigma, \hat\Sigma, x, d)_n$ with $n,d\rightarrow\infty,$ $d/n\rightarrow y\in[0,1).$ $\Sigma_n\in \mathbb{R}^{d\times d}$ is the true covariance and data distribution satisfies $ \mathbb{E}[x]=0, \mbox{Cov}(x)=\Sigma_n$. Empirical covariance is denoted as   $\hat\Sigma_n=\frac{1}{n}\sum_{i=1}^{n} x_ix_i^T$ where $x_i$'s are i.i.d. samples of $x$. Consider any $\lambda>0,\lambda\in\mathbb{R}$.   $d_\lambda(\Sigma)=\mathrm{tr}(\Sigma(\Sigma+\lambda I)^{-1})$ is the effective dimension of $\Sigma$ and  we define $d_\lambda^n=d_\lambda(\Sigma_n)$. The empirical spectral distribution (e.s.d.) of $\Sigma_n$ is $F_{\Sigma_n}(u) = d^{-1}\sum_{i=1}^{d} \mathbb{I}_{(\lambda_i\leq u)}$ where $\lambda_i$'s are eigenvalues of  $\Sigma_n$. We use $F_{\Sigma_n}(u)\rightarrow F_\Sigma(u)$ to indicate that the e.s.d. converges almost surely to a density $F_\Sigma(u)$. We further define $M=\sup_{|e|=1} \mathbb{E}(e^Tx)^4$, $\nu = \sup_{\|\Omega\|=1}\mathrm{Var}(x^T\Omega x/d)$.

We collect the main assumptions required below,
\begin{assumption}
\begin{align*}
\vspace{-5cm}
&\mbox{A1.}~F_{\Sigma_n}(u)\rightarrow F_\Sigma(u) \mbox{ almost for any }u\geq 0\\
&\mbox{A2.}~M<\infty, \nu\rightarrow 0\\
&\mbox{A3.}~\mbox{eigenvalues of each }\Sigma_n\mbox{ are in the interval }[\sigma_{\min},\sigma_{\max}]\\
&\quad~\mbox{ where }\sigma_{\min}>0\mbox{ and }\sigma_{\max}\mbox{ does not depend on }n
\end{align*}
\label{ass_1}
\vspace{-1cm}
\end{assumption}

\subsection{Asymptotically Unbiased Shrinkage Formula for the Resolvent of Covariance}\label{main}
We propose an asymptotically unbiased estimator of the resolvent of covariance  matrices in the form of a local shrinkage formula under notations defined at the beginning of this section, and we characterize an explicit convergence rate for the isotropic case in Section \ref{con_rate}.
\begin{theorem}
\label{thm2}
Under Assumption \ref{ass_1}, for any $\lambda>0$ and any data matrix with i.i.d. rows, assume additionally $d_\lambda^n<n$ for each $n,d$, then we have
\[\mathbb{E}\left[\left(\frac{1}{1- \frac{d_\lambda^n}{n}}\hat \Sigma_n+\lambda I\right)^{-1}\right]=(\Sigma_n+\lambda I)^{-1}+\Omega_0\]
where $\|\Omega_0\|\rightarrow 0$ as $n,d\rightarrow\infty$, $d/n\rightarrow y\in[0,1)$. \footnote{Here $\Omega_0$ depends on $n$ and we write $\Omega_0=\Omega_0(n)$ for notational simplicity. The same convention is used in later sections as well as in the appendix.}
\end{theorem}
\begin{proof}
See Appendix \ref{thm2proof}.
\end{proof}
\label{constraint}
\vspace{-0.2cm}
Note that this result universally holds for a large class of random data matrices with i.i.d. rows. In Assumption \ref{ass_1}, we only require $M$  bounded and $\nu$ vanishing as $n,d$ tend to infinity. To see such constraints are quite mild, note requiring $M=\sup_{|e|=1} \mathbb{E} (e^Tx)^4$ staying bounded is essentially bounding the dependence of $x$'s components. When $x\sim \mathcal{N}(0,\Sigma)$, $M=3\|\Sigma\|^2.$ For $\nu=\sup_{\|\Omega\|=1}\mbox{Var}(x^T\Omega x/d)$, 
when components of $x$ are independent,  $\nu\leq M/d$, and when $x\sim \mathcal{N}(0,\Sigma),\nu\leq 2\|\Sigma\|^2/d$. For more discussions on these constraints, see \citet{multi_stats}, Chapter 3.1.

\subsubsection{Isotropic Convergence Rate} \label{con_rate}

To evaluate the convergence rate in Theorem \ref{thm2}, note if there is no constraint on how fast $d/n$ converges to $y$, then the convergence rate  can be arbitrarily bad. To characterize the convergence rate, here we require $d/n=y$ always holds. Then from the expression for $\Omega_0$  (see  inequality (\ref{convergence_bound2}) in Appendix \ref{thm2proof}), we need to find the convergence rate of the Stieltjes transform of the spectral distribution of gram matrices. Such task is in general hard if no constraint on the covariance matrices is imposed. Prior work analyzing such bounds exists for  covariance matrices with correlated Curie-Weiss entries \cite{fleermann2019high}, sparse covariance matrices \cite{erdHos2020random}, covariance matrices with independent but not necessarily identically distributed entries \cite{bai2010spectral}.

Here we focus on the isotropic covariance case where we can exploit isotropic local Marchenko-Pastur law to derive an explicit convergence rate for Theorem  \ref{thm2}.

\begin{theorem}
\label{thm2.4}
When $\Sigma_n=I$, under Assumption \ref{ass_1}, 
\[\|\Omega_0\|\in\mathcal{O}\left(\frac{1}{\sqrt{n}}+\sqrt{\nu}\right)\]
\end{theorem}
\begin{proof}
See Appendix \ref{thm2.4proof}.
\end{proof}
\vspace{-0.2cm}
\textbf{$\mbox{Remark.}$} For $x\sim\mathcal{N}(0,I)$, Assumption \ref{ass_1} holds and
$\|\Omega_0\|\in\mathcal{O}\left(\frac{1}{\sqrt{n}}\right)$

\subsubsection{Small Regularizer Regime}\label{small-regularizer}

In this subsection, we study the behavior of Theorem \ref{thm2} when the regularizer diminishes to zero, which results in a simpler local shrinkage coefficient requiring no estimation of any effective dimension.
\begin{theorem}
\label{thm2.5}
Under Assumption \ref{ass_1}, assume $d<n$ always holds. If  $~\Sigma_n$ is invertible and eigenvalues of $\hat\Sigma_n$ are bounded away from zero, then for $\epsilon>0$, 
\[ \mathbb{E}\left[\left(\frac{1}{1-\frac{d}{n}}\hat\Sigma_n+\epsilon I\right)^{-1}\right]=\Sigma_n^{-1} + \Omega_1\]
with $\|\Omega_1\|\rightarrow 0$ as $n,d\rightarrow\infty$ and $\epsilon\rightarrow 0$.  
\end{theorem}
\begin{proof}
See Appendix \ref{thm2.5proof}.
\end{proof}

\textbf{$\mbox{Remark.}$} With a small regularizer, the estimation of the resolvent of covariance matrices is close to the estimation of precision matrices. Theorem \ref{thm2.5} parallels results on shrinkage estimators for the precision matrix as investigated in Theorem 3.2 of \citet{bodnar2015direct}. However, their work only considers the asymptotic setting.

\subsection{Asymptotic Bias of the Naive Averaging Method}\label{avg-bias}
After introducing our asymptotic unbiased estimator for the resolvent of covariance matrices, we now analyze the asymptotic bias for simple averaging method without shrinkage. Theorem \ref{thm2.6} below states that this bias is non-zero for small $\lambda$. See Appendix \ref{thm2.6proof} for analysis for general positive $\lambda$.

\begin{theorem}\label{thm2.6}
Under Assumption \ref{ass_1}, and if $\lambda\in o(1),$
\begin{equation*}
\begin{aligned}
&\lim_{\substack{n,d\rightarrow\infty\\ d/n\rightarrow y} }\left\|\mathbb{E}\left(\hat\Sigma_n+\lambda I\right)^{-1}-\left(\Sigma_n+\lambda I\right)^{-1}\right\|\geq \frac{y\sigma_{\min} }{\sigma_{\max}^2}
\end{aligned}
\end{equation*}
\begin{proof}
See Appendix \ref{thm2.6proof}.
\end{proof}
\end{theorem}
This result demonstrates that the asymptotic bias for simple averaging method without shrinkage can be substantial, and our proposed shrinkage formula serves as an effective improvement to solve this issue.

As a summary, for random data matrices of growing size satisfying Assumption \ref{ass_1},  Table \ref{table1} summarizes the bias of various methods for estimating  the resolvent of covariance matrices under both non-asymptotic and asymptotic settings. 


\begin{table*}[h]
\vskip 0.15in
\begin{center}
\begin{small}
\begin{sc}
\begin{tabular}{lccccr}
\toprule
Method  & sample size & non-asymptotic bias& asymptotic bias & complexity \\
\midrule
averaging & $\forall k$& see  (\ref{avg_lwb_finite}) & $\geq y\sigma_{\min} /\sigma_{\max}^2$ & 0\\
determinantal averaging& $\forall k$ & 0 & 0&$\mathcal{O}(d^{3})$\\
optimal shrinkage &   $k\geq d_\lambda$& $\leq \mathcal{O}\left(\frac{1}{\sqrt{n}}+\sqrt{\nu}\right)$ & 0 & 0\\
\bottomrule
\end{tabular}
\caption{The bias in the estimation of covariance resolvent with different methods under asymptotic/non-asymptotic settings. For any real positive $\lambda,$ $d_\lambda$  denotes the effective dimension of the covariance matrix. The asymptotic bias for the averaging method is given for $\lambda\in o(1)$, see Theorem \ref{thm2.6} for details. The non-asymptotic bias for optimal shrinkage method is given for isotropic covariance, see Section \ref{con_rate} for details. The complexity column refers to agents' additional computational overhead compared to the averaging method. When $\lambda\in o(1),$ shrinkage coefficient requires no effective dimension estimation and thus no computational overhead, see Section \ref{small-regularizer} for details. 
When $\lambda\not \in o(1),$ we assume the effective dimension of the covariance is either known or has been estimated in advance and is distributed to all agents.}\label{table1}
\end{sc}
\end{small}
\end{center}
\vskip -0.1in
\end{table*}

\section{Application to Distributed Second-Order Optimization Algorithms}\label{section3}

Now we are ready to utilize theorems studied in Section \ref{section2} in distributed second-order optimization algorithms. We outline the algorithms for distributed Newton's method with optimal shrinkage and its inexact version  solved with distributed preconditioned conjugate gradient method with optimal shrinkage below. The convergence proofs for quadratic loss and general convex smooth loss are provided in Section \ref{quad}, Section \ref{cvx}, and Appendix \ref{conv-append}. Finally, we analyze communication and computation complexity for the proposed algorithms in Section \ref{complexity}.

Let $n$ denote the number of data samples and $m$ denote the number of agents.  Consider the $\ell_2$ regularized loss function $f(x)=\frac{1}{m}\sum_{i=1}^m f_i(x)+\frac{\lambda}{2}\|x\|_2^2$ where $f_i(x)=\frac{1}{k}\sum_{j=1}^k \ell_{ij}(x)$ denotes loss function corresponding to agent $i$ and $k$ is the number of samples available to each agent. Here, we consider the case where the data is evenly split to all agents and thus $k=n/m$. Denote 
$\nabla^2 f$ as the Hessian and $\nabla f$ as the gradient of function $f$. In order to apply our results in Section \ref{section2}, we need the effective dimension of the population Hessian $\mathbb{E} \frac{1}{m}\sum_{i=1}^m \nabla^2 f_i(x)$. We use the empirical effective dimension $d_{\lambda, i}=\mbox{tr}\left(\nabla^2 f_i(\nabla^2 f_i+\lambda I)^{-1}\right)$ available at each agent as an approximation. Algorithm \ref{alg:newton} gives a description of the proposed method.

\begin{algorithm}[h]
   \caption{Distributed Newton's method with optimal shrinkage}
   \label{alg:newton}
\begin{algorithmic}
   \STATE {\bfseries Initialize:} starting point $x^{(0)}, t=1$
   \REPEAT
   \STATE Gather local gradients $\nabla f_i(x^{(t-1)})$ from each agent $i$
   \STATE Compute global gradient $\nabla f(x^{(t-1)})=\frac{1}{m}\sum_{i=1}^m \nabla f_i(x^{(t-1)})+\lambda x^{(t-1)}$ and broadcast to all agents
   \FOR{$i=1,2,...~$}
   \STATE agent $i$ computes $x_i^{(t)}=$\\$\left(\frac{1}{1-\frac{m d_{\lambda, i}}{n}}\nabla^2 f_i\left(x^{(t-1)}\right)+\lambda I\right)^{-1} \nabla f(x^{(t-1)})$
   \ENDFOR
 \STATE Compute approximate Newton step $\Delta x^{(t)}=$\\$\frac{1}{m}\sum_{i=1}^m x_i^{(t)}$
\STATE (Optional) Choose step size $\eta$ by line search
 \STATE Update $x^{(t)}=x^{(t-1)}-\eta \Delta x^{(t)}, ~t=t+1$
   \UNTIL{convergence criterion or maximum iterates reached}
\end{algorithmic}
\end{algorithm}


We then provide a preconditioned conjugate gradient method that exploits optimal shrinkage for solving a single Newton step in Newton's method. Let $v$ denote the current point at which the next Newton step needs to be performed. Algorithm \ref{alg:pcg} gives the distributed preconditioned conjugate gradient method for inexact Newton's method.
\begin{algorithm}[h]
   \caption{Distributed preconditioned conjugate gradient  with optimal shrinkage}
   \label{alg:pcg}
\begin{algorithmic}
   \STATE Compute $b=\nabla f(v)=\frac{1}{m}\sum_{i=1}^m \nabla f_i(v)+\lambda v$ by gathering local gradients $\nabla f_i(v)$ from each agent
   \STATE {\bfseries Initialize:} $x=0,r=b$ 
   \FOR{$i=1$ {\bfseries to} $m$}
   \STATE agent $i$ computes $z_i=$\\$\left(\frac{1}{1-\frac{md_{\lambda, i}}{n}}\nabla^2 f_i(v)+\lambda I\right)^{-1} r$
   \ENDFOR
   \STATE $z=\frac{1}{m}\sum_{i=1}^m z_i, p=z,\rho_1=r^T z$
   \FOR{$t=1$ {\bfseries to} $t_{\mbox{max}}$}
   \STATE quit if \mbox{stopping criterion achieved}
   \STATE $\omega=\sum_{i=1}^m \nabla^2 f_i(v)p+\lambda p$ by gathering $\nabla^2 f_i(v)p$
   \STATE $\alpha=\frac{\rho_t}{\omega^T p}$
   \STATE $x=x+\alpha p$
   \STATE $r=r-\alpha\omega$
   \FOR{$i=1$ {\bfseries to} $m$}
   \STATE agent $i$ computes $z_i$\\$=\left(\frac{1}{1-\frac{md_{\lambda, i}}{n}}\nabla^2 f_i(v)+\lambda I\right)^{-1} r$
   \ENDFOR
   \STATE  $z=\frac{1}{m}\sum_{i=1}^m z_i$
   \STATE $\rho_{t+1}=z^T r$
   \STATE $p=z+\frac{\rho_{t+1}}{\rho_t}p$
   \ENDFOR

\end{algorithmic}
\end{algorithm}

\subsection{Convergence Analysis for Regularized Quadratic Loss}\label{quad}

Given data matrix $A\in \mathbb{R}^{n\times d}$ with data i.i.d. with mean zero and covariance $\Sigma$, and label $b\in \mathbb{R}^n$, let $A^{(i)}\in\mathbb{R}^{(n/m)\times d}$ denote agent $i$'s local data. Consider the regularized quadratic loss function $f(x)=\frac{1}{2n}\|Ax-b\|_2^2+\frac{\lambda}{2}\|x\|_2^2$
with gradient $g(x):=\nabla f(x)$ and Hessian $H:=\nabla^2 f(x) =\frac{1}{n}A^TA+\lambda I.$
Denote
\vspace{-0.3cm}
\[\tilde H=\left(\frac{1}{m}\sum_{i=1}^m \left(\frac{1}{1-\frac{md_{\lambda}}{n}}\frac{m}{n}A^{(i)^T}A^{(i)}+\lambda I\right)^{-1}\right)^{-1}\]
with the true effective dimension of the covariance, $\Sigma:=\mathbb{E} \frac{1}{n} A^T A$
\[d_{\lambda}=\mbox{tr}\left(\Sigma \left(\Sigma+\lambda I\right)^{-1}\right)\]
\begin{theorem}\label{thm3.1}
 (Convergence of Newton's method with Shrinkage) Denote $\Delta_{t+1}=\omega_{t+1}-\omega^\star$ with $\omega^\star=\mbox{argmin}~f(x)$ and $\omega_{t+1}=\omega_t-\tilde H^{-1}g(\omega_t )$. Then,
\[\|\Delta_{t+1}\|\leq \beta\|\Delta_t\|\]
where $\beta=\frac{\sqrt{2}\alpha}{\sqrt{1-\alpha^2}}\sqrt{\frac{\sigma_{\max}+\lambda+\alpha_0}{\sigma_{\min}+\lambda-\alpha_0}}$ with $\alpha_0=\|\Sigma-\frac{1}{n}A^TA\|, \alpha_1=\|\tilde H^{-1}-\mathbb{E}[\tilde H^{-1}]\|,$ and $\alpha=(\sigma_{\max}+\lambda+\alpha_0)\left(\frac{1}{\lambda^2}\alpha_0+\alpha_1+\|\Omega_0\|\right)$. $\sigma_{\min}$ and $\sigma_{\max}$ denote the smallest and largest eigenvalues of $\Sigma$ correspondingly.
\end{theorem}
\vspace{-0.2cm}
\begin{proof}
See Appendix \ref{thm3.1proof}.
\end{proof}
The most important aspect of the above result is that the contraction rate $\beta$ \emph{vanishes to zero} as the number of workers increase and data dimensions grow asymptotically as we formalize next.\\
\textbf{$\mbox{Remark.}$} 
Consider a sequence of data matrices $\{A\}_n$ with $n/m\rightarrow\infty, d\rightarrow\infty, md/n\rightarrow y\in[0,1)$ and each $A^{(i)}$ satisfies Assumption \ref{ass_1}, then  $\|\Omega_0\|\rightarrow 0$ by 
 Theorem \ref{thm2}. When data is Gaussian or sub-Gaussian, $\alpha_0\rightarrow 0$ almost surely given $d/n\rightarrow 0$.  By standard matrix concentration bounds, $\alpha_1\leq \epsilon$ with probability $\geq 1-2d\exp\left(-\epsilon^2/\left((4\epsilon)/(3m\lambda)+2/(m\lambda^2)\right)\right)$. Thus  $\beta\rightarrow 0$  almost surely when each $A^{(i)}$
 satisfies Assumption \ref{ass_1}, data is Gaussian or subGaussian, and $m,d,n\rightarrow\infty,d,m = o(n),\log d = o(m)$, $md/n\rightarrow y\in[0,1)$.

\begin{theorem}
 (Convergence of inexact Newton's method with Shrinkage)  Let $\alpha,\alpha_0, \Delta_{t+1},\sigma_{\min},\sigma_{\max}$ defined as in Theorem \ref{thm3.1}. Then,
\[\|\Delta_{t+1}\|\leq \frac{\sqrt{2}\alpha'}{\sqrt{1-\alpha'^2}}\sqrt{\frac{\sigma_{\max}+\lambda+\alpha_0}{\sigma_{\min}+\lambda-\alpha_0}}\|\Delta_t\|\]
where $\alpha'=\sqrt{4\left(\frac{1-\sqrt{1-\frac{\alpha}{1-\alpha}}}{1+\sqrt{1-\frac{\alpha}{1-\alpha}}}\right)^{s_t}}$ and $s_t$ denotes the number of iterations in preconditioned conjugate gradient method.
\end{theorem}
\begin{proof}
The derivation follows by Lemma 14 in \citet{dereziński2019distributed} and Lemma \ref{lemmab.1}.
\end{proof}

\subsection{Convergence Analysis for Regularized General Convex  Smooth 
 Loss}\label{cvx}

Given data matrix $A\in \mathbb{R}^{n\times d}$ with data i.i.d. with mean zero. Let $A^{(i)}\in\mathbb{R}^{(n/m)\times d}$ denote agent $i$'s local data and $A^{(i)}_j$ denote the $j$th piece of data held by agent $i$. Consider general convex smooth loss function $f$ of the following form,
\[f(x)=\frac{1}{n}\sum_{i=1}^n f_i(x^TA_i)+\frac{\lambda}{2}\|x\|^2\]
with gradient $g(x):=\nabla f(x)$
and hessian $H(x):=\nabla^2 f(x)=\frac{1}{n}\sum_{i=1}^n f_i''(x^TA_i)A_iA_i^T+\lambda I$. Assume $f$ is twice differentiable and its hessian is $L$-Lipschitz.  Denote
\[\tilde H(x)^{-1}=\frac{1}{m}\sum_{i=1}^m \left(\gamma\sum_{j=1}^{n/m} f_i''\left(x^TA_j^{(i)}\right)A_{j}^{(i)}A_{j}^{(i)^T}+\lambda I\right)^{-1}\]
where $\gamma=m/n\left(1-\frac{m d_\lambda(x)}{n}\right)$ with
$d_\lambda(x)=\mbox{tr}\left(\Sigma(x)\left(\Sigma(x)+\lambda I\right)^{-1}\right)$
and $\Sigma(x)$ defined in Theorem \ref{thm3.3}.
\begin{theorem}\label{thm3.3}
 (Convergence of Newton's method with Shrinkage)  Assume $\omega_t$ independent of all $A_j^{(i)}$'s, denote $\Sigma(\omega_t)$ $= \mbox{Cov}\left(f_i''\left(\omega_t^TA_j^{(i)}\right)^{\frac{1}{2}}A_j^{(i)}\right)$.  Denote $\Delta_{t+1}=\omega_{t+1}-\omega^\star$ with $\omega^\star=\mbox{argmin}~f(x)$ and $\omega_{t+1}=\omega_t-\tilde H(\omega_t)^{-1}g(\omega_t )$, 
\begin{equation*}
\begin{aligned}
\|\Delta_{t+1}\|\leq \max\left\{ \frac{2L}{\sigma_{\min}+\lambda-\alpha_0}\|\Delta_t\|^2, \right.\\ \left. \frac{\sqrt{2}\alpha}{\sqrt{1-\alpha^2}}\sqrt{\frac{\sigma_{\max}+\lambda+\alpha_0}{\sigma_{\min}+\lambda-\alpha_0}}\|\Delta_t\|\right\}
\end{aligned}
\end{equation*}
where $\alpha=(\sigma_{\max}+\lambda+\alpha_0)\left(\frac{1}{\lambda^2}\alpha_0+\alpha_1+\|\Omega_0\|\right)$ with $\alpha_0=\left\|\Sigma(\omega_t)-\frac{1}{n}\sum_{i=1}^n f_i''\left(\omega_t^TA_i\right)A_iA_i^T\right\|$ and $\alpha_1=\left\|\tilde H(\omega_t)^{-1}-\mathbb{E}\left[\tilde H(\omega_t)^{-1}\right]\right\|$. $\sigma_{\min}$ and $\sigma_{\max}$ denote the smallest and largest eigenvalues of $\Sigma(\omega_t)$ correspondingly.  
\end{theorem}
\begin{proof}
See Appendix \ref{thm3.3_proof}.
\end{proof}


\vspace{-0.4cm}
\subsection{Communication and Computation Complexity Analysis}\label{complexity}
\vspace{-0.1cm}
In this section, we analyze  the communication and computation complexity of distributed Newton's method with optimal shrinkage. The analysis for inexact Newton's method is similar and omitted.

On the communication side, in each Newton iteration\footnote{Assuming fixed step size instead of line search for simplicity.}, four rounds of communication between server and agents are required: the server broadcasts the current iterate, collects the local gradients, and then computes and broadcasts the global gradient to the agents and collects local approximate descent directions. Each communication involves $\mathcal{O}(d)$ words. In ridge regression, to achieve some fixed accuracy $\|\Delta_t\|\leq \epsilon$, the number of iterations is bounded by $\mathcal{O}\left(\frac{\log (\sqrt{\kappa}/\epsilon)}{\log (\sqrt{1-\alpha^2}/\alpha)}\right)$ where $\kappa$ is the condition number of the Hessian matrix and $\alpha$ is asymptotically vanishing  (see Theorem \ref{thm3.1} for definition). Therefore, the number of iterations decreases as the number of workers increases and vanishes asymptotically. This should be compared with GIANT's bound $\mathcal{O}\left(\frac{\log(d\kappa/\epsilon)}{\log(n/\mu d m)}\right)$ on the number  of Newton iterations to achieve the same degree of accuracy, which does not vanish asymptotically due to the Hessian inversion bias. Note that our bound also gets rid of the dependency on the matrix coherence number $\mu$, and we do not impose assumption such as $n>\mathcal{O}(\mu d m)$.  
Compared to  DiSCO's bound $\tilde {\mathcal{O}}\left(\frac{d\kappa^{1/2}m^{3/4}}{n^{3/4}}+\frac{\kappa^{1/2}m^{1/4}}{n^{1/4}}\log\frac{1}{\epsilon}\right)$ which is also non-vanishing asymptotically due to the inversion bias, our bound only has log dependency on the square root of $\kappa$ instead of polynomial, which is an improvement\footnote{The iteration bounds for GIANT and DiSCO are taken from \citet{wang2017giant}, Section 1.1. Note the number of iterations for  DiSCO is required to achieve $\epsilon$-accuracy in terms of function value evaluations.}.

The per-iteration computation complexity for each agent involves operations required for forming the local gradient, solving a linear equation involving the local Hessian matrix, and computing the effective dimension of the local Hessian matrix.  The only overhead compared to GIANT is the computation  of the effective dimension of the local Hessian matrix, which can be done in $\mathcal{O}(kd\min\{k,d\})$ by singular value decomposition or can be estimated much faster by trace estimation methods such as Hutch++ \cite{meyer2020hutch}. Note if the effective dimension of the global Hessian matrix is known beforehand, then it can be used in place of the effective dimension of local Hessian matrices and there is no additional computational overhead for each agent. Another option is to use $md/n$ in place of $md_{\lambda i}/n$ in  Algorithm \ref{alg:newton}, which should work well for small regularizer $\lambda$ by Theorem \ref{thm2.5}. There is no significant computational complexity on the server side since only simple averaging and arithmetic operations are required.




\vspace{-0.3cm}
\section{Application to Randomized Second-Order Optimization Algorithms}\label{section4}
\vspace{-0.1cm}
In this section, we study the application of our main theoretic result in Section \ref{section2} to randomized second-order optimization algorithms. We mainly focus on the Iterative Hessian Sketch (IHS) method and discuss how our shrinkage formula can be used for bias correction. Our results can also be used for randomized preconditioned conjugate gradient method and the more general Newton Sketch \cite{pilanci2017newton}.

We first give an introduction on randomized sketching which is used in the Iterative Hessian Sketch method \cite{pilanci2016iterative, ozaslan2019iterative}. Randomized sketching is an important tool in randomized linear algebra for dealing large-scale data problems. Given a data matrix $A\in \mathbb{R}^{n\times d},$ consider a randomized sketching matrix $S\in\mathbb{R}^{m\times n},$ $SA$ is referred to as a sketch of $A$ and is of size $m$ by $d$. It is usually the case $m<n$, thus storage is reduced when $SA$ is stored instead of the original data matrix $A$, and with carefully chosen $S$, some properties of $A$ can be preserved by considering only $SA$. When $SA$ is composed of randomly selected rows of $A$, it  is referred to as row subsampling and when $SA$ is composed of random linear combination of rescaled rows of $A$, it is referred  to as a random projection. A commonly used random projection is Gaussian projection, where $S$ contains i.i.d. Gaussian entries $\mathcal{N}\left(0,\frac{1}{m}\right)$. 

Consider regularized quadratic loss function defined in Section \ref{quad}. With data matrix $A\in \mathbb{R}^{n\times  d}$ and labels $b\in\mathbb{R}^n$, the Hessian matrix is $H=\frac{1}{n}A^TA+\lambda I$. Note we are not requiring the data to be i.i.d. with mean zero here as required in Section \ref{quad}. Let $S\in \mathbb{R}^{m\times n}$ be Gaussian projection matrix. Iterative Hessian Sketch method is a preconditioned first-order method  that replaces the Newton direction $H^{-1}g_t$ at the $t$-th Newton's step in Newton's method by $H_S^{-1}g_t$ where $H_S=\frac{1}{n}A^TS^TSA+\lambda I$ and $g_t$ denotes the gradient at the $t$-th Newton's step.

Since $H_S$ can be expressed as $\frac{1}{m}\sum_{i=1}^m x_ix_i^T+\lambda I$ where $x_i\sim\mathcal{N}\left(0,\frac{1}{n}A^TA\right)$, debiasing IHS reduces to minimizing $\|H_S^{-1}-H^{-1}\|$ and is exactly the problem of debiasing estimation of a covariance resolvent. We adapt Theorem \ref{thm2} as below and obtain an asymptotically unbiased estimation of $H^{-1}$ as a shrinkage formula involving $H_S$.

Consider a sequence of problems $\mathcal{B}_n=(A,S,d,m)_n$ with data matrix $A\in\mathbb{R}^{n\times d},S\in\mathbb{R}^{m\times n}$ being Gaussian projection matrix, $m,d\rightarrow\infty$ and $d/m\rightarrow y\in[0,1)$. Assume the empirical spectral distribution $F_{(\frac{1}{n}A^TA)}(u)\rightarrow F_\Sigma(u)$ almost surely for almost every $u\geq 0$, $\left\|n^{-1}A^TA\right\|^2\leq \infty$, and eigenvalues of each $n^{-1}A^TA$ are located on a segment  $[\sigma_{\min},\sigma_{\max}]$ where $\sigma_{\min}>0$ and $\sigma_{\max}$ does not depend on $n$. Denote $d_\lambda^n=\mbox{tr}\left(\frac{1}{n}A^TA\left(\frac{1}{n}A^TA+\lambda I\right)^{-1}\right)$.
\begin{theorem}\label{thm4.2}
Assume additionally $d_\lambda^n<m$ always holds. For any $\lambda>0$,
\[\lim_{m,d\rightarrow\infty}\left\|\mathbb{E}\left[\left(\frac{1}{1- \frac{d_\lambda^n}{m}}(H_S-\lambda I)+\lambda I\right)^{-1}\right]-H^{-1}\right\|=0\]
\end{theorem}
\begin{figure*}[ht!]
\begin{subfigure}{0.23\linewidth}
\includegraphics[width=\linewidth]{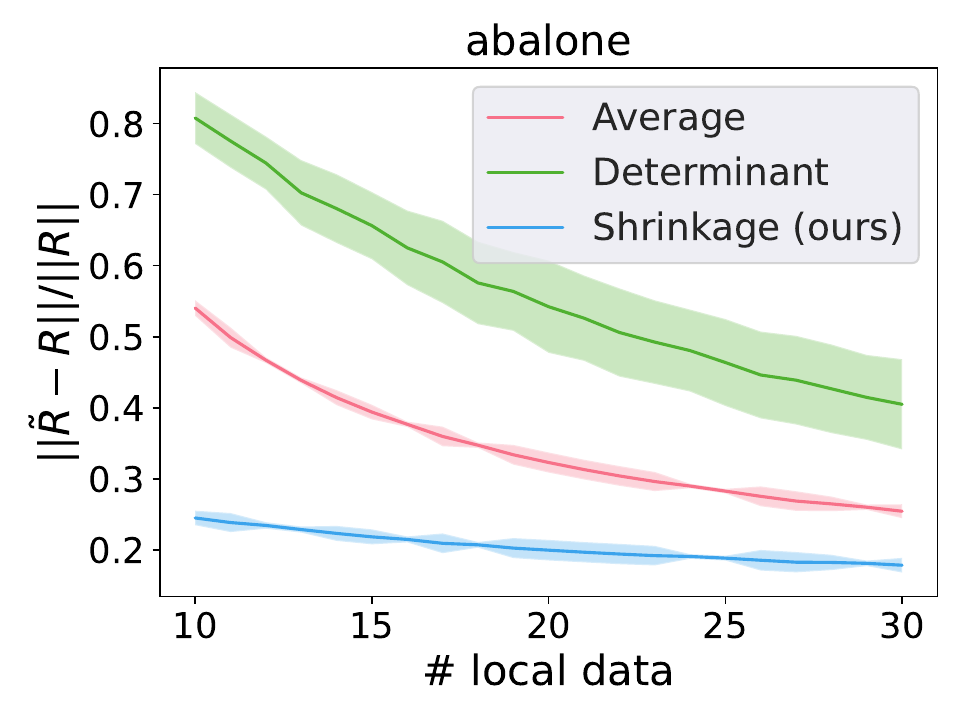}
\end{subfigure}
\hfill
\begin{subfigure}{0.23\linewidth}
\includegraphics[width=\linewidth]{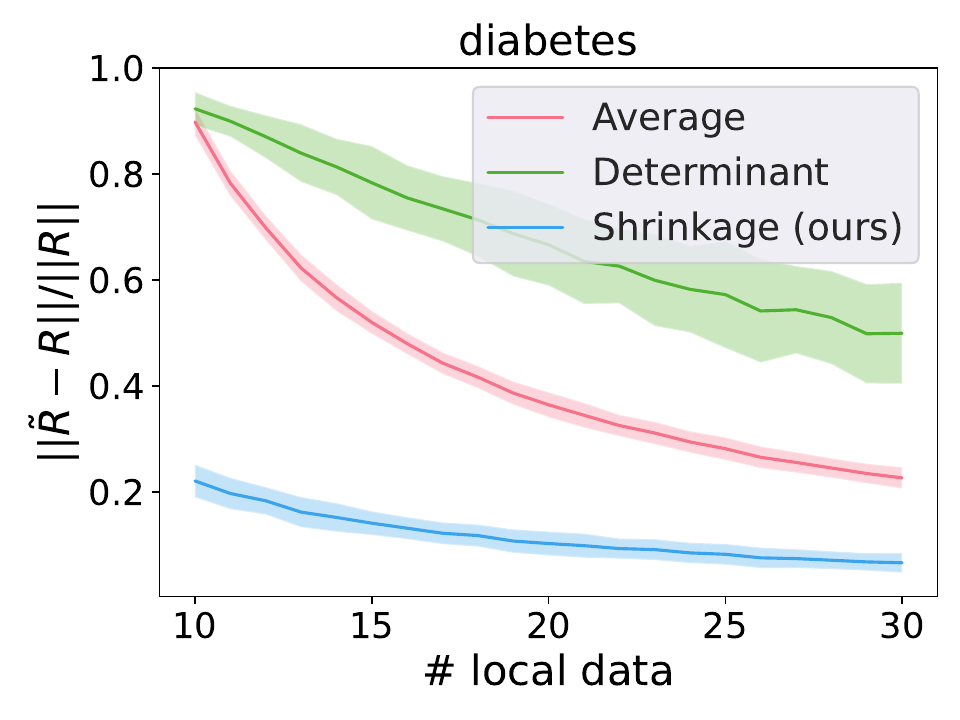}
\end{subfigure}
\hfill
\begin{subfigure}{0.23\linewidth}
\includegraphics[width=\linewidth]{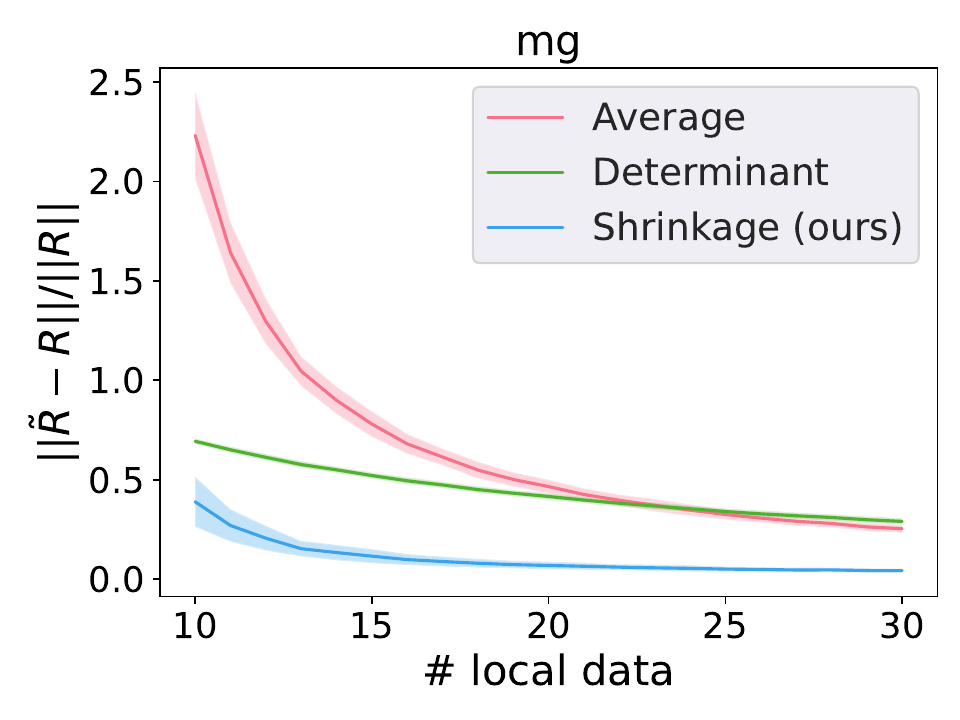}
\end{subfigure}
\hfill
\begin{subfigure}{0.23\linewidth}
\includegraphics[width=\linewidth]{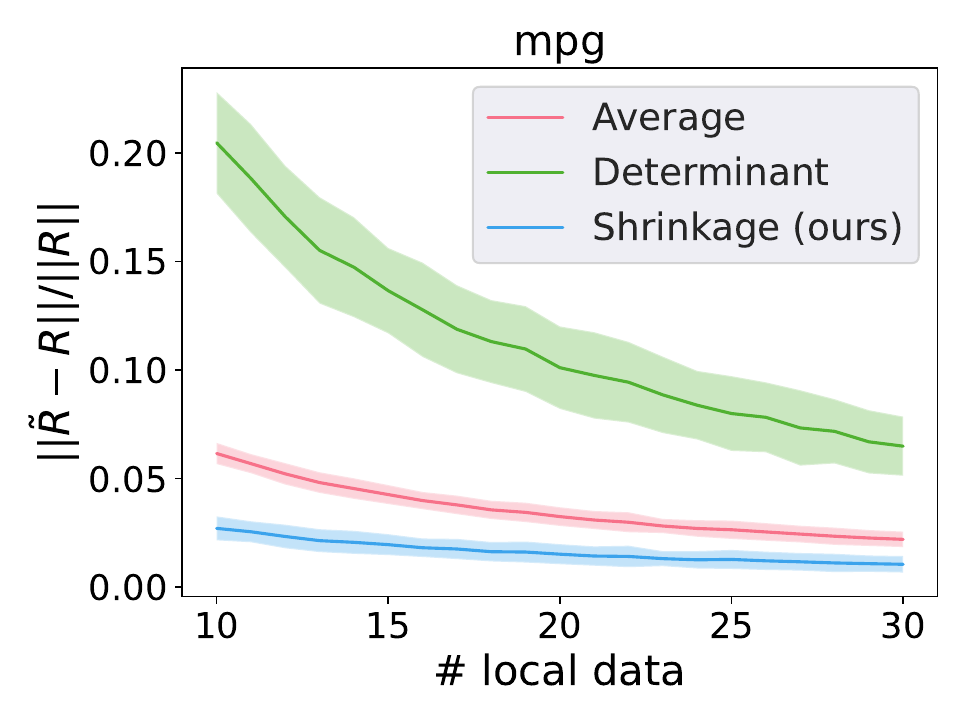}
\end{subfigure}
\caption{Experiments with real data on covariance resolvent estimation. The dataset is split evenly to each agent. We let $\lambda=0.001$ and $\Sigma=\frac{1}{n}A^TA$. The relative matrix spectral norm difference between true covariance resolvent $R$ and estimated covariance resolvent $\tilde R$ is plotted, see Section \ref{section5.1} for details.}\label{sec5.1.2plot}
\end{figure*}
Theorem \ref{thm4.2} suggests to use $\left(\frac{1}{1- \frac{d_\lambda^n}{m}}(H_S-\lambda I)+\lambda I\right)^{-1}$ in place of $H_S^{-1}$ when $d_\lambda^n$ is available. For practicality, when only sketched data $SA$ is available, $\tilde{d}_\lambda^n=\mbox{tr}\left(\frac{1}{n}A^TS^TSA\left(\frac{1}{n}A^TS^TSA+\lambda I\right)^{-1}\right)$ can be used as an approximation for $d_\lambda^n$. According to our real data simulation results in Section \ref{section5.3}, using $\tilde{d}_\lambda^n$ in place of $d_\lambda^n$ significantly improves the plain IHS method in most cases.

In \citet{lacotte2021adaptive, lacotte2019faster}, the authors study preconditioned iterative methods with $H_S^{-1}$ as the preconditioning matrix, which can be replaced with the shrinked version $\left(\left(1- \frac{d_\lambda^n}{m}\right)^{-1}(H_S-\lambda I)+\lambda I\right)^{-1}$.

Newton Sketch method generalizes IHS to regularized general convex smooth losses defined in Section \ref{cvx}. At $t$th Newton's step, the Hessian matrix can be expressed as $H_t=\frac{1}{n}A_t^TA_t+\lambda I$ with an appropriate choice of  the matrix $A_t$, and Newton Sketch method proposes to use $\left(\frac{1}{n}A_t^TS^TSA_t+\lambda I\right)^{-1}g_t$ as approximate Newton descent direction with $g_t$ denoting the gradient at the $t$-th Newton step. By replacing $A$ with $A_t$ in Theorem \ref{thm4.2}, the asymptotically unbiased estimation for $H_t^{-1}$ can be derived.

Table \ref{table2} summarizes the bias  to estimate the Hessian inverse for regularized quadratic loss with classic IHS paradigm and IHS with optimal shrinkage described above  under both non-asymptotic setting and asymptotic setting.

\begin{table*}[ht]
\vskip 0.15in
\begin{center}
\begin{small}
\begin{sc}
\begin{tabular}{lcccr}
\toprule
Method & sketch method & sketch size & non-asymptotic bias&asymptotic bias\\
\midrule
IHS  & gaussian projection& $\forall k$&  see  (\ref{avg_lwb_finite}) & $\geq y\sigma_{\min}/\sigma_{\max}^2$\\
IHS with optimal shrinkage &  gaussian projection& $k\geq d_\lambda$& $\leq \mathcal{O}\left(\frac{1}{\sqrt{m}}\right)$ & 0 \\
\bottomrule
\end{tabular}
\caption{Bias of estimation of Hessian inverse in IHS for regularized quadratic loss under asymptotic/non-asymptotic setting. For any real positive $\lambda$, $d_\lambda$ 
 denoting the effective dimension of $A^TA/n$. The asymptotic bias for the IHS method is for $\lambda\in o(1)$, see Theorem \ref{thm2.6} for details. The non-asymptotic bias for IHS with optimal shrinkage method is for isotropic covariance, see Section \ref{con_rate} for details. }\label{table2}
\end{sc}
\end{small}
\end{center}
\vskip -0.1in
\end{table*}

\vspace{-0.3cm}
\section{Numerical Simulation}\label{simu}
\vspace{-0.2cm}
We now present synthetic  and real data simulation results
. All real datasets used in this section are public and available at \url{https://www.csie.ntu.edu.tw/~cjlin/libsvmtools/datasets/}. For normalized real data plots, we experiment with ten random permutations. For sketched real data plots, we experiment with ten random sketches.  Median is plotted with 0.2/0.8 quantile shaded. We interpolate over $x-$axis whenever $x$ ticks vary for different trials. We run all experiments on google cloud n1-standard-8 machine. One-hot embedding is used to transfer classification labels to regression labels when classification datasets are used for regression tasks. Code for experiments is included in the submission. According to our  simulation results, optimal shrinkage helps speeding up both second-order optimization algorithms and sketching based algorithms significantly. 
\vspace{-0.2cm}
\subsection{Estimation of the Effective Dimension}
\vspace{-0.2cm}
Our optimal shrinkage method requires the knowledge of the effective dimension $d_\lambda$ of the true covariance matrix. In Algorithms \ref{alg:newton} and \ref{alg:pcg}, we employ the empirical effective dimension available at each worker as an approximation to the true effective dimension. Although this is a heuristic to approximate effective dimension in Theorem \ref{thm2}, our numerical results show that this approach works extremely well. Alternatively, the effective dimension can be estimated from a sketch of the data. We illustrate the effectiveness of this approach in the simulation for Iterative Hessian Sketch method in Section \ref{section5.3}.
\vspace{-0.3cm}
\subsection{Experiments on Covariance Resolvent Estimation}\label{section5.1} 
\vspace{-0.01cm}
In order to show that the shrinkage-based covariance resolvent estimation method studied in Section \ref{section2} helps improve the accuracy for covariance resolvent estimation  compared to classic averaging method and determinantal averaging methods discussed in the introduction, we include simulation results for covariance resolvent esimation. Specifically, we let $\Sigma$ denote the covariance matrix, we plot the relative matrix spectral norm difference $\|\tilde R-R\|_2/\|R\|_2$ where $R=(\Sigma+\lambda I)^{-1}$ is the resolvent of the true covariance matrix. The expressions for computing the estimated covariance resolvent $\tilde R$ with different methods are given in Appendix \ref{5.1formula}.
Since we require data rows to be i.i.d. with mean zero for the optimal shrinkage formula to hold, we standardize datasets by removing the mean and scaling to unit variance.  

Figure \ref{sec5.1.2plot} shows the relative matrix spectral norm difference between $R$ and $\tilde R$ plotted over different number of local data, which suggests that our shrinkage method provides a more accurate estimate of the resolvent of covariance matrices than the naive averaging method and determinantal averaging over all the datasets we have tested on. The improvement is more pronounced when the local data size is small. Since the determinantal averaging method is also unbiased, this result also suggests that our shrinkage method does not need a large number of distributed agents to achieve an accurate estimation compared to the determinantal averaging. For simulations on additional datasets, see Appendix \ref{sec5.1.2append}. We also experiment with synthetic data and sketched real data in Appendix \ref{syn-simu-1} and Appendix \ref{sec5.1.3append}. We include the simulation results for the the small regularizer regime discussed in Section \ref{small-regularizer} in Appendix \ref{sec5.1.4append}.

\vspace{-0.3cm}
\subsection{Experiments on Distributed Second-Order Optimization}\label{section5.2}


In this subsection, we include simulation results for distributed Newton's method and an inexact version where each Newton's step solved by the distributed preconditioned conjugate gradient method (see  Section \ref{section3} for a description). We implement distributed line search for choosing step sizes in all methods. Datasets are standardized by removing the mean and scaling to unit variance. Due to limited space, we present additional results on different datasets in Appendix \ref{c.2.1append}, the inexact Newton's method applied to ridge regression in Appendix \ref{c.2.2append} and experiments on logistic regression in
 Appendix \ref{logit}.





\begin{figure}[h]
\centering
\begin{subfigure}{0.48\linewidth}
\includegraphics[width=\linewidth]{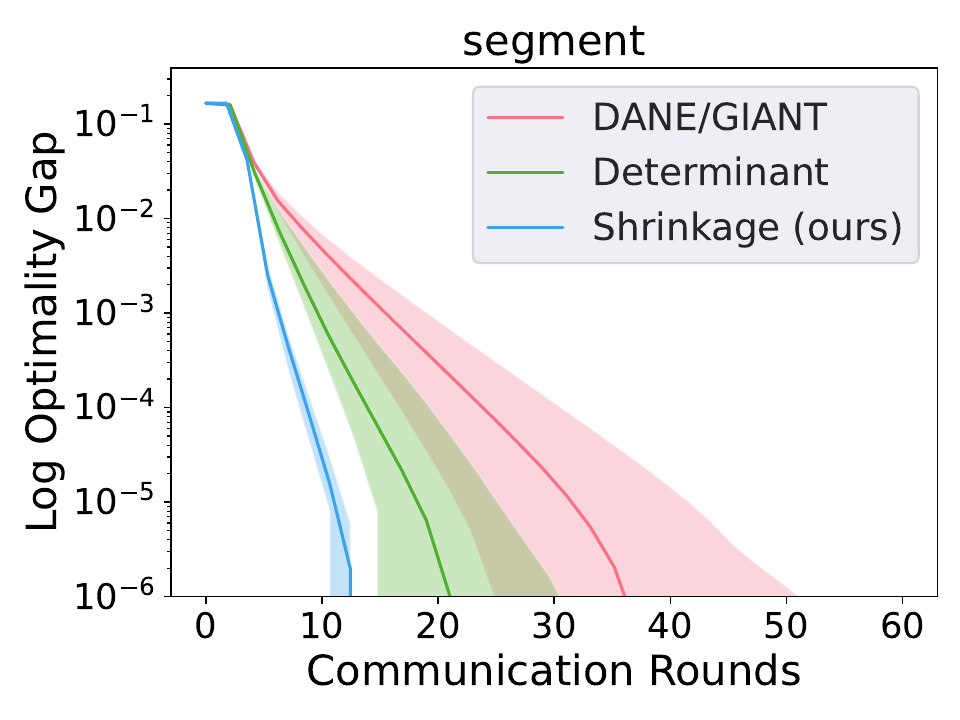}
\end{subfigure}
\begin{subfigure}{0.48\linewidth}
\includegraphics[width=\linewidth]{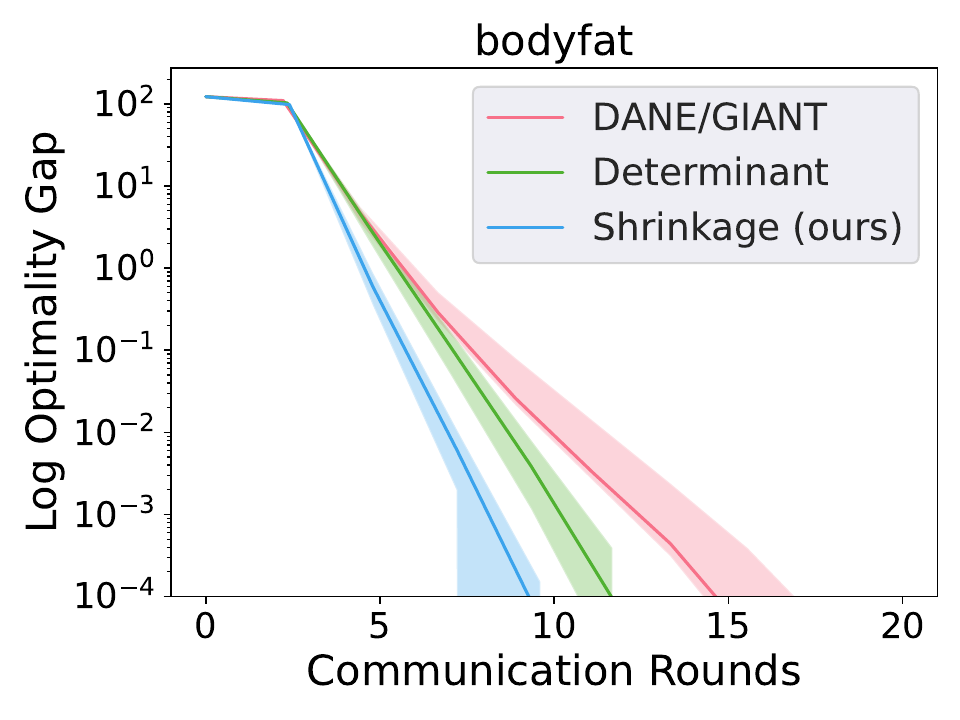}
\end{subfigure}

\begin{subfigure}{0.48\linewidth}
\includegraphics[width=\linewidth]{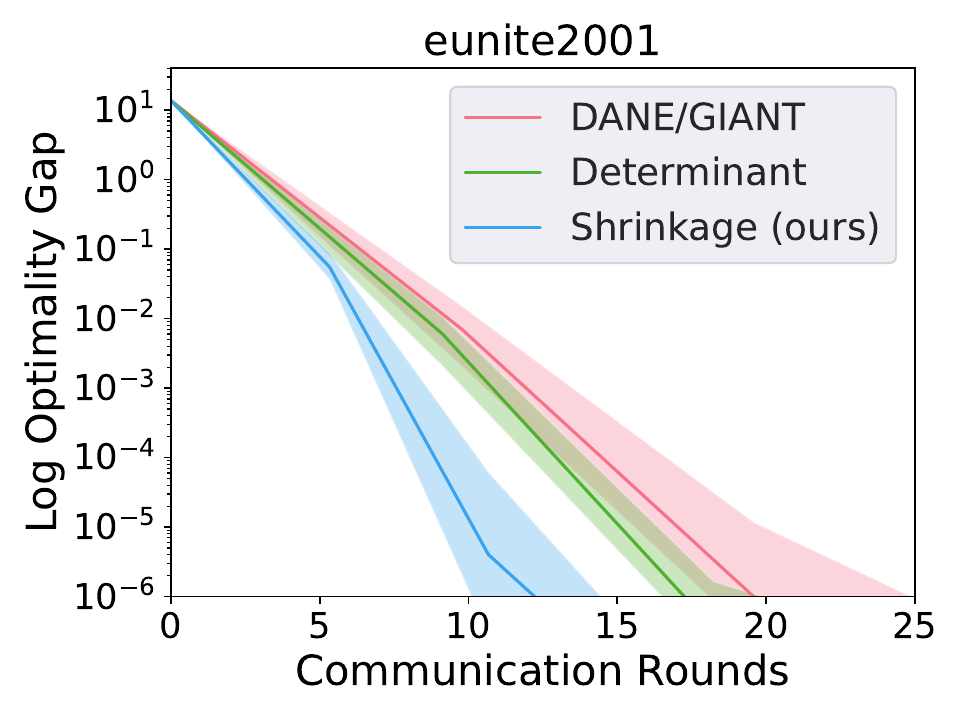}
\end{subfigure}
\begin{subfigure}{0.48\linewidth}
\includegraphics[width=\linewidth]{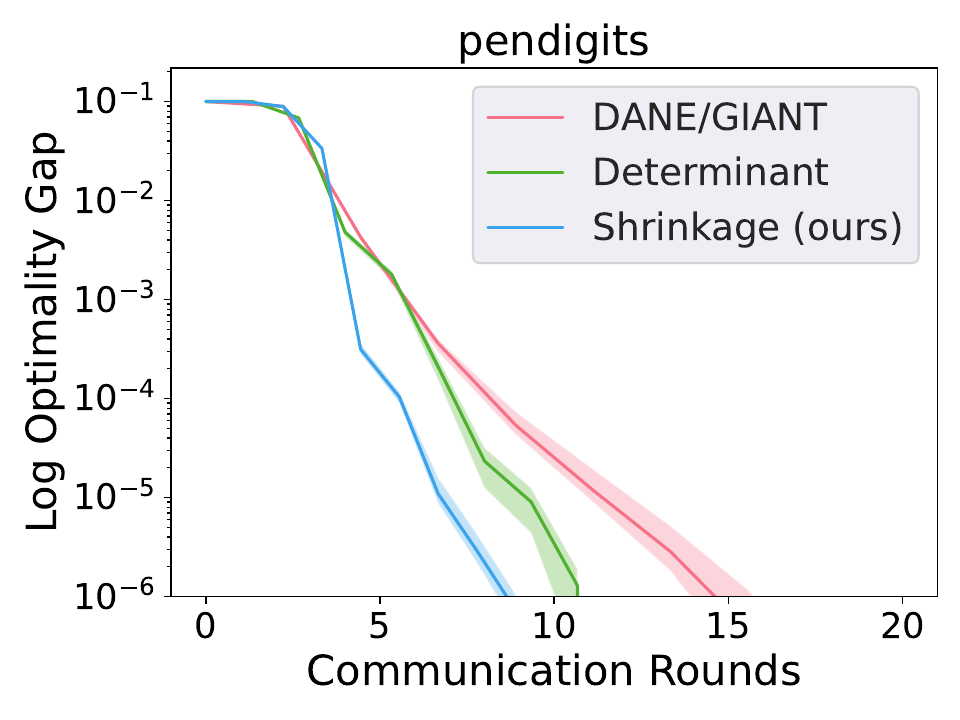}
\end{subfigure}
\caption{Experiments with real data on distributed Newton's method applied to ridge regression. Line search is used in all methods to determine the step sizes. Number of total samples is rounded down to a multiple of the number of agents and split evenly to each agent. We let $m=100,\lambda=0.1$ for segment, $m=20,\lambda=0.05$ for bodyfat, $m=20,\lambda=0.5$ for eunite2001, $m=100,\lambda=0.01$ for pendigits, where $\lambda$ is the regularization parameter and $m$ denotes the number of agents.}\label{newton-ridge-real}
\end{figure}
Figure \ref{newton-ridge-real} shows simulation results for distributed Newton's method applied to ridge regression. We compare with DANE and GIANT (see Section \ref{prior-work} for a discussion of these two methods). \emph{Determinant} stands for the determinantal averaging method (see also Section \ref{prior-work} for details). Note
DANE and GIANT coincides when regularized quadratic loss is minimized and both methods are simply taking the average of local descent direction as the global step, while the determinant averaging method adds a bias correction. The plots suggest that distributed Newton's method with optimal shrinkage achieves better log optimality gap within fewer communication rounds than other methods, which reveals that our shrinkage method is approximating the Hessian inverse more accurately. The ability of bias correction of determinantal averaging method can also be seen in the plots, but it is less effective compared to our approach.
\vspace{-0.1cm}

\subsection{Experiments on Iterative Hessian Sketch }\label{section5.3}

In this subsection, we consider the Iterative Hessian Sketch paradigm, which can be seen as a stochastic Newton's method and incorporate our optimal shrinkage. This method is discussed in Section \ref{section4}.  We consider the ridge regression problem  defined in Section \ref{quad}. We use line search for the step size. Since the sketching matrix is generated from a random Gaussian ensemble, our required assumptions for the optimal shrinkage formula holds. We use the effective dimension of the sketched data as an approximation. We include additional simulation results with the effective dimension of the true covariance in  Appendix \ref{ihs-append}.

Figure \ref{ihs-plot} presents simulation results for the IHS method and IHS with optimal shrinkage.  Although an inexact effective dimension is used for practicality, Iterative Hessian Sketch equipped with shrinkage still leads to significant speedups, which once more confirms our shrinkage formula's ability for bias correction in Hessian inverse estimation.

\begin{figure}[H]
\centering
\begin{subfigure}{0.48\linewidth}
\includegraphics[width=\linewidth]{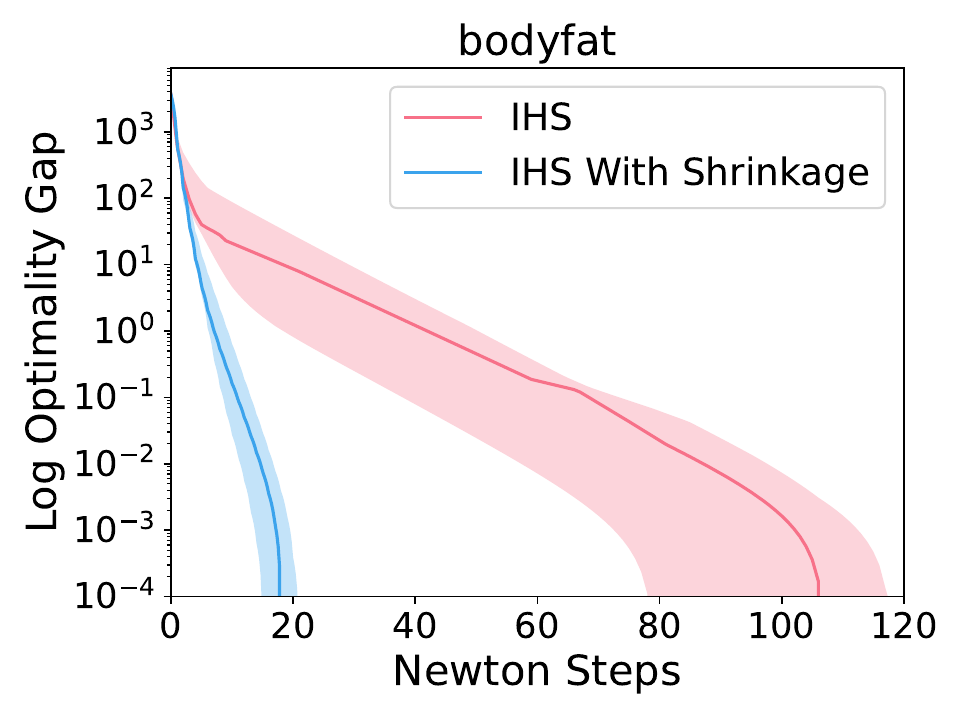}
\end{subfigure}
\begin{subfigure}{0.48\linewidth}
\includegraphics[width=\linewidth]{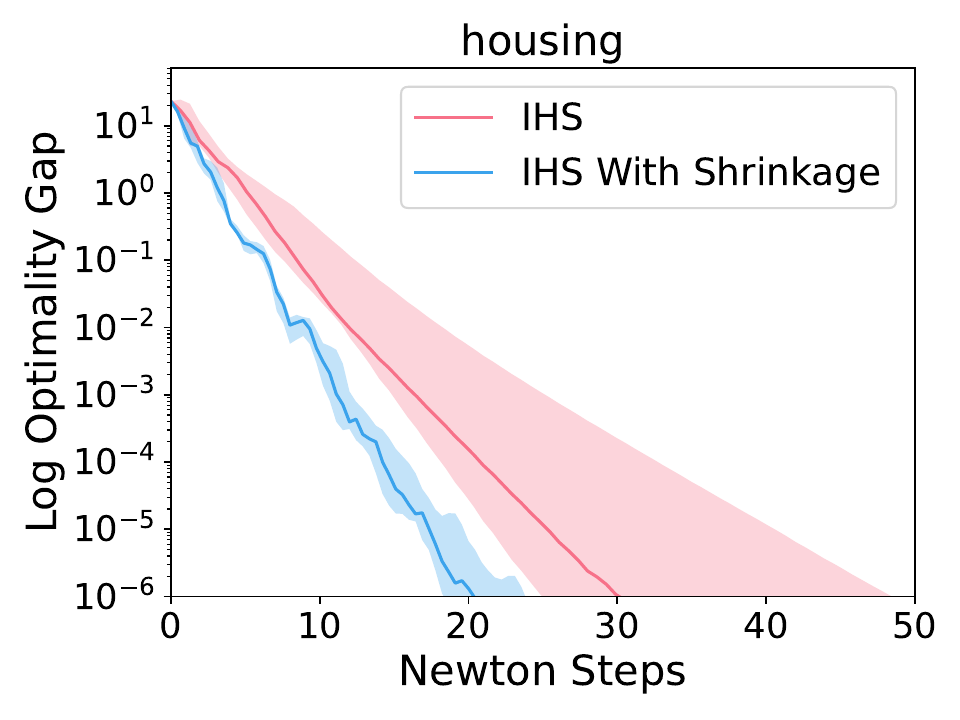}
\end{subfigure}

\begin{subfigure}{0.48\linewidth}
\includegraphics[width=\linewidth]{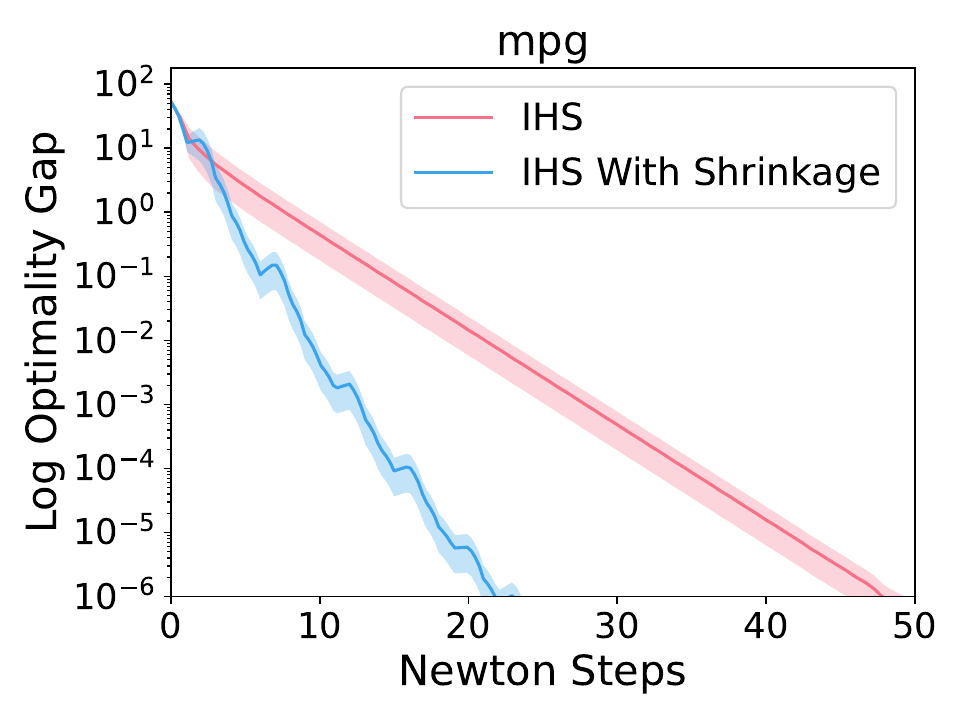}
\end{subfigure}
\begin{subfigure}{0.48\linewidth}
\includegraphics[width=\linewidth]{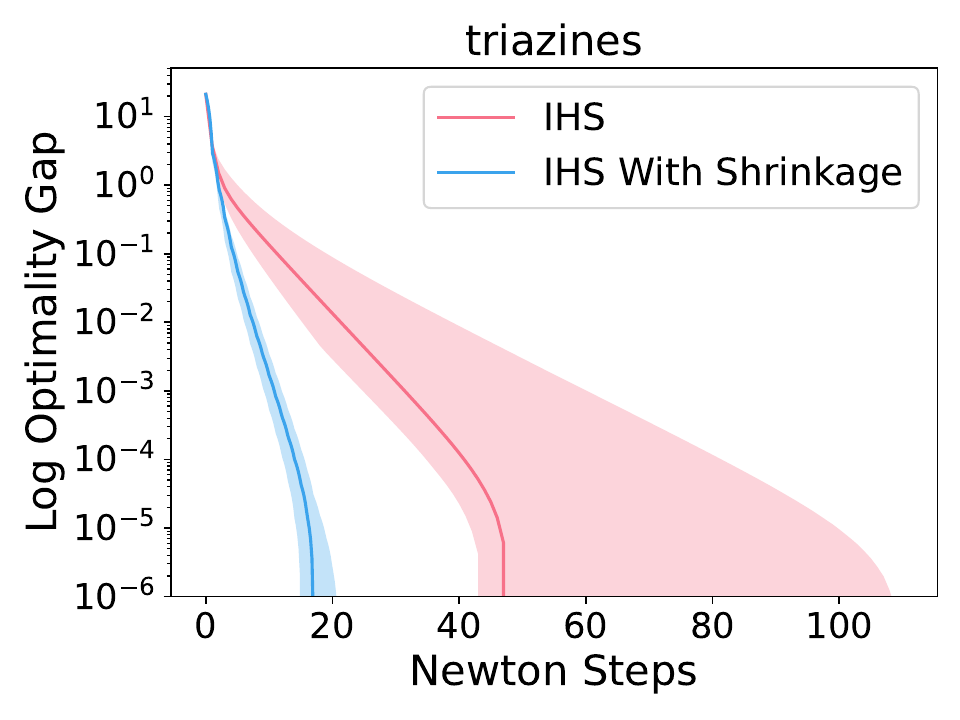}
\end{subfigure}
\caption{Experiments with real data on Iterative Hessian Sketch method applied to ridge regression. Line search is used to determine the step sizes. We let  $\lambda=0.01$ for bodyfat, housing, mpg and $\lambda=0.001$ for triazines; $m=100$ for bodyfat, $m=50$  for housing, $m=30$ for mpg, $m=300$ for triazines where $\lambda$ denotes the regularization parameter and $m$ denotes the sketch size.}\label{ihs-plot}
\end{figure}

\vspace{-0.8cm}
\section{Conclusion}\label{section6}
In this work, we addressed bias correction in distributed second-order optimization and sketching based methods. Specifically, both types of algorithms require accurate estimation of a Hessian inverse. When either data can be modeled random or data sketching is used, this problem amounts to estimation of the resolvent of appropriately defined covariance matrix. We studied an asymptotically unbiased estimator for this resolvent,  characterized its convergence rate, and leveraged it in Hessian inversion bias reduction, where significant convergence speedups are observed in real and synthetic datasets. One limitation of our theory is the need for prior knowledge of the effective dimension. Despite this limitation, we have shown that empirical approximations of the effective dimension yield highly accurate results. There exist other approaches to estimate the effective dimension, including trace estimation methods \cite{meyer2020hutch}, which can be used to further improve our scheme.



\vspace{-0.3cm}
\section*{Acknowledgements}

This work was supported in part by the National Science Foundation (NSF) under Grant ECCS-2037304 and Grant DMS-2134248; in part by the NSF CAREER Award under Grant CCF-2236829; in part by the U.S. Army Research Office Early Career Award under Grant W911NF-21-1-0242; in part by the Stanford Precourt Institute; and in part by the ACCESS—AI Chip Center for Emerging Smart Systems through InnoHK, Hong Kong, SAR.


\bibliography{example_paper}
\bibliographystyle{icml2023}


\newpage
\appendix
\onecolumn
\section{Proofs in Section \ref{section2}}\label{sec2proof}

\subsection{Technical Lemmas}
\begin{lemma}
\label{lemma-a1}
For any $z<0,z\in\mathbb{R}$, if
$F_{\Sigma_n}\rightarrow F_{\Sigma}\mbox{ almost everywhere}, M<\infty$, $\nu\rightarrow 0$. Then $\lim_{n,d\rightarrow\infty} \frac{1}{n}\mathbb{E}\left[\mbox{tr}\left(I-z\hat\Sigma_n\right)^{-1}\right]$ exists.
\end{lemma}
\begin{proof}
Denote $\hat\Sigma_{n}'=\hat\Sigma_n -\frac{1}{n}x_nx_n^T, \psi_n=\frac{1}{n}x_n^T(I-z\hat\Sigma_n)^{-1}x_n.$ By $y<1$, without loss of generality consider the regime $d<n$ for all $n$. Denote $s_0(z)=1-\frac{d}{n}+\frac{1}{n}\mathbb{E}\left[\mbox{tr}\left(I-z\hat\Sigma_n\right)^{-1}\right]$, then $s_0(z)\in (0,1]$. Follow Theorem 3.1 in book \cite{multi_stats}, 
\[\left(I-zs_0(z)\Sigma_n\right)\mathbb{E}\left[\left(I-z\hat \Sigma_n\right)^{-1}\right]=I+\Omega\]
where 
\begin{equation*}
\begin{aligned}
&\Omega=z^2\mathbb{E}\left[\left(I-z\hat\Sigma_n'\right)^{-1}x_nx_n^T\left(\psi_n-\mathbb{E}[\psi_n]\right)\right]+zs_0(z)\mathbb{E}\left[\left(I-z\hat\Sigma_n'\right)^{-1}-\left(I-z\hat\Sigma_n\right)^{-1}\right]\Sigma_n\\
&\mbox{and}\\
&\|\Omega\|\leq \frac{M z^2}{n}+\sqrt{5M}z^2\sqrt{\frac{M^2d^2z^2}{n^3}+\frac{d^2\nu}{n^2}}
\end{aligned}
\end{equation*}
Thus,
\[\mathbb{E}\left[\left(I-z\hat \Sigma_n\right)^{-1}\right]=\left(I-zs_0(z)\Sigma_n\right)^{-1}+\left(I-zs_0(z)\Sigma_n\right)^{-1}\Omega\]
Since $M$ bounded and $\nu$ diminishing,
\[\lim_{n,d\rightarrow\infty} \left\|\mathbb{E}\left[\left(I-z\hat \Sigma_n\right)^{-1}\right]-\left(I-zs_0(z)\Sigma_n\right)^{-1}\right\|=0\]
which indicates
\[\lim_{n,d\rightarrow\infty} \left(\frac{1}{n}\mathbb{E}\left[\mbox{tr}\left(\left(I-z\hat \Sigma_n\right)^{-1}\right)\right]-\frac{1}{n}\mbox{tr}\left(\left(I-zs_0(z)\Sigma_n\right)^{-1}\right)\right)=0\]
Denote $e_n=\frac{1}{n}\mathbb{E}\left[\mbox{tr}\left(\left(I-z\hat \Sigma_n\right)^{-1}\right)\right],z_n=z(1-\frac{d}{n})$,
\begin{equation*}
    \begin{aligned}
\lim_{n,d\rightarrow\infty} \left(e_n-\frac{1}{n}\mbox{tr}\left(\left(I-(z_n+ze_n)\Sigma_n\right)^{-1}\right)\right)&=0
    \end{aligned}
\end{equation*}
Since $F_{\Sigma_n}\rightarrow F_\Sigma$ abd $\frac{d}{n}\rightarrow y$, denote the Stieltjes transform as $m_\mu(z)=\int_R \frac{1}{x-z}\mu(dx)$, define 
\[f(e_n)=-\frac{y}{(z(1-y)+ze_n)}m_\Sigma\left(\frac{1}{z(1-y)+ze_n}\right)\]
with $f(e_n)$ monotone decreasing and
\[\lim_{n,d\rightarrow\infty} \left(e_n-f(e_n)\right)=0\]
which indicates that for any $\epsilon>0,$ there exists $N>0$ such that for any $n_1,n_2>N$, $|(e_{n_1}-f(e_{n_1}))-(e_{n_2}-f(e_{n_2}))|<\epsilon$. Assume $\lim_{n,d\rightarrow\infty}  e_n$ doesn't exist, then there exists $l_1,l_2>N$ and $|e_{l_1}-e_{l_2}|>\epsilon$. Without loss of generality assume $e_{l_1}>e_{l_2}$. But then
\[|(e_{l_1}-f(e_{l_1}))-(e_{l_2}-f(e_{l_2}))|=|(e_{l_1}-e_{l_2})-(f(e_{l_1})-f(e_{l_2}))|>\epsilon\]
Contradiction. Therefore $\lim_{n,d\rightarrow\infty} e_n=\lim_{n,d\rightarrow\infty} \frac{1}{n}\mathbb{E}\left[\mbox{tr}\left(I-z\hat \Sigma_n\right)^{-1}\right]$ exists.
\end{proof}

\begin{lemma}
\label{lemma-a2}
If $F_{\Sigma_n}(u)\rightarrow F_\Sigma(u)$ almost surely for almost every $u>0$, $\lim_{n,d\rightarrow\infty}\frac{d_\lambda^n}{n}$ exists for any $\lambda>0$.
\end{lemma}

\begin{proof}
\begin{equation*}
\begin{aligned}
\frac{d_\lambda^n}{n}&= \frac{\mbox{tr}(\Sigma_n(\Sigma_n+\lambda I)^{-1})}{n}
\\&=\frac{\mbox{tr}(I-\lambda(\Sigma_n+\lambda I)^{-1})}{n}
\\&=\frac{d}{n}- \frac{\lambda}{n}\mbox{tr}(\Sigma_n+\lambda I)^{-1}
\end{aligned}
\end{equation*}
Denote the Stieltjes transform as $m_\mu(z)=\int_R \frac{1}{x-z}\mu(dx)$. Since $F_{\Sigma_n}\rightarrow F_\Sigma$ almost everywhere, $\lim_{n,d\rightarrow\infty} \frac{1}{n}\mbox{tr}(\Sigma_n+\lambda I)^{-1}$ exists and 
\[\lim_{n,d\rightarrow\infty} \frac{1}{n}\mbox{tr}(\Sigma_n+\lambda I)^{-1}=y m_{F_{\Sigma}}(-\lambda)\in (0,\frac{y}{\lambda}]\]
Thus
\[\lim_{n,d\rightarrow\infty}\frac{d_\lambda^n}{n}=y-\lambda y m_{F_{\Sigma}}(-\lambda)\in [0,y)\mbox{ exists}\] 
\end{proof}

\begin{lemma}
\label{lemma-a3}
If $y>0,$ $F_{\Sigma_n}(u)\rightarrow F_\Sigma(u)$ almost surely for almost every $u>0$, eigenvalues of each $\Sigma_n$ are located on a segment $[\sigma_{\min},\sigma_{\max}]$. For any $z<0,s>0,z,s\in \mathbb{R},$  the fixed point equation in $s$
\begin{equation}\label{proof_eq01}
s=1+\lim_{n,d\rightarrow\infty} \frac{1}{n}\mbox{tr}\left(zs\Sigma_n\left(I-zs\Sigma_n\right)^{-1}\right)
\end{equation}
has at most one non-negative real solution.
\end{lemma}

\begin{proof}
First note zero is not a solution to (\ref{proof_eq01}). Assume (\ref{proof_eq01}) has two positive real solutions $s_1$ and $s_2$ such that $s_1\neq s_2$. Without loss of generality assume $s_1>s_2$. Then 
\[s_1=1-y+\lim_{n,d\rightarrow\infty} \frac{\mbox{tr}(I-zs_1\Sigma_n)^{-1}}{n}=1-y+\lim_{n,d\rightarrow\infty} \frac{\mbox{tr}(I-zs_2\Sigma_n)^{-1}}{n}=s_2\]
Therefore,
\[\lim_{n,d\rightarrow\infty} \frac{1}{n}\left(\mbox{tr}(I-zs_1\Sigma_n)^{-1}-\mbox{tr}(I-zs_2\Sigma_n)^{-1}\right)=0\]
Apply resolvent identity, 
\[\lim_{n,d\rightarrow\infty} \frac{z(s_1-s_2)}{n}\mbox{tr}\left(\Sigma_n(I-zs_2\Sigma_n)^{-1}(I-zs_1\Sigma_n)^{-1}\right)=0\]
Equivalently,
\[\frac{1}{z}\left(\frac{1}{s_2}-\frac{1}{s_1}\right)\lim_{n,d\rightarrow\infty} \frac{1}{n}\sum_{i=1}^{d} \frac{\sigma_{ni}}{\left(\sigma_{ni}-\frac{1}{zs_2}\right)\left(\sigma_{ni}-\frac{1}{zs_1}\right)}=0\]
where $\sigma_{ni}$ is the $i$th eigenvalue of $\Sigma_n$. But
\begin{equation*}
    \begin{aligned}
\frac{1}{z}\left(\frac{1}{s_2}-\frac{1}{s_1}\right)\lim_{n,d\rightarrow\infty} \frac{1}{n}\sum_{i=1}^{d} \frac{\sigma_{ni}}{\left(\sigma_{ni}-\frac{1}{zs_2}\right)\left(\sigma_{ni}-\frac{1}{zs_1}\right)}\leq \frac{y}{z}\left(\frac{1}{s_2}-\frac{1}{s_1}\right)\frac{\sigma_{\min}}{\left(\sigma_{\max}-\frac{1}{zs_2}\right)\left(\sigma_{\max}-\frac{1}{zs_1}\right)}
\end{aligned}
\end{equation*}
Thus, $s_1=s_2.$
\end{proof}

\begin{lemma} \label{lemma-a4}
With $\Sigma_n=I$. Fix a small $\omega\in(0,1)$. Define the domain $\mathcal{D}=\{z=u+vi\in\mathbb{C}:\sqrt{y}+\frac{1}{\sqrt{y}}-2+|\frac{u}{\sqrt{y}}|\leq\omega^{-1},d^{\omega-1}\leq \frac{v}{\sqrt{y}}\leq \omega^{-1},|z|\geq\omega\sqrt{y}\}.$ Then
\[\left|\mathbb{E}\left[\frac{1}{n}\mbox{tr}(\hat\Sigma_n-zI)^{-1}\right]-\lim_{n,d\rightarrow\infty}\mathbb{E}\left[\frac{1}{n}\mbox{tr}(\hat\Sigma_n-zI)^{-1}\right]\right|\leq \frac{1}{v\sqrt{n d}}\]
for any $z\in\mathcal{D}$.
\end{lemma}
\begin{proof} 
Lemma \ref{lemma-a4} is a direct corollary from Theorem 2.4 in \cite{bloemendal2013isotropic}. For sake of complementness, we restate the theorem here again together with the derivation of Lemma \ref{lemma-a4}.
\begin{lemma}
(Theorem 2.4 in \cite{bloemendal2013isotropic}) With $\Sigma_n=\frac{1}{\sqrt{y}}I$. Fix a small $\omega\in(0,1)$. Define the domain $\mathcal{D}=\{z=u+vi\in\mathbb{C}:\sqrt{y}+\frac{1}{\sqrt{y}}-2+|u|\leq\omega^{-1},d^{\omega-1}\leq v\leq \omega^{-1},|z|\geq\omega\}.$ Then
\[\left|\frac{1}{d}\mbox{tr}(\hat\Sigma_n-zI)^{-1}-m_\phi\right|\prec \frac{1}{d v}\]

uniformly for $z\in\mathcal{D}$, where $\prec$ denote stochastic dominance and $m_\phi$ is Stieltjes transform of the Marchenko-Pastur law with variance $\frac{1}{\sqrt{y}}$. 
\end{lemma}
Note given $\left|\frac{1}{d}\mbox{tr}(\hat\Sigma_n-zI)^{-1}-m_\phi\right|\prec \frac{1}{d v}$, since $\frac{1}{d}\mbox{tr}(\hat\Sigma_n-zI)^{-1} \leq z$ for each $n$, by dominated convergence theorem, we can bound
\begin{equation*}
    \begin{aligned}
    &\left|\mathbb{E}\left[\frac{1}{d}\mbox{tr}(\hat\Sigma_n-z I)^{-1}\right]-\lim_{n,d\rightarrow\infty}\mathbb{E}\left[\frac{1}{d}\mbox{tr}(\hat\Sigma_n-zI)^{-1}\right]\right|=\left|\mathbb{E}\left[\frac{1}{d}\mbox{tr}(\hat\Sigma_n-z I)^{-1}\right]-\mathbb{E}[m_\phi]\right|\\
    &\leq \mathbb{E}\left[\left|\frac{1}{d}\mbox{tr}(\hat\Sigma_n-zI)^{-1}-m_\phi\right|\right]\\
    &< 2z d^{-D} + \frac{d^\epsilon}{d v}(1-d^{-D})
    \end{aligned}
\end{equation*}
for any $\epsilon>0, D>0$. Since $\epsilon$ can be taken arbitrarily close to 0 and D can be taken arbitrarily large, 
\begin{equation*}
    \begin{aligned}
    \left|\mathbb{E}\left[\frac{1}{d}\mbox{tr}(\hat\Sigma_n-z I)^{-1}\right]-\lim_{n,d\rightarrow\infty}\mathbb{E}\left[\frac{1}{d}\mbox{tr}(\hat\Sigma_n-zI)^{-1}\right]\right|\leq \frac{1}{d v}
    \end{aligned}
\end{equation*}

With some algebra, we get Lemma \ref{lemma-a4}
\end{proof}

\begin{lemma}\label{lemma-a6}
Under Assumption \ref{ass_1}, for any $z>0, \epsilon_n>0$, if
\[\left|\mathbb{E}\left[\frac{\mbox{tr}(\hat\Sigma_n+(z-\epsilon_n i)I)^{-1}}{n}\right]-\lim_{n,d\rightarrow\infty} \mathbb{E}\left[\frac{\mbox{tr}(\hat\Sigma_n+(z-\epsilon_n i)I)^{-1}}{n}\right|\right|\leq \Omega(n,\epsilon_n)\]
then
\[\left|\mathbb{E} \left[\frac{\mbox{tr}(\hat\Sigma_n+z I)^{-1}}{n}\right]-\lim_{n,d\rightarrow\infty} \mathbb{E}\left[\frac{\mbox{tr}(\hat\Sigma_n+z I)^{-1}}{n}\right]\right|\leq \frac{2d|\epsilon_n|}{nz^2}+\Omega(n,\epsilon_n)\]
\end{lemma}
\begin{proof}
by resolvent identity,
\begin{equation*}
    \begin{aligned}
    &\left|\mathbb{E}\left[\frac{\mbox{tr}(\hat\Sigma_n+z I)^{-1}}{n}\right]-\mathbb{E}\left[\frac{\mbox{tr}(\hat\Sigma_n+(z-\epsilon_n i)I)^{-1}}{n}\right]\right|
    \\&= \left|\frac{1}{n}\mathbb{E}\left[\mbox{tr}\left((\hat\Sigma_n+z I)^{-1}(-\epsilon_n i I)(\hat\Sigma_n+(z-\epsilon_n i)I)^{-1}\right)\right]\right|\\&\leq \frac{d|\epsilon_n|}{nz^2}
    \end{aligned}
\end{equation*}
Since the above inequality holds for each $n,$
\[\left|\lim_{n,d\rightarrow\infty} \mathbb{E}\left[\frac{\mbox{tr}(\hat\Sigma_n+z I)^{-1}}{n}\right]-\lim_{n,d\rightarrow\infty}  \mathbb{E}\left[\frac{\mbox{tr}(\hat\Sigma_n+(z-\epsilon_n i)I)^{-1}}{n}\right]\right|\leq \frac{|\epsilon_n|d}{nz^2}\]
Thus,
\begin{equation*}
    \begin{aligned}
    &\left|\mathbb{E} \left[\frac{\mbox{tr}(\hat\Sigma_n+z I)^{-1}}{n}\right]-\lim_{n,d\rightarrow\infty} \mathbb{E}\left[\frac{\mbox{tr}(\Sigma_n+z I)^{-1}}{n}\right]\right|
    \\
    \leq &\left|\left(\mathbb{E} \left[\frac{\mbox{tr}(\hat\Sigma_n+z I)^{-1}}{n}\right]-\mathbb{E} \left[\frac{\mbox{tr}(\hat\Sigma_n+(z-\epsilon_n i)I)^{-1}}{n}\right]\right)\right|+\\&\left|\left(\lim_{n,d\rightarrow\infty} \mathbb{E} \left[\frac{\mbox{tr}(\hat\Sigma_n+z I)^{-1}}{n}\right]-\lim_{n,d\rightarrow\infty} \mathbb{E} \left[\frac{\mbox{tr}(\hat\Sigma_n+(z-\epsilon_n i)I)^{-1}}{n}\right]\right)\right|+\\&\left|\left(\mathbb{E} \left[\frac{\mbox{tr}(\hat\Sigma_n+(z-\epsilon_n i)I)^{-1}}{n}\right]-\lim_{n,d\rightarrow\infty} \mathbb{E} \left[\frac{\mbox{tr}(\hat\Sigma_n+(z-\epsilon_n i)I)^{-1}}{n}\right]\right)\right|\\
    \leq & \frac{2d|\epsilon_n|}{n z^2}+\Omega(n,\epsilon_n)
    \end{aligned}
\end{equation*}
\end{proof}

\begin{lemma}
\label{thm1}
Under Assumption \ref{ass_1}, for any $\lambda>0$, denote $\overline{d_\lambda}=\liminf_{n,d\rightarrow\infty} n^{-1}d_\lambda^n$, then
\[\mathbb{E}\left[\left(\frac{1}{1-\overline{d_\lambda}}\hat \Sigma_n+\lambda I\right)^{-1}\right]=(\Sigma_n+\lambda I)^{-1}+\Omega_2\]
where $\|\Omega_2\|\rightarrow 0$ as $n,d\rightarrow\infty, d/n\rightarrow y\in [0,1)$.
\end{lemma}
\begin{proof}
The existence of $\overline{d_\lambda}=\lim_{n,d\rightarrow\infty}\frac{d_\lambda^n}{n}$ has been established in Lemma \ref{lemma-a2}. Take any $z<0,z\in\mathbb{R}$, by Lemma \ref{lemma-a1}, $\lim_{n,d\rightarrow\infty} \mathbb{E}\left[\frac{\mbox{tr}(I-z\hat\Sigma_n)^{-1}}{n}\right]$ exists. Denote $s(z)=1-y+\lim_{n,d\rightarrow\infty} \mathbb{E}\left[\frac{\mbox{tr}(I-z\hat\Sigma_n)^{-1}}{n}\right]$, by Lemma \ref{lemma-a1} again, $s(z)=\lim_{n,d\rightarrow \infty} s_0(z)\in [0,1]$. We can compute 
\begin{equation}\label{eq1}
    \begin{aligned}
(I-zs(z)\Sigma_n)\mathbb{E}\left[(I-z\hat\Sigma_n)^{-1}\right]=I+\Omega+\Omega'
    \end{aligned}
\end{equation}
with $\Omega'=z\left(s_0(z)-s(z)\right)\Sigma_n\mathbb{E}\left[(I-z\hat\Sigma_n)^{-1}\right]$. Thus,
\begin{equation}\label{proof_eq0}
\begin{aligned}
\mathbb{E}\left[(I-z\hat\Sigma_n)^{-1}\right]&=(I-zs(z)\Sigma_n)^{-1}+(I-zs(z)\Sigma_n)^{-1}(\Omega+\Omega')
\end{aligned}
\end{equation}
We can bound 
\begin{equation*}
\begin{aligned}
\|(I-z s(z)\Sigma_n)^{-1}(\Omega+\Omega')\|&\leq \|\Omega+\Omega'\|\\
&\leq \|\Omega\|+\|\Omega'\|\\
&\leq \frac{M z^2}{n}+\sqrt{5M}z^2\sqrt{\frac{M^2d^2z^2}{n^3}+\frac{d^2\nu}{n^2}}+\sigma_{\max}|z||s_0(z)-s(z)|
\end{aligned}
\end{equation*}
Since $M$ stays bounded, $\nu\rightarrow 0$ and $s(z)=\lim_{n,d\rightarrow\infty}s_0(z)$, 
\[\lim_{n,d\rightarrow\infty}\left\|\mathbb{E}\left[(I-z\hat\Sigma_n)^{-1}\right]-(I-zs(z)\Sigma_n)^{-1}\right\|=0\] 
which indicates 
\[\lim_{n,d\rightarrow\infty}\left(\frac{1}{n}\mathbb{E}\left[\mbox{tr}(I-z\hat\Sigma_n)^{-1}\right]-\frac{1}{n}\mbox{tr}(I-zs(z)\Sigma_n)^{-1}\right)=0\]
by Lemma \ref{lemma-a1}, $\lim_{n,d\rightarrow\infty} \frac{1}{n}\mathbb{E}\left[\mbox{tr}(I-z\hat\Sigma_n)^{-1}\right]$ exists and therefore,
\begin{equation}\label{proof_eq00}
    \begin{aligned}
\lim_{n,d\rightarrow\infty}\left(\frac{1}{n}\mathbb{E}\left[\mbox{tr}(I-z\hat\Sigma_n)^{-1}\right]\right)=\lim_{n,d\rightarrow\infty}\left(\frac{1}{n}\mbox{tr}(I-zs(z)\Sigma_n)^{-1}\right)
    \end{aligned}
\end{equation}
substitute (\ref{proof_eq00}) back into expression for $s(z)$, we get
\begin{equation}\label{proof_eq1}
\begin{aligned}s(z)&=1-y+\lim_{n,d\rightarrow\infty} \frac{\mbox{tr}(I-zs(z)\Sigma_n)^{-1}}{n}\\&=1-\lim_{n,d\rightarrow\infty}\left(\frac{d}{n}-\frac{\mbox{tr}(I-zs(z)\Sigma_n)^{-1}}{n}\right)\\&=1+\lim_{n,d\rightarrow\infty} \frac{\mbox{tr}(zs(z)\Sigma_n(I-zs(z)\Sigma_n)^{-1})}{n}
\end{aligned}
\end{equation}
Set $z=-\frac{1}{\lambda(1-\overline{d_\lambda})}$. Since $\frac{d_\lambda^n}{n}\leq \frac{d}{n}$ for each $n$, $\overline{d_\lambda}\leq y<1$, and therefore $z<0$. Note $s(z)=1-\overline{d_\lambda}$ satisfies (\ref{proof_eq1}). When $y>0$, by Lemma \ref{lemma-a3}, we conclude $s(z)=1-\overline{d_\lambda}$.
When $y=0,\overline{d_\lambda}=0$. We get $s(z)=1+\lim_{n,d\rightarrow\infty} \frac{\mbox{tr}(I-zs(z)\Sigma_n)^{-1}}{n}\geq 1$ from (\ref{proof_eq1}), but since $s(z)\in [0,1]$, $s(z)=1=1-\overline{d_\lambda}$. Thus, $s(z)=1-\overline{d_\lambda}\mbox{ for }y\in [0,1)$.
Substitute $(z=-\frac{1}{\lambda(1-\overline{d_\lambda})},s(z)=1-\overline{d_\lambda})$ back into (\ref{proof_eq0}), we get
\begin{equation*}
\begin{aligned}
\mathbb{E}\left[\left(\frac{1}{1-\overline{d_\lambda}}\hat\Sigma_n+\lambda I\right)^{-1}\right]&=(\Sigma_n+\lambda I)^{-1}+(\Sigma_n+\lambda I)^{-1}(\Omega+\Omega')\\
&=(\Sigma_n+\lambda I)^{-1}+\Omega_2
\end{aligned}
\end{equation*}
with $\|\Omega_2\|=\|(\Sigma_n+\lambda I)^{-1}(\Omega+\Omega')\|$.
Therefore, we can bound
\begin{equation}\label{convergence_bound}
    \begin{aligned}
\|\Omega_2\|&=\|(\Sigma_n+\lambda I)^{-1}(\Omega+\Omega')\|\\&\leq \frac{1}{\lambda}\|\Omega+\Omega'\|\\&\leq \frac{1}{\lambda}\left(\frac{M z^2}{n}+\sqrt{5M}z^2\sqrt{\frac{M^2d^2z^2}{n^3}+\frac{d^2\nu}{n^2}}+\sigma_{\max}|z|\left|s_0(z)-s(z)\right|\right)\\&\leq \frac{1}{\lambda}\left(\frac{M z^2}{n}+\sqrt{5M}z^2\sqrt{\frac{M^2d^2z^2}{n^3}+\frac{d^2\nu}{n^2}}+\sigma_{\max}|z|\left|\left(y-\frac{d}{n}\right)+\left(\mathbb{E}\left[\frac{1}{n}\mbox{tr}\left(I-z\hat\Sigma_n\right)^{-1}\right]-\right .\right.\right.\\&\left.\left.\left.\quad \lim_{n,d\rightarrow\infty}\mathbb{E}\left[\frac{1}{n}\mbox{tr}\left(I-z\hat\Sigma_n\right)^{-1}\right]\right)\right|\right)
    \end{aligned}
\end{equation}
where $z=-\frac{1}{\lambda(1-\overline{d_\lambda})}$. Thus $\|\Omega_2\|\rightarrow 0$ as $n,d\rightarrow\infty$.
\end{proof}

\subsection{Proof of Theorem \ref{thm2}}\label{thm2proof}
\begin{proof}
Theorem \ref{thm2} can be derived as a corollary of Lemma \ref{thm1}. We include its derivation for sake of completeness.  Build on proof of Theorem \ref{thm1}, 
\begin{equation*}
\begin{aligned}
\mathbb{E}\left[\left(\frac{1}{1-\frac{d_\lambda^n}{n}}\hat\Sigma_n+\lambda I\right)^{-1}\right]=(\Sigma_n+\lambda I)^{-1}+\left(\mathbb{E}\left[\left(\frac{1}{1-\frac{d_\lambda^n}{n}}\hat\Sigma_n+\lambda I\right)^{-1}\right] \right.\\\left.-\mathbb{E}\left[\left(\frac{1}{1-\lim_{n,d\rightarrow\infty} \frac{d_\lambda^n}{n}}\hat\Sigma_n+\lambda I\right)^{-1}\right]\right)+\Omega_2
\end{aligned}
\end{equation*}
Denote $\omega_n=\lim_{n,d\rightarrow\infty} \frac{d_\lambda^n}{n}-\frac{d_\lambda^n}{n}$, by resolvent identity, 
\begin{equation*}
\begin{aligned}
&\left\|\mathbb{E}\left[\left(\frac{1}{1-\frac{d_\lambda^n}{n}}\hat\Sigma_n+\lambda I\right)^{-1}\right] -\mathbb{E}\left[\left(\frac{1}{1-\lim_{n,d\rightarrow\infty} \frac{d_\lambda^n}{n}}\hat\Sigma_n+\lambda I\right)^{-1}\right]\right\|
\\&=\left\|\mathbb{E}\left[\left(\frac{1}{1-\frac{d_\lambda^n}{n}}\hat\Sigma_n+\lambda I\right)^{-1}\left(\frac{\omega_n}{\left(1-\lim_{n,d\rightarrow\infty} \frac{d_\lambda^n}{n}\right)\left(1-\frac{d_\lambda^n}{n}\right)}\hat\Sigma_n\right)\left(\frac{1}{1-\lim_{n,d\rightarrow\infty} \frac{d_\lambda^n}{n}}\hat\Sigma_n+\lambda I\right)^{-1}\right]\right\|
\\&\leq \lambda^2 \frac{|\omega_n|}{\left(1-\lim_{n,d\rightarrow\infty} \frac{d_\lambda^n}{n}\right)\left(1-\frac{d_\lambda^n}{n}\right)}\|\mathbb{E}[\hat\Sigma_n]\|\\
&\leq \frac{\lambda^2\sigma_{\max}|\omega_n|}{\left(1-\lim_{n,d\rightarrow\infty} \frac{d_\lambda^n}{n}\right)\left(1-\frac{d_\lambda^n}{n}\right)}
\end{aligned}
\end{equation*}
Since $\lim_{n,d\rightarrow\infty}\frac{d_\lambda^n}{n}\leq \frac{y+1}{2}<1$, therefore,
\begin{equation}\label{convergence_bound2}
\begin{aligned}
\left\|\mathbb{E}\left[\left(\frac{1}{1- \frac{d_\lambda^n}{n}}\hat \Sigma_n+\lambda I\right)^{-1}\right]-(\Sigma_n+\lambda I)^{-1}\right\|\leq \frac{2\sigma_{\max}\lambda^2|\omega_n|}{\left(1-y\right)\left(1-\frac{d_\lambda^n}{n}\right)}+\|\Omega_2\|
\end{aligned}
\end{equation}
with the right-hand side diminishing to 0.
\end{proof}

\subsection{Proof of Theorem \ref{thm2.4}}\label{thm2.4proof}

\begin{proof}
When $\Sigma_n=I$,  $\frac{d_\lambda^n}{n}=\frac{y}{\lambda+1}$ for each $n$ and thus $\overline{d_\lambda}=\frac{y}{\lambda+1}$. Let $z=\lambda(1-\frac{y}{\lambda+1})$ and $\epsilon_n=\frac{1}{\sqrt{n}}$ in Lemma \ref{lemma-a6}, under Assumption \ref{ass_1} and with Lemma \ref{lemma-a4},
\begin{equation*}
\begin{aligned}
&\left|\mathbb{E}\left[\frac{1}{n}\mbox{tr}\left(\hat\Sigma_n+\lambda\left(1-\frac{y}{\lambda+1}\right)I\right)^{-1}\right]-\lim_{n,d\rightarrow\infty}\mathbb{E}\left[\frac{1}{n}\mbox{tr}\left(\hat\Sigma_n+\lambda\left(1-\frac{y}{\lambda+1}\right)I\right)^{-1}\right]\right|\\
&\leq \frac{2d \epsilon_n}{n z^2}+\frac{1}{\epsilon_n\sqrt{nd}}=\frac{2y}{z^2\sqrt{n}}+\frac{1}{\sqrt{d}}
\end{aligned}
\end{equation*}
Since $\|\Sigma_n\|=1$ for each $n$, $\sigma_{\max}=1$ and from (\ref{convergence_bound}), denote $z'=-\frac{1}{z}$,
\begin{equation*}
    \begin{aligned}
    &\|\Omega_2\|\leq \frac{1}{\lambda}\left(\frac{M z'^2}{n}+\sqrt{5M}z'^2\sqrt{\frac{M^2d^2z'^2}{n^3}+\frac{d^2\nu}{n^2}}+\sigma_{\max}|z'|\left|s_0(z')-s(z')\right|\right)\\&\leq \frac{1}{\lambda}\left(\frac{M z'^2}{n}+\sqrt{5M}z'^2\sqrt{\frac{M^2d^2z'^2}{n^3}+\frac{d^2\nu}{n^2}}+|z'|\left|\left(\mathbb{E}\left[\frac{1}{n}\mbox{tr}\left(I-z'\hat\Sigma_n\right)^{-1}\right]-\right .\right.\right.\\&\left.\left.\left.\quad \lim_{n,d\rightarrow\infty}\mathbb{E}\left[\frac{1}{n}\mbox{tr}\left(I-z'\hat\Sigma_n\right)^{-1}\right]\right)\right|\right)\\
    &\leq \frac{1}{\lambda}\left(\frac{M z'^2}{n}+\sqrt{5M}z'^2\sqrt{\frac{M^2d^2z'^2}{n^3}+\frac{d^2\nu}{n^2}}+\frac{1}{\sqrt{n}}\left(2yz'^2+\frac{1}{\sqrt{y}}\right)\right)\\
    &=\frac{1}{\lambda}\left(\frac{Mz'^2}{n}+z'^2y\sqrt{5M}\sqrt{\frac{M^2z'^2}{n}+\nu}+\frac{1}{\sqrt{n}}\left(2yz'^2+\frac{1}{\sqrt{y}}\right)\right)\in\mathcal{O}(\frac{1}{\sqrt{n}}+\sqrt{\nu})
    \end{aligned}
\end{equation*}
In (\ref{convergence_bound2}), since $|\omega_n|=0$, 
\begin{equation*}
\begin{aligned}
\left\|\mathbb{E}\left[\left(\frac{1}{1- \frac{d_\lambda^n}{n}}\hat \Sigma_n+\lambda I\right)^{-1}\right]-(\Sigma_n+\lambda I)^{-1}\right\|\leq \|\Omega_2\|\in\mathcal{O}(\frac{1}{\sqrt{n}}+\sqrt{\nu})
\end{aligned}
\end{equation*}
\end{proof}

\subsection{Proof of Theorem \ref{thm2.5}}\label{thm2.5proof}

\begin{proof}
Assume eigenvalues of $\hat \Sigma_n$ are no less than $\sigma>0$. By resolvent identity,
\begin{equation*}
    \begin{aligned}
    &\left \|\mathbb{E}\left[\left(\frac{1}{1-\frac{d_\epsilon^n}{n}}\hat\Sigma_n+\epsilon I\right)^{-1}\right]-\mathbb{E}\left[\left(\frac{1}{1-\frac{d}{n}}\hat\Sigma_n+\epsilon I\right)^{-1}\right]\right\|\\
    &=\left\|\mathbb{E}\left[\epsilon \left(\frac{1}{1-\frac{d_\epsilon^n}{n}}\hat\Sigma_n+\epsilon I\right)^{-1}\left(\frac{\mbox{tr}(\Sigma_n+\epsilon I)^{-1}}{n(1-\frac{d}{n})(1-\frac{d_\epsilon^n}{n})}\hat\Sigma_n\right)\left(\frac{1}{1-\frac{d}{n}}\hat\Sigma_n+\epsilon I\right)^{-1}\right]\right\|\\
    &\leq \frac{\epsilon}{\left(\frac{\sigma}{1-\frac{d_\epsilon^n}{n}}+\epsilon\right)^2}\left\|\mathbb{E}\left[\frac{\mbox{tr}(\Sigma_n+\epsilon I)^{-1}}{n(1-\frac{d}{n})(1-\frac{d_\epsilon^n}{n})}\hat\Sigma_n\right]\right\|\\
    &\leq \frac{\epsilon}{\left(\frac{\sigma}{1-\frac{d_\epsilon^n}{n}}+\epsilon\right)^2}\cdot \frac{\sigma_{\max}d}{n(1-\frac{d}{n})^2(\sigma_{\min}+\epsilon)}
    \end{aligned}
\end{equation*}
From Theorem \ref{thm2}, we know
\begin{equation*}
    \begin{aligned}
    \mathbb{E}\left[\left(\frac{1}{1-\frac{d}{n}}\hat\Sigma_n+\epsilon I\right)^{-1}\right] = (\Sigma_n+\epsilon I)^{-1}+ \mathbb{E}\left[\left(\frac{1}{1-\frac{d}{n}}\hat\Sigma_n+\epsilon I\right)^{-1}\right]-\mathbb{E}\left[\left(\frac{1}{1-\frac{d_\epsilon^n}{n}}\hat\Sigma+\epsilon I\right)^{-1}\right]+\Omega_0
    \end{aligned}
\end{equation*}   
Therefore,
\begin{equation*}
    \begin{aligned}
\lim_{n,d\rightarrow\infty, \epsilon\rightarrow 0} \left\|\mathbb{E}\left[\left(\frac{1}{1-\frac{d}{n}}\hat\Sigma_n+\epsilon I\right)^{-1}\right]-(\Sigma_n+\epsilon I)^{-1}\right\|=0.
    \end{aligned}
\end{equation*}     
Furthermore, if $\Sigma_n^{-1}$ exists for each $n$, by resolvent identity,
\begin{equation*}
    \begin{aligned}
    \|(\Sigma_n+\epsilon I)^{-1}-\Sigma_n^{-1}\|=\|\epsilon(\Sigma_n+\epsilon I)^{-1}\Sigma_n^{-1}\|\leq \frac{\epsilon}{\sigma_{\min}^2}
    \end{aligned}
\end{equation*}     
Thus,
\begin{equation*}
    \begin{aligned}
    \mathbb{E}\left[\left(\frac{1}{1-\frac{d}{n}}\hat\Sigma_n+\epsilon I\right)^{-1}\right] = \Sigma_n^{-1}+((\Sigma_n+\epsilon I)^{-1}-\Sigma_n^{-1})+ \mathbb{E}\left[\left(\frac{1}{1-\frac{d}{n}}\hat\Sigma_n+\right.\right.\\\left.\left.\epsilon I\right)^{-1}\right]-\mathbb{E}\left[\left(\frac{1}{1-\frac{d_\epsilon^n}{n}}\hat\Sigma_n+\epsilon I\right)^{-1}\right]+\Omega_0
    \end{aligned}
\end{equation*}  
and
\[\lim_{n,d\rightarrow\infty, \epsilon\rightarrow 0} \left\|\mathbb{E}\left[\left(\frac{1}{1-\frac{d}{n}}\hat\Sigma_n+\epsilon I\right)^{-1}\right]-\Sigma_n^{-1}\right\|=0\]
\end{proof}

\subsection{Proof of Theorem \ref{thm2.6}}\label{thm2.6proof}

\begin{proof}
Built on equation (\ref{eq1}) in proof of Lemma \ref{thm1}, for any $z<0,z\in\mathbb{R}$,
\[(I-z\Sigma_n)\mathbb{E}\left[\left(I-z\hat\Sigma_n\right)^{-1}\right]=\left(z s(z)\Sigma_n-z\Sigma_n\right)\mathbb{E}\left[\left(I-z\hat\Sigma_n\right)^{-1}\right]+I+\Omega+\Omega'\]
Take $z=-\frac{1}{\lambda}$, note $z<0$, and thus
\begin{equation*}
\mathbb{E}\left[\left(\hat\Sigma_n+\lambda I\right)^{-1}\right]=\left(\Sigma_n+\lambda I\right)^{-1}+(\Sigma_n+\lambda I)^{-1}\left[\left(1-s\left(-\frac{1}{\lambda}\right)\right)\Sigma_n \mathbb{E}\left[(\hat\Sigma_n+\lambda I)^{-1}\right]+\Omega+\Omega'\right]
\end{equation*}
Therefore,
\begin{equation}\label{avg_lwb_finite}
\begin{aligned}
   \left\|\mathbb{E}\left[\left(\hat\Sigma_n+\lambda I\right)^{-1}\right]-\left(\Sigma_n+\lambda I\right)^{-1}\right\|&=\left\|(\Sigma_n+\lambda I)^{-1}\left(1-s\left(-\frac{1}{\lambda}\right)\right)\Sigma_n\mathbb{E}\left[(\hat \Sigma_n+\lambda I)^{-1}\right]+\Omega_2\right\|\\
   &\geq \frac{1}{\sigma_{\max}+\lambda} \left\|\left(1-s\left(-\frac{1}{\lambda}\right)\right)\Sigma_n \mathbb{E}\left[\left(\hat\Sigma_n+\lambda I\right)^{-1}\right]\right\|-\|\Omega_2\|
\end{aligned}
\end{equation}
Since $\left(1-s\left(-\frac{1}{\lambda}\right)\right)=y-\lambda \lim_{n,d\rightarrow\infty}\mathbb{E}\left[\frac{1}{n}\mbox{tr}\left(\hat\Sigma_n+\lambda I\right)^{-1}\right]$, with $\lambda'$ satisfies $\lambda=\lambda'\left(1-\overline{d_{\lambda'}}\right)$, \footnote{There exists $\lambda'$ satisfying the equation for any $\lambda>0.$ To see this, define $f:\mathbb{R}_+\rightarrow \mathbb{R}_+^\ast$ by $f(\lambda')=\lambda'(1-\overline{d_{\lambda'}})$, note $\lim_{\lambda'\rightarrow 0} f(\lambda')=0$ and  $\lim_{\lambda'\rightarrow +\infty} f(\lambda')=+\infty$, and $f(\lambda')$ is continuous.}
\begin{equation*}
\begin{aligned}
    &\left\|\mathbb{E}\left[\left(\hat\Sigma_n+\lambda I\right)^{-1}\right]-\left(\Sigma_n+\lambda I\right)^{-1}\right\| \\
    & \geq \frac{1}{\sigma_{\max}+\lambda} \left(\left|y-\lambda \lim_{n,d\rightarrow\infty}\mathbb{E}\left[\frac{1}{n}\mbox{tr}\left(\hat\Sigma_n+\lambda I\right)^{-1}\right]\right|\left\|\Sigma_n \mathbb{E}\left[\left(\hat\Sigma_n+\lambda I\right)^{-1}\right]\right\|\right)-\|\Omega_2\|\\
    &\geq \frac{\sigma_{\min}}{\sigma_{\max}+\lambda}\left|y-\lambda \lim_{n,d\rightarrow\infty}\mathbb{E}\left[\frac{1}{n}\mbox{tr}\left(\hat\Sigma_n+\lambda I\right)^{-1}\right]\right|\left\|\mathbb{E}\left[\left(\hat\Sigma_n+\lambda I\right)^{-1}\right]\right\|-\|\Omega_2\|\\
    &\geq \frac{\sigma_{\min}}{\sigma_{\max}+\lambda}\left|y-\lambda \lim_{n,d\rightarrow\infty}\mathbb{E}\left[\frac{1}{n}\mbox{tr}\left(\hat\Sigma_n+\lambda I\right)^{-1}\right]\right|\left\|\frac{\lambda'}{\lambda}\left(\Sigma_n+\lambda' I\right)^{-1}-\left(\frac{\lambda'}{\lambda}(\Sigma_n+\lambda' I)^{-1}-\mathbb{E}\left[\left(\hat\Sigma_n+\lambda I\right)^{-1}\right]\right)\right\|\\&-\|\Omega_2\|\\
    &\geq \frac{\sigma_{\min}}{\sigma_{\max}+\lambda}\left|y-\lambda \lim_{n,d\rightarrow\infty}\mathbb{E}\left[\frac{1}{n}\mbox{tr}\left(\hat\Sigma_n+\lambda I\right)^{-1}\right]\right|\left\|\frac{\lambda'}{\lambda}(\Sigma_n+\lambda' I)^{-1}\right\|\\
    &-\frac{\sigma_{\min}}{\sigma_{\max}+\lambda}\left|y-\lambda \lim_{n,d\rightarrow\infty}\mathbb{E}\left[\frac{1}{n}\mbox{tr}\left(\hat\Sigma_n+\lambda I\right)^{-1}\right]\right|\left\|\frac{\lambda'}{\lambda}(\Sigma_n+\lambda' I)^{-1}-\mathbb{E}\left[\left(\hat\Sigma_n+\lambda I\right)^{-1}\right]\right\|\\
    &-\|\Omega_2\|
\end{aligned}
\end{equation*}
By Lemma \ref{thm1}, $\lim_{n,d\rightarrow\infty}\|\Omega_2\|=0$, and also
\[\lim_{n,d\rightarrow\infty}\left\|\mathbb{E}\left[\left(\hat\Sigma_n+\lambda I\right)^{-1}\right]-\frac{\lambda'}{\lambda}\left(\Sigma_n+\lambda' I\right)^{-1}\right\|=0\]
thus
\[\lim_{n,d\rightarrow\infty} \mathbb{E}\left[\frac{1}{n}\mbox{tr}\left(\hat\Sigma_n+\lambda I\right)^{-1}\right]=\frac{\lambda'}{\lambda}\lim_{n,d\rightarrow\infty} \frac{1}{n}\mbox{tr}\left(\Sigma_n+\lambda' I\right)^{-1}\]
Therefore, since $\left|y-\lambda \lim_{n,d\rightarrow\infty}\mathbb{E}\left[\frac{1}{n}\mbox{tr}\left(\hat\Sigma_n+\lambda I\right)^{-1}\right]\right| \leq y,$ denote the Stieltjes transform as $m_\mu(z)=\int_R \frac{1}{x-z}\mu(dx)$,
\begin{equation}\label{avg_lwb}
\begin{aligned}
\lim_{n,d\rightarrow\infty} \left\|\mathbb{E}\left[\left(\hat\Sigma_n+\lambda I\right)^{-1}\right]-\left(\Sigma_n+\lambda I\right)^{-1}\right\|&\geq \frac{\sigma_{\min}}{\sigma_{\max}+\lambda}\cdot \frac{\lambda'}{\lambda(\sigma_{\max}+\lambda')}\left|y-\lambda' \lim_{n,d\rightarrow\infty} \frac{1}{n}\mbox{tr}\left(\Sigma_n+\lambda' I\right)^{-1}\right|\\
&= \frac{\lambda' \sigma_{\min}}{\lambda(\sigma_{\max}+\lambda)(\sigma_{\max}+\lambda')}\left|y-\lambda' y m_{F_{\Sigma}}(-\lambda')\right|\\
& \geq \frac{\lambda' \sigma_{\min} y}{\lambda(\sigma_{\max}+\lambda)(\sigma_{\max}+\lambda')}\left|\frac{\sigma_{\min}}{\sigma_{\min}+\lambda'}\right|
\end{aligned}
\end{equation}
When $\lambda\in o(1)$, since $1-\overline{d_{\lambda'}}\geq 1-y>0$, thus $\lambda'\in o(1)$ and therefore as $n,d\rightarrow \infty$,
\begin{equation*}
\frac{\lambda' \sigma_{\min} y}{\lambda(\sigma_{\max}+\lambda)(\sigma_{\max}+\lambda')}\left|\frac{\sigma_{\min}}{\sigma_{\min}+\lambda'}\right|\rightarrow \frac{\lambda'\sigma_{\min} y}{\lambda \sigma_{\max}^2}=\frac{\sigma_{\min} y}{(1-\overline{d_{\lambda'}})\sigma_{\max}^2}\geq \frac{\sigma_{\min} y}{\sigma_{\max}^2}
\end{equation*}
\end{proof}
\section{Proofs in Section \ref{section3}}

\subsection{Technical Lemmas}
\begin{lemma}\label{lemmab.1}
If data matrix $A$ is random and satisfies Assumption \ref{ass_1}, with $A_i$ i.i.d. mean zero and covariance $\Sigma$.  Denote $\tilde p_t$ as the result point of preconditioned conjugate gradient method, $p_t^\star$ as the true Newton's step. Then,
\[\frac{\|\tilde p_t-p_t^\star\|_H^2}{\|p_t^\star\|_H^2}\leq 4\left(\frac{1-\sqrt{1-\frac{\alpha}{1-\alpha}}}{1+\sqrt{1-\frac{\alpha}{1-\alpha}}}\right)^{s_t}\]
where $\alpha=(\sigma_{\max}+\lambda+\alpha_0)\left(\frac{1}{\lambda^2}\alpha_0+\alpha_1+\|\Omega_0\|\right)$ with $\alpha_0=\|\Sigma-\frac{1}{n}A^TA\|,$  $\alpha_1=\|\tilde H^{-1}-\mathbb{E}[\tilde H^{-1}]\|$ and $\sigma_{\max}$ being the largest eigenvalues of $\Sigma$. $s_t$ denotes the number of iterations in preconditioned conjugate gradient method.
\end{lemma}
\begin{proof}
 since 
 \begin{equation}\label{lemmab.1eq0}
\begin{aligned}
&\left\|\tilde H^{-1}-H^{-1}\right\|\\&=\left\|\tilde H^{-1}-\mathbb{E}\left[\tilde H^{-1}\right]+\mathbb{E}\left[\tilde H^{-1}\right]-(\Sigma+\lambda I)^{-1}+(\Sigma+\lambda I)^{-1}-H^{-1}\right\|\\
&\leq \frac{1}{\lambda^2}\alpha_0+\alpha_1+\|\Omega_0\|
\end{aligned}
\end{equation}
thus
\begin{equation*}
\begin{aligned}
\left\|H^{\frac{1}{2}}\tilde H^{-1}H^{\frac{1}{2}}-I\right\|&=\left\|H^{\frac{1}{2}}\left(\tilde H^{-1}-H^{-1}\right)H^{\frac{1}{2}}\right\|\\
&\leq \|H\|\left\|\tilde H^{-1}-H^{-1}\right\|\\
&\leq (\sigma_{\max}+\lambda+\alpha_0)\left(\frac{1}{\lambda^2}\alpha_0+\alpha_1+\|\Omega_0\|\right)=\alpha
\end{aligned}
\end{equation*}
Thus,
\begin{equation*}
\begin{aligned}
\left\|H^{-\frac{1}{2}}\tilde H H^{-\frac{1}{2}}-I\right\|\leq \frac{\alpha}{1-\alpha}
\end{aligned}
\end{equation*}
By equation (3.3) in \cite{lacotte2021FastCQ}, set the initial point in the preconditioned conjugate gradient method at 0,
\[\frac{\|\tilde p_t-p_t^\star\|_H^2}{\|p_t^\star\|_H^2}\leq 4\left(\frac{1-\sqrt{1-\frac{\alpha}{1-\alpha}}}{1+\sqrt{1-\frac{\alpha}{1-\alpha}}}\right)^{s_t}\]
\end{proof}
\begin{lemma}\label{lemmab.2}
 If data matrix $A$ is random and satisfies Assumption \ref{ass_1}, with $A_i$ i.i.d. mean zero, assume $\omega_t$ independent of all $A_j^{(i)}$'s, denote $\Sigma(\omega_t) = \mbox{Cov}\left(f_i''\left(\omega_t^TA_j^{(i)}\right)^{\frac{1}{2}}A_j^{(i)}\right)$,
 
\[\frac{\|\tilde p_t-p_t^\star\|_{H(\omega_t)}^2}{\|p_t^\star\|_{H(\omega_t)}^2}\leq 4\left(\frac{1-\sqrt{1-\frac{\alpha}{1-\alpha}}}{1+\sqrt{1-\frac{\alpha}{1-\alpha}}}\right)^{s_t}\]
where $\alpha=(\sigma_{\max}+\lambda+\alpha_0)\left(\frac{1}{\lambda^2}\alpha_0+\alpha_1+\|\Omega_0\|\right)$ with $\alpha_0=\left\|\Sigma(\omega_t)-\frac{1}{n}\sum_{i=1}^n f_i''\left(\omega_t^TA_i\right)A_iA_i^T\right\|,$  $\alpha_1=\left\|\tilde H(\omega_t)^{-1}-\mathbb{E}\left[\tilde H(\omega_t)^{-1}\right]\right\|$ and $\sigma_{\max}$ being the largest eigenvalues of $\Sigma(\omega_t)$. $s_t$ denotes the number of iterations in preconditioned conjugate gradient method.
\end{lemma}
\begin{proof}
Follow the same argument as in Lemma \ref{lemmab.1} with $H$ replaced by $H(\omega_t),$ $\tilde H$ replaced by $\tilde H(\omega_t)$, and $\Sigma$ replaced by $\Sigma(\omega_t)$.
\end{proof}

\subsection{Proof of Theorem \ref{thm3.1}} \label{thm3.1proof}

\begin{proof}
We build on proof in \citet{dereziński2019distributed}. Consider the $t$th Newton's step. Denote $p_t^\star=H^{-1}g_t$ and $\tilde p_t=\tilde H^{-1}g_t$ where $g_t$ is the gradient and the current Newton's point, by Lemma 14 in \citet{dereziński2019distributed}, if
\[\|\tilde p_t-p_t^\star\|_H\leq \alpha \|p_t^\star\|_H\]
for some $\alpha$, denote $\kappa=\mbox{cond}(H)$, then it holds
\[\|\Delta_{t+1}\|\leq \frac{\alpha\sqrt{2\kappa}}{\sqrt{1-\alpha^2}}\|\Delta_t\|\]
Since
\begin{equation*}
\begin{aligned}
\|\tilde p_t-p_t^\star\|_H&=\left\|H^{\frac{1}{2}}\left(\tilde H^{-1}-H^{-1}\right)H^{\frac{1}{2}}H^{-\frac{1}{2}}g_t\right\|\\&\leq \|H\|\left\|\tilde H^{-1}-H^{-1}\right\|\left\|H^{-\frac{1}{2}}g_t\right\|
\\&=\|H\|\left\|\tilde H^{-1}-H^{-1}\right\|\|p_t^\star\|_H
\end{aligned}
\end{equation*}
and
\[\sigma_{\max}(H)\leq \sigma_{\max}+\lambda+\alpha_0,\sigma_{\min}(H)\geq \sigma_{\min}+\lambda-\alpha_0\]
together with (\ref{lemmab.1eq0}), we can derive
\begin{equation*}
\begin{aligned}
\|\Delta_{t+1}\|
&\leq \frac{\sqrt{2}\alpha}{\sqrt{1-\alpha^2}}\sqrt{\frac{\sigma_{\max}+\lambda+\alpha_0}{\sigma_{\min}+\lambda-\alpha_0}}\|\Delta_t\|
\end{aligned}
\end{equation*}
where $\alpha=(\sigma_{\max}+\lambda+\alpha_0)\left(\frac{1}{\lambda^2}\alpha_0+\alpha_1+\|\Omega_0\|\right)$
\end{proof}

\subsection{Proof of Theorem \ref{thm3.3}}\label{thm3.3_proof}

\begin{proof}
Follow similar argument as in the convergence proof for Newton's method with quadratic loss (Appendix \ref{thm3.1proof}),  it can be derived that
\[\|H(\omega_t)^{-1}-\tilde H(\omega_t)^{-1}\|\leq \frac{1}{\lambda^2}\alpha_0+\alpha_1+\|\Omega_0\|\]
Let $\tilde p_t=\tilde H(\omega_t)^{-1}g(\omega_t)$ and $p^\star_t=H(\omega_t)^{-1}g(\omega_t)$. Since
\begin{equation*}
\begin{aligned}
\|\tilde p_t-p_t^\star\|_{H(\omega_t)}&=\left\|H(\omega_t)^{\frac{1}{2}}\left(\tilde H(\omega_t)^{-1}-H(\omega_t)^{-1}\right)H(\omega_t)^{\frac{1}{2}}H(\omega_t)^{-\frac{1}{2}}g(\omega_t)\right\|\\&\leq \|H(\omega_t)\|\left\|\tilde H(\omega_t)^{-1}-H(\omega_t)^{-1}\right\|\left\|H(\omega_t)^{-\frac{1}{2}}g(\omega_t)\right\|
\\&=\|H(\omega_t)\|\left\|\tilde H(\omega_t)^{-1}-H(\omega_t)^{-1}\right\|\|p^\star_t\|_{H(\omega_t)}
\\ &\leq (\sigma_{\max}+\lambda+\alpha_0)\left(\frac{1}{\lambda^2}\alpha_0+\alpha_1+\|\Omega_0\|\right)\|p^\star_t\|_{H(\omega_t)}
\end{aligned}
\end{equation*}
By Lemma 14 in \cite{dereziński2019distributed},
\[\|\Delta_{t+1}\|\leq \max\left\{\frac{\sqrt{2}\alpha}{\sqrt{1-\alpha^2}}\sqrt{\frac{\sigma_{\max}+\lambda+\alpha_0}{\sigma_{\min}+\lambda-\alpha_0}}\|\Delta_t\|,\frac{2L}{\sigma_{\min}+\lambda-\alpha_0}\|\Delta_t\|^2\right\}\]
where $\alpha=(\sigma_{\max}+\lambda+\alpha_0)\left(\frac{1}{\lambda^2}\alpha_0+\alpha_1+\|\Omega_0\|\right)$.
\end{proof}

\subsection{More Convergence Analysis Results}\label{conv-append}
\subsubsection{Convergence of inexact Newton's method for Regularized General Convex  Smooth 
 Loss}
\begin{theorem}\label{thm3.4}
 (Convergence of inexact Newton's method with Shrinkage) Assume $\omega_t$ independent of all $A_j^{(i)}$'s.  Let $\alpha, \Sigma(\omega_t),\alpha_0,\alpha_1, \Delta_{t+1},\sigma_{\min},\sigma_{\max}$ as defined in Theorem \ref{thm3.3},
\[\|\Delta_{t+1}\|\leq\max\left\{ \frac{2L}{\sigma_{\min}+\lambda-\alpha_0}\|\Delta_t\|^2,\frac{\sqrt{2}\alpha'}{\sqrt{1-\alpha'^2}}\sqrt{\frac{\sigma_{\max}+\lambda+\alpha_0}{\sigma_{\min}+\lambda-\alpha_0}}\|\Delta_t\|\right\}\]
where $\alpha'=\sqrt{4\left(\frac{1-\sqrt{1-\frac{\alpha}{1-\alpha}}}{1+\sqrt{1-\frac{\alpha}{1-\alpha}}}\right)^{s_t}}$ with $s_t$ being the number of iterations in preconditioned conjugate gradient method.
\end{theorem}
\begin{proof}
The proof follows by a combination of Lemma 14 in \cite{dereziński2019distributed} and Lemma \ref{lemmab.2}.
\end{proof}



\section{Supplementary Simulation Results} \label{append_simu}



\subsection{Supplementary Simulation for Section \ref{section5.1}}
\subsubsection{Formulas for Different Methods for Covariance Resolvent Estimation}\label{5.1formula}
Here we explicitly list the formulas used by different methods to compute $\tilde R$ which is used for the plots. Let $n$ denote the number of data, $m$ denote number of agents, create $n,m$ in the way that $n$ is divided by $m$ and split the data evenly to each agent, let $A\in\mathbb{R}^{n\times d}$ denote the global data and $A_i$ denote local data of size $\mathbb{R}^{(n/m)\times d}$, we compare three methods: 
\begin{align*}
&\mbox{Average: }\tilde R_a=\frac{1}{m}\sum_{i=1}^m R_{ai}\mbox{ where }R_{ai}=\left(\frac{m}{n}A_i^TA_i+\lambda I\right)^{-1}\\
&\mbox{Shrinkage: }\tilde R_s=\frac{1}{m}\sum_{i=1}^m R_{si}\mbox{ where }R_{si}=\left(\frac{m}{n(1-\frac{md_\lambda}{n})}A_i^TA_i+\lambda I\right)^{-1} \\
& \mbox{Determinant: }\tilde R_d= \frac{1}{m\mbox{det}(\Sigma+\lambda I)}\sum_{i=1}^m R_{di}\mbox{ where }R_{di}=\mbox{det}\left(\frac{m}{n}A_i^TA_i+\lambda I\right)\left(\frac{m}{n}A_i^TA_i+\lambda I\right)^{-1}
\end{align*}
where $d_\lambda=\mbox{tr}\left(\Sigma(\Sigma+\lambda I)^{-1}\right)$.
\subsubsection{Experiments on Covariance Resolvent Estimation with Synthetic data}\label{syn-simu-1}
Here we give simulation results for covariance resolvent estimation with synthetic data. See Section \ref{section5.1} for a description of the setup and different methods being compared. Figure \ref{syn-cov-simu} shows that our shrinkage method gives more accurate estimation of covariance resolvent than both the averaging method and the determinantal averaging method for synthetic data created with difference covariance matrices. 
\begin{figure}[h]
\centering
\begin{subfigure}{0.23\linewidth}
\includegraphics[width=\linewidth]{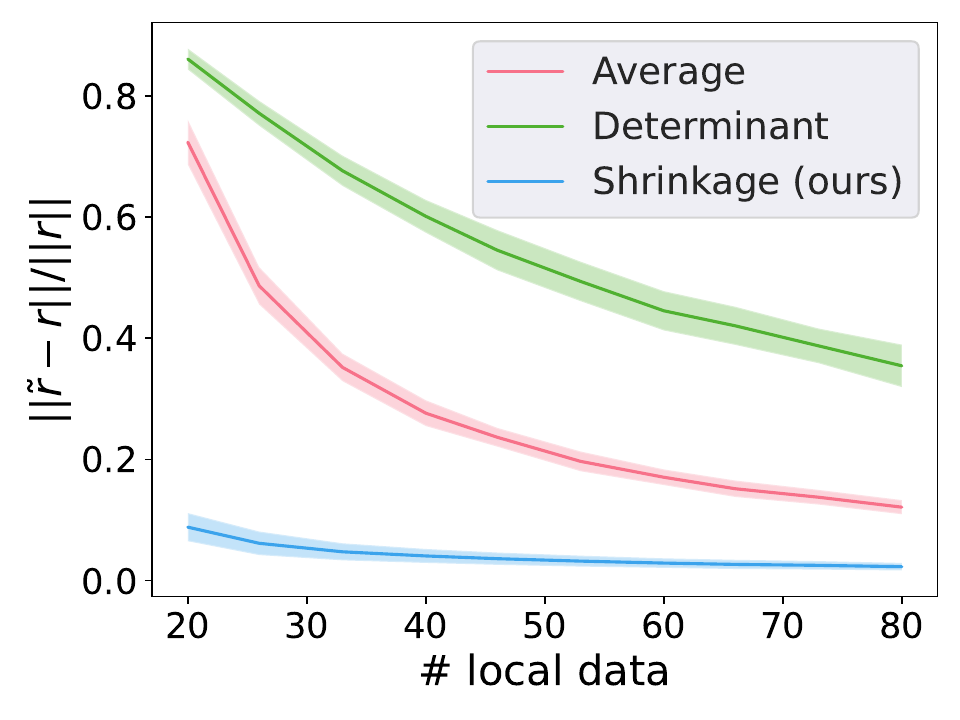}
\end{subfigure}
\begin{subfigure}{0.23\linewidth}
\includegraphics[width=\linewidth]{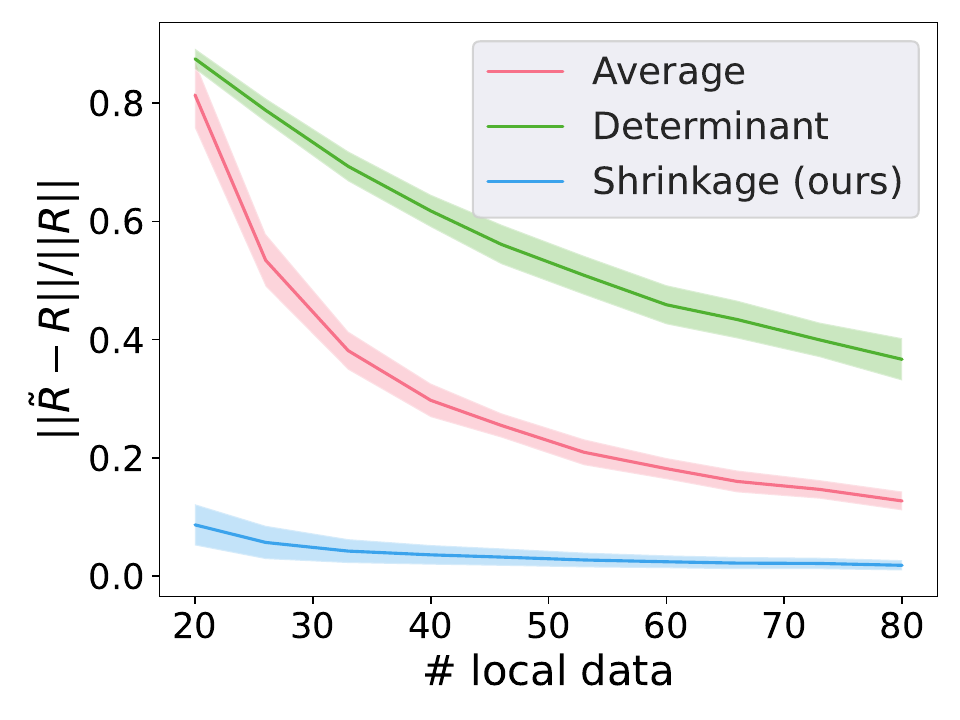}
\end{subfigure}
\caption{Synthetic data experiments on covariance resolvent estimation. Let $m$ denote the number of agents, $d$ denote data dimension, and $\lambda$ denote the regularizer. We take $m=100, d=10, \lambda=0.1$. Data is i.i.d. $\mathcal{N}(0,\Sigma)$ with $\Sigma=0.1I$ in the left plot, and $\Sigma=100C^TC, C_{ij}\sim U(0,1)$ in the right plot.}\label{syn-cov-simu}
\end{figure}

\subsubsection{More Experiments on Covariance Resolvent Estimation with Normalized Data}\label{sec5.1.2append}
Here we give more simulation results for covariance resolvent estimation with normalized real datasets. See Section \ref{section5.1} for a description of the setup and different methods being compared. Figure \ref{fig6} shows our simulation results, which confirms the shrinkage method's superiority over the averaging method and the determinantal averaging method in covariance resolvent estimation for normalized real data.
\begin{figure}[h]
\begin{subfigure}{0.23\linewidth}
\includegraphics[width=\linewidth]{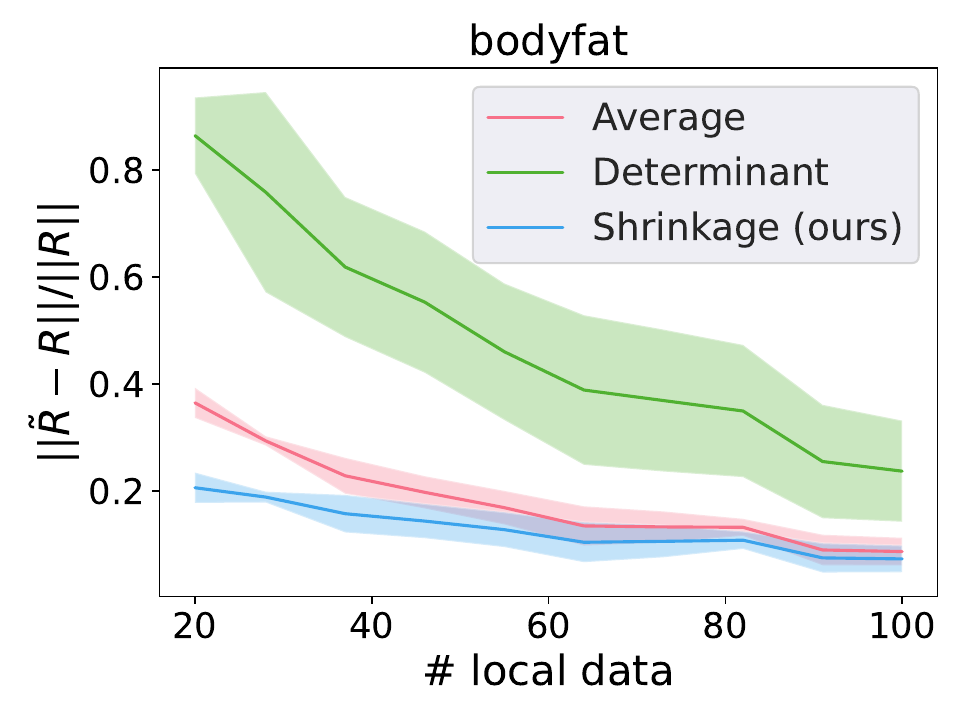}
\end{subfigure}
\hfill
\begin{subfigure}{0.23\linewidth}
\includegraphics[width=\linewidth]{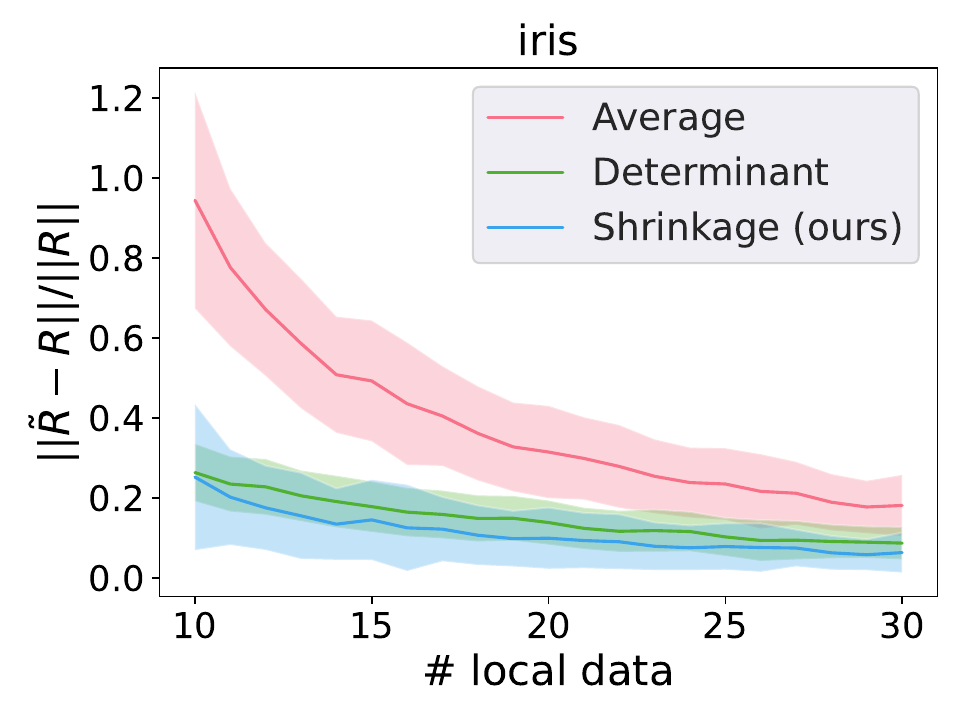}
\end{subfigure}
\hfill
\begin{subfigure}{0.23\linewidth}
\includegraphics[width=\linewidth]{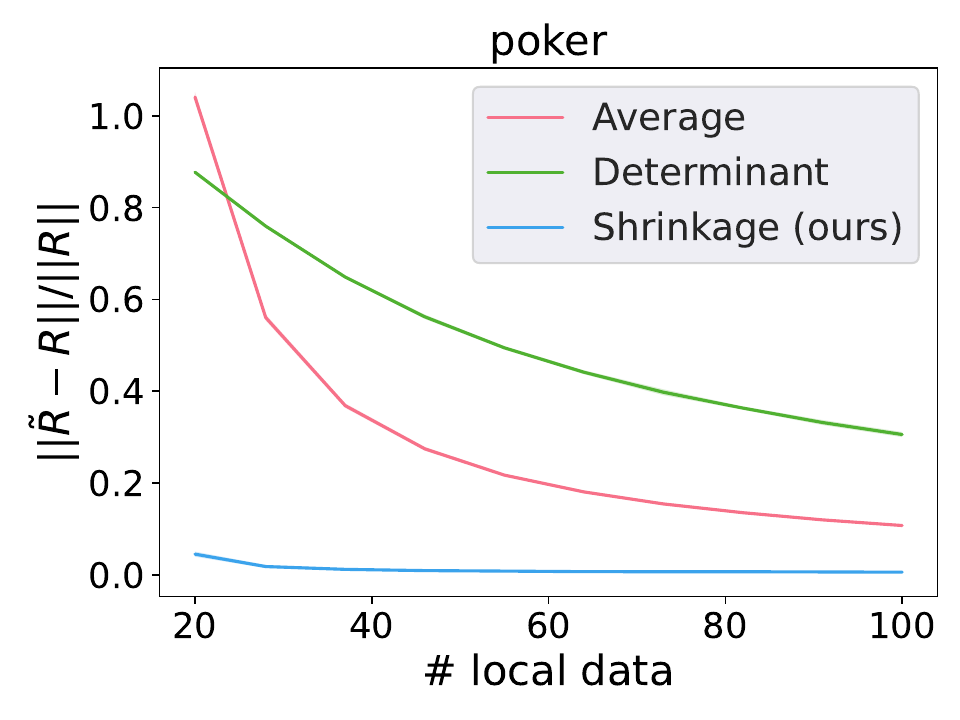}
\end{subfigure}
\hfill
\begin{subfigure}{0.23\linewidth}
\includegraphics[width=\linewidth]{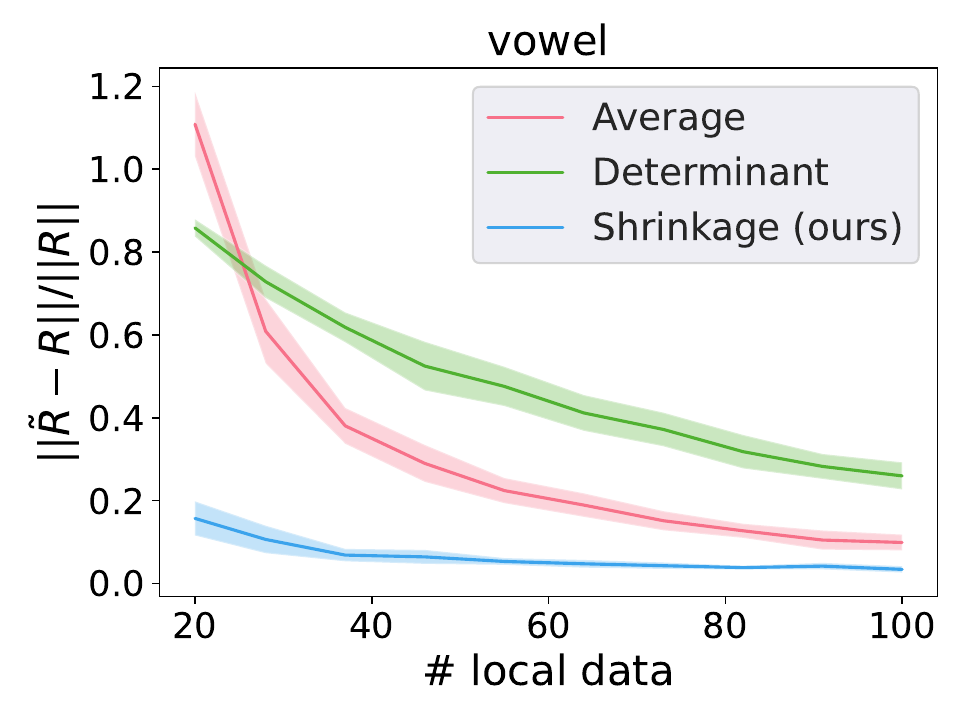}
\end{subfigure}

\begin{subfigure}{0.23\linewidth}
\includegraphics[width=\linewidth]{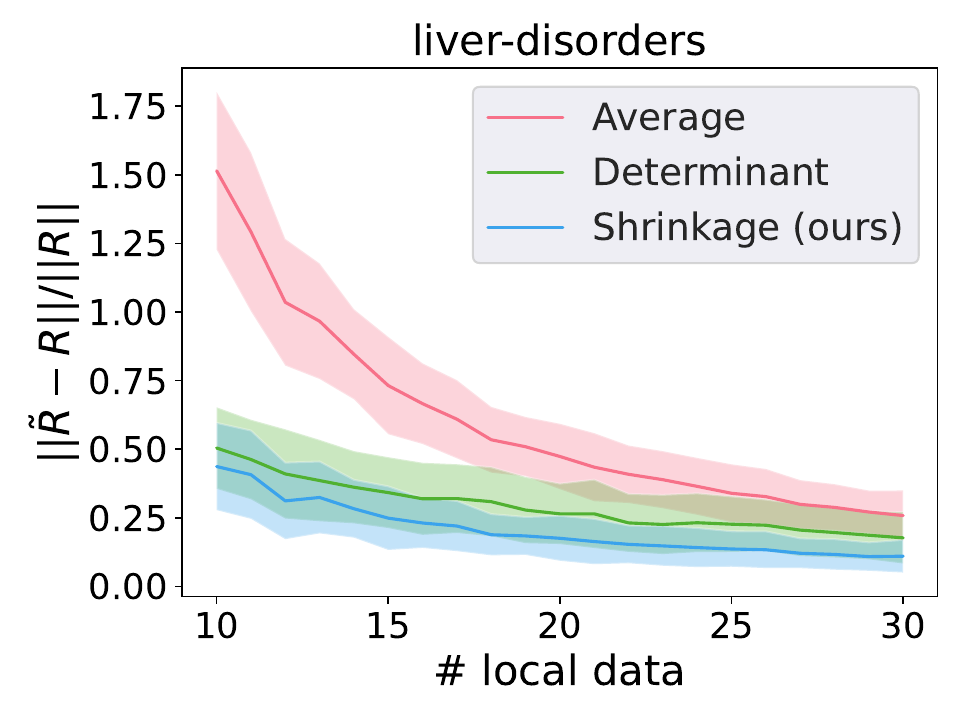}
\end{subfigure}
\hfill
\begin{subfigure}{0.23\linewidth}
\includegraphics[width=\linewidth]{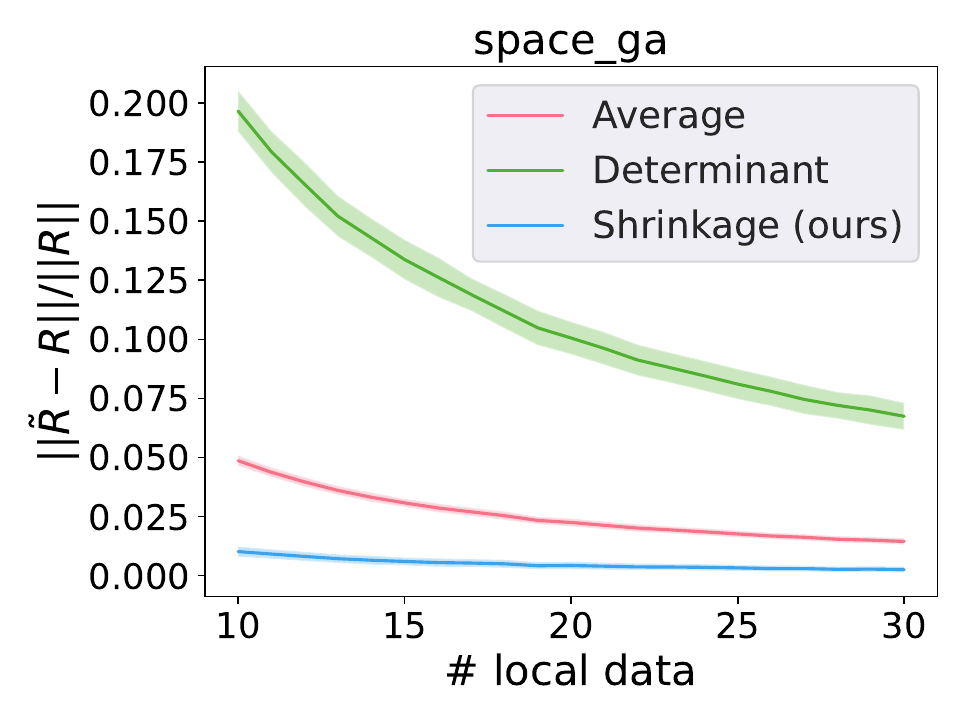}
\end{subfigure}
\hfill
\begin{subfigure}{0.23\linewidth}
\includegraphics[width=\linewidth]{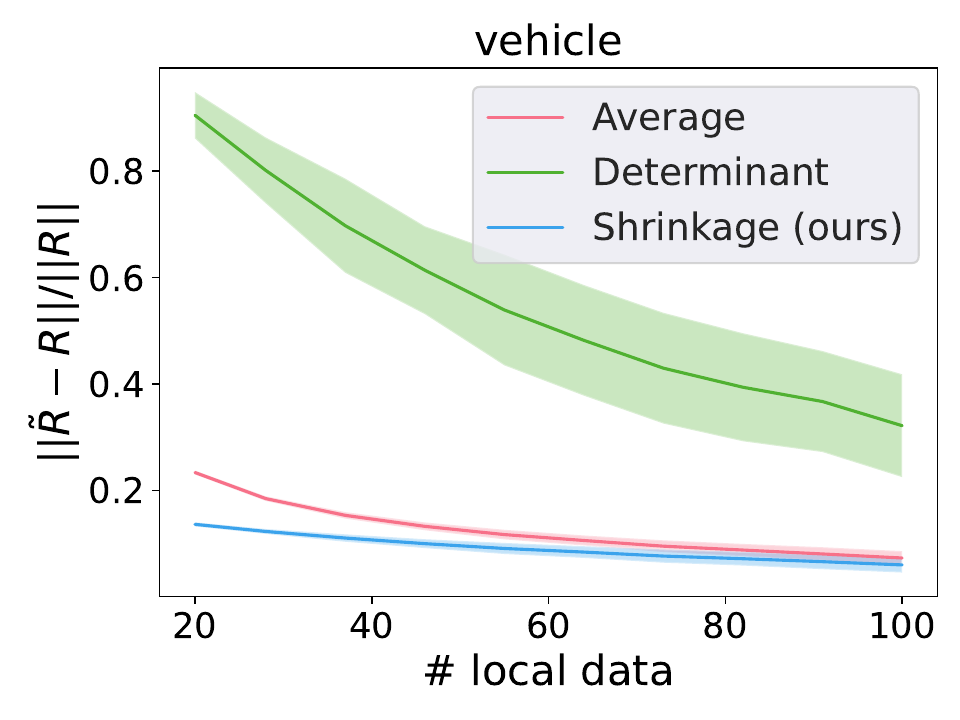}
\end{subfigure}
\hfill
\begin{subfigure}{0.23\linewidth}
\includegraphics[width=\linewidth]{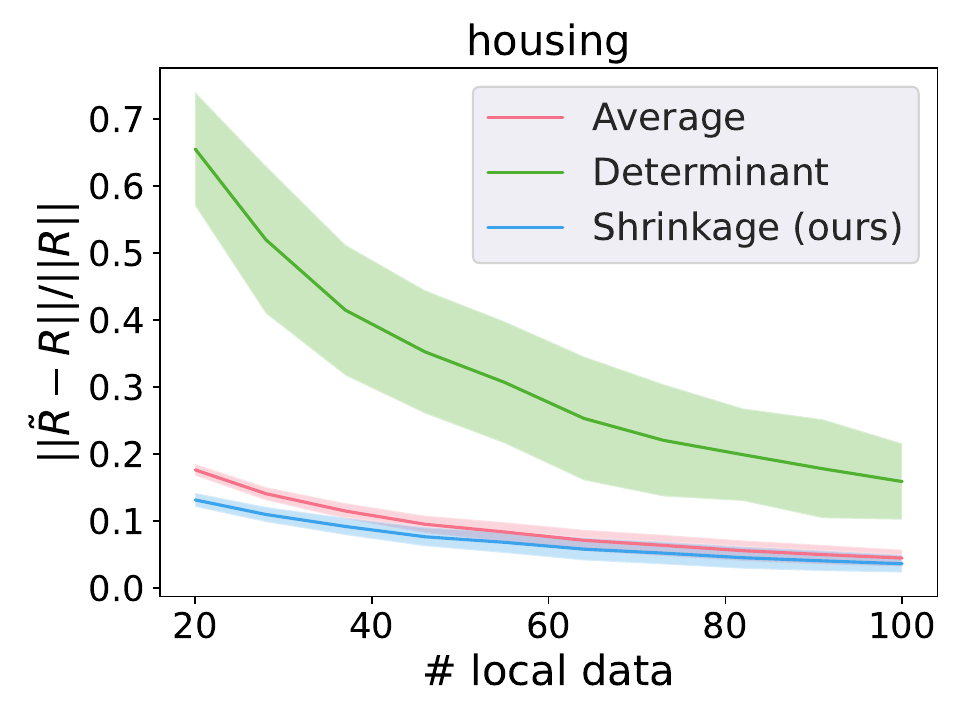}
\end{subfigure}
\caption{Normalized real data experiments on covariance resolvent estimation. $\lambda=0.001$. Number of total data is rounded down as multiple of number of local data, number of agent is number of total data divided by number of local data. We use $\Sigma=\frac{1}{n}A^TA$.}\label{fig6}
\end{figure}

\subsubsection{Experiments on Covariance Resolvent Estimation with Sketched Real Data}\label{sec5.1.3append}
Here we give simulation results for covariance resolvent estimation with sketched real datasets. See Section \ref{section5.1} for a description of the setup and different methods being compared. Take a real dataset,  we experiment with a sketch of it. See Section \ref{section4} for an introduction to data sketching. We test with sketching matrix with both gaussian entries and uniform entries and they give similar results as what can be witnessed from Figure \ref{fig7} and Figure \ref{fig8}. Figure \ref{fig7} and Figure \ref{fig8} demonstrate that the advantage of our method is not constrained to a specific data distribution.
\begin{figure}[h]
\begin{subfigure}{0.23\linewidth}
\includegraphics[width=\linewidth]{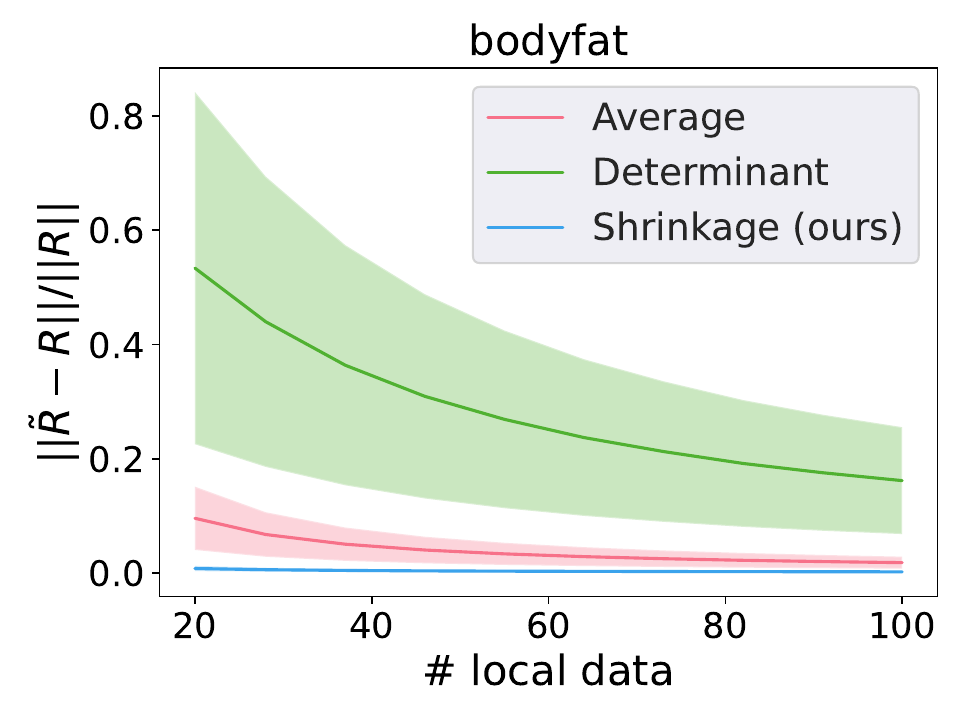}
\end{subfigure}
\hfill
\begin{subfigure}{0.23\linewidth}
\includegraphics[width=\linewidth]{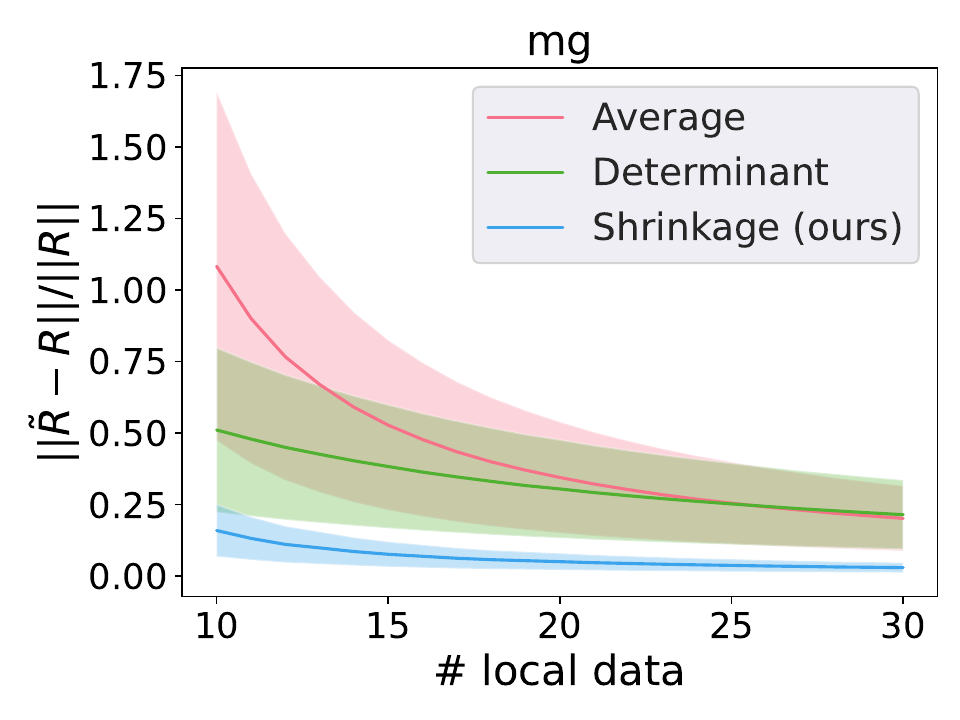}
\end{subfigure}
\hfill
\begin{subfigure}{0.23\linewidth}
\includegraphics[width=\linewidth]{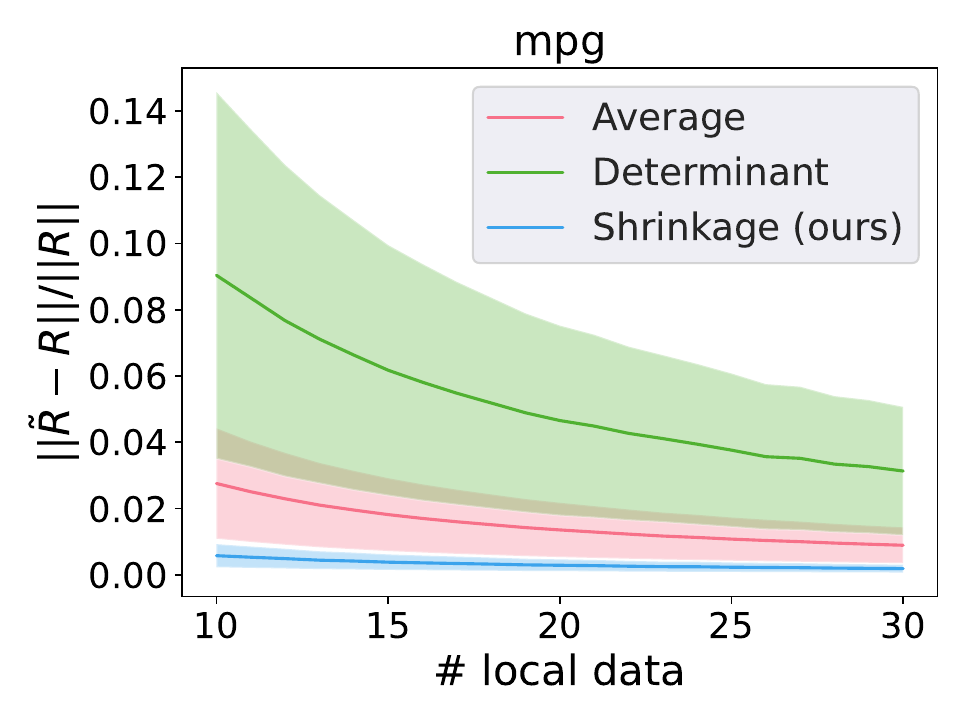}
\end{subfigure}
\hfill
\begin{subfigure}{0.23\linewidth}
\includegraphics[width=\linewidth]{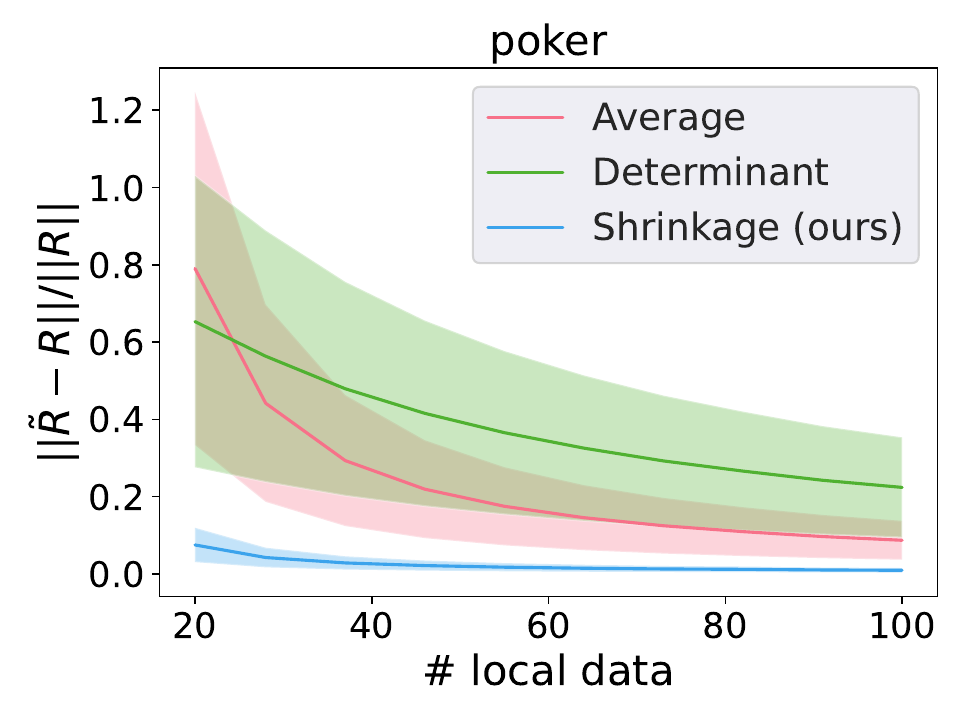}
\end{subfigure}

\begin{subfigure}{0.23\linewidth}
\includegraphics[width=\linewidth]{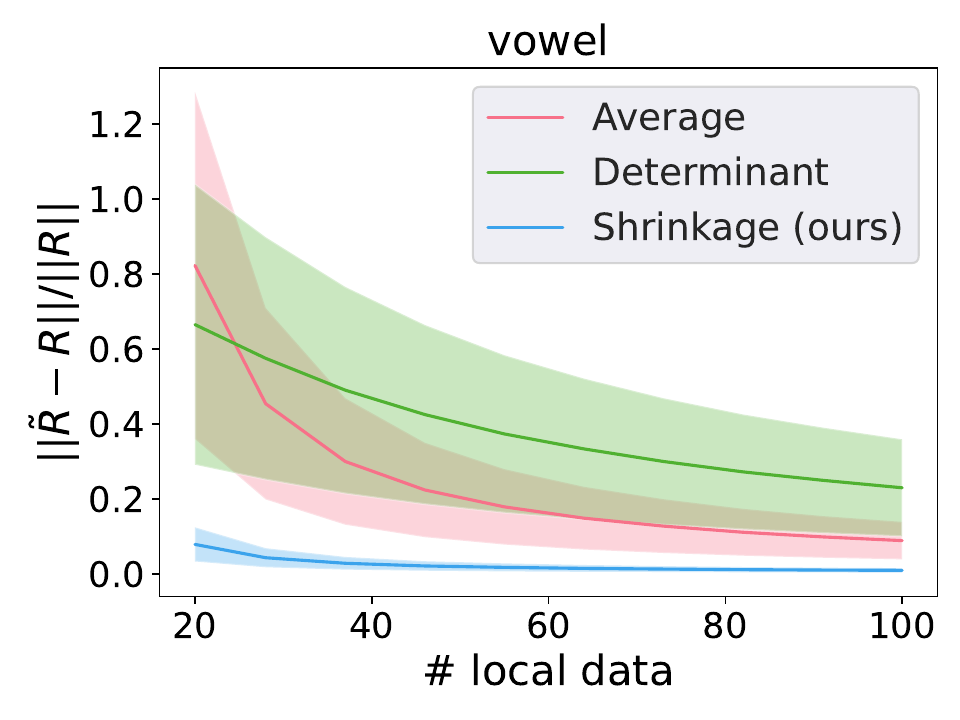}
\end{subfigure}
\hfill
\begin{subfigure}{0.23\linewidth}
\includegraphics[width=\linewidth]{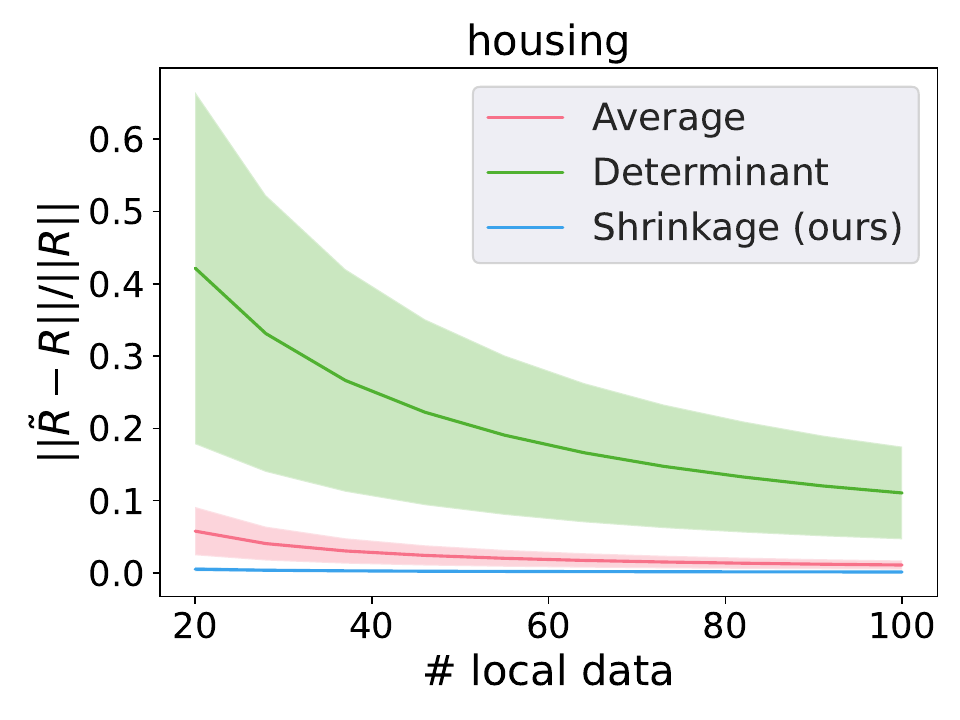}
\end{subfigure}
\hfill
\begin{subfigure}{0.23\linewidth}
\includegraphics[width=\linewidth]{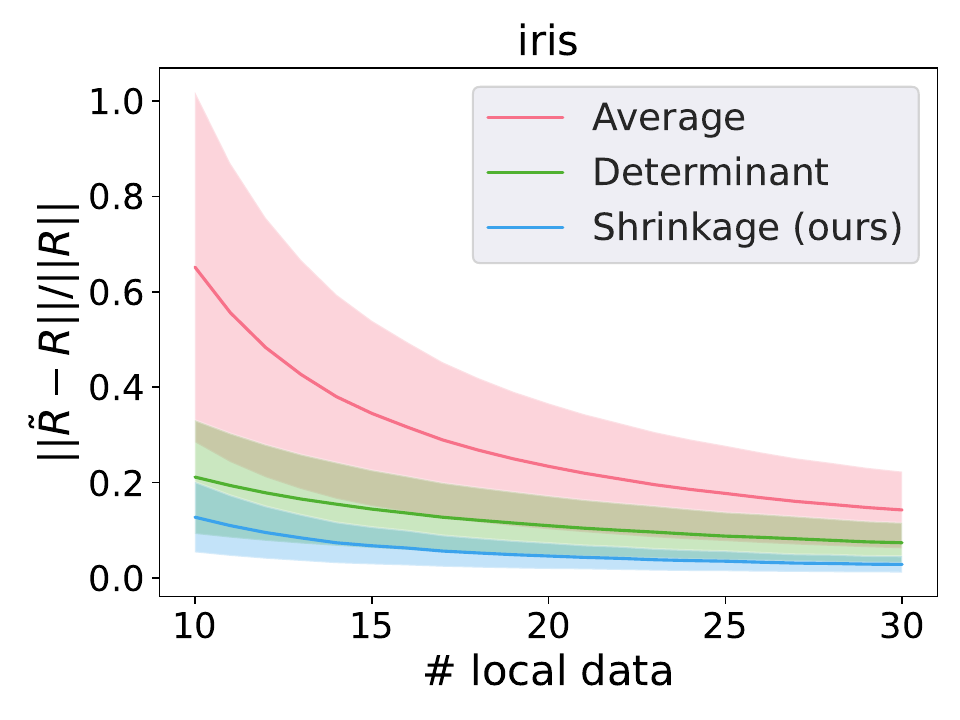}
\end{subfigure}
\hfill
\begin{subfigure}{0.23\linewidth}
\includegraphics[width=\linewidth]{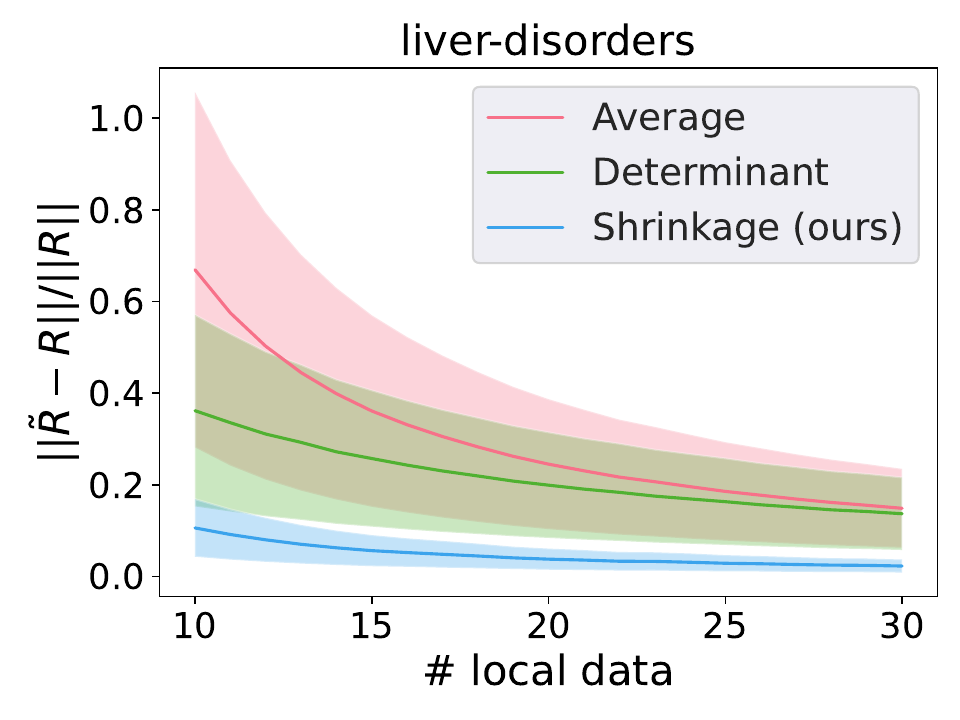}
\end{subfigure}

\caption{Sketched real data experiments on covariance resolvent estimation. We use sketching matrix with entry i.i.d. $\mathcal{N}(0,\frac{1}{m})$, sketch size $m=10000$, regularizer $\lambda=0.001$. We use $\Sigma=\frac{1}{n}A^TA$.}\label{fig7}
\end{figure}
\begin{figure}[h]
\begin{subfigure}{0.23\linewidth}
\includegraphics[width=\linewidth]{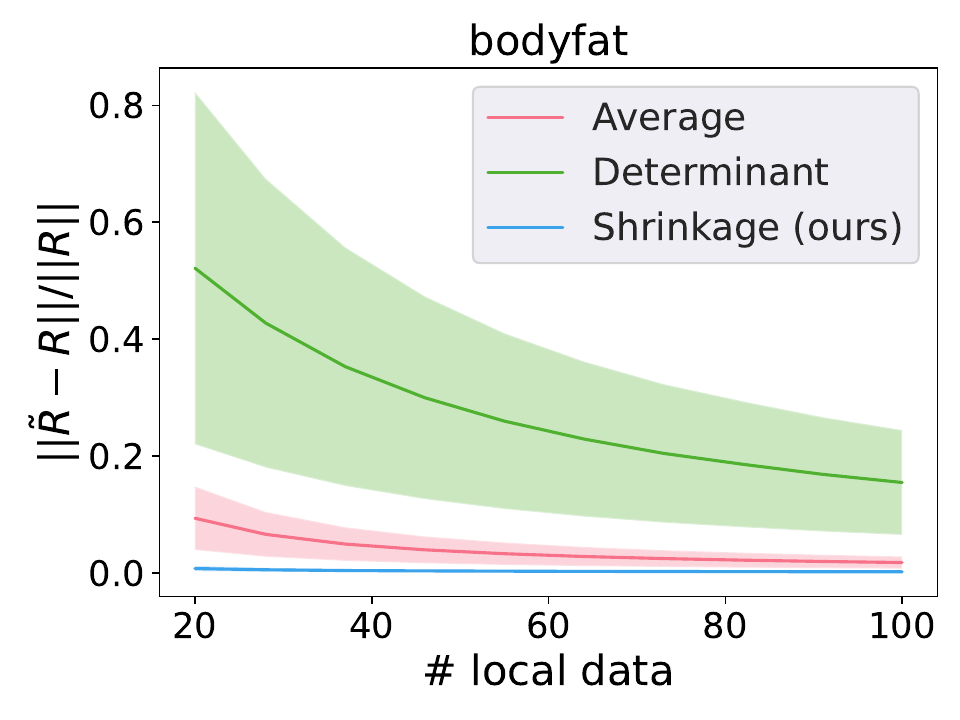}
\end{subfigure}
\hfill
\begin{subfigure}{0.23\linewidth}
\includegraphics[width=\linewidth]{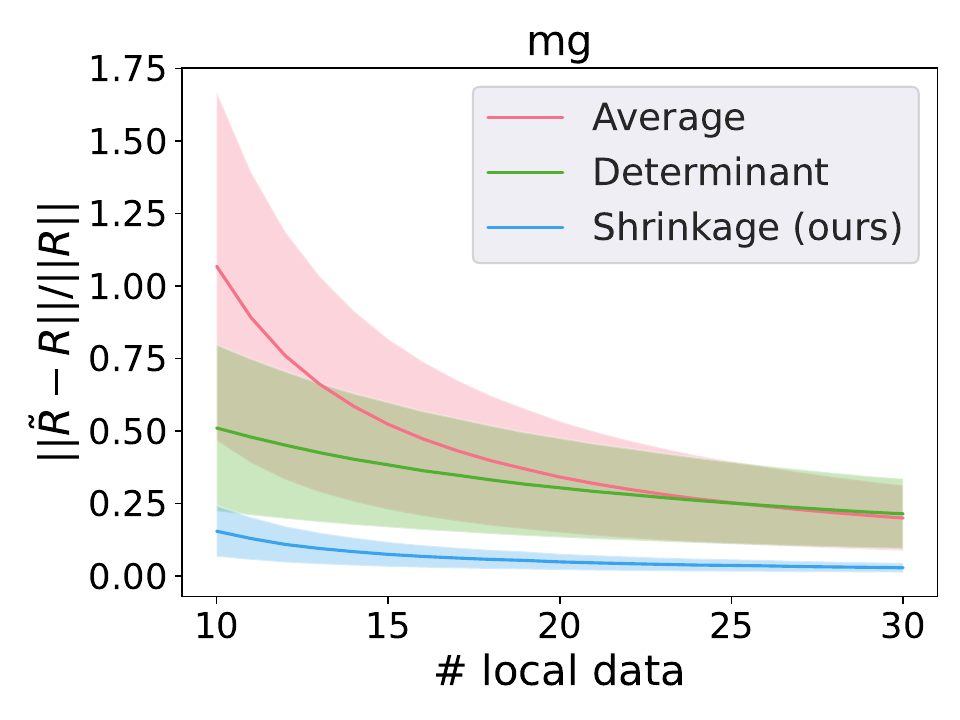}
\end{subfigure}
\hfill
\begin{subfigure}{0.23\linewidth}
\includegraphics[width=\linewidth]{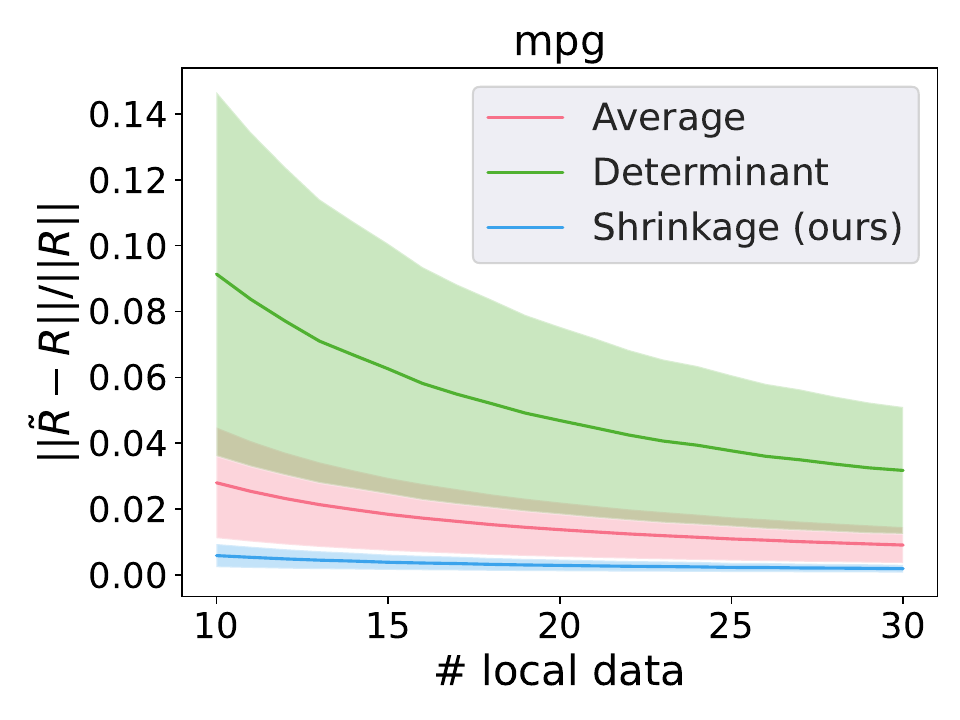}
\end{subfigure}
\hfill
\begin{subfigure}{0.23\linewidth}
\includegraphics[width=\linewidth]{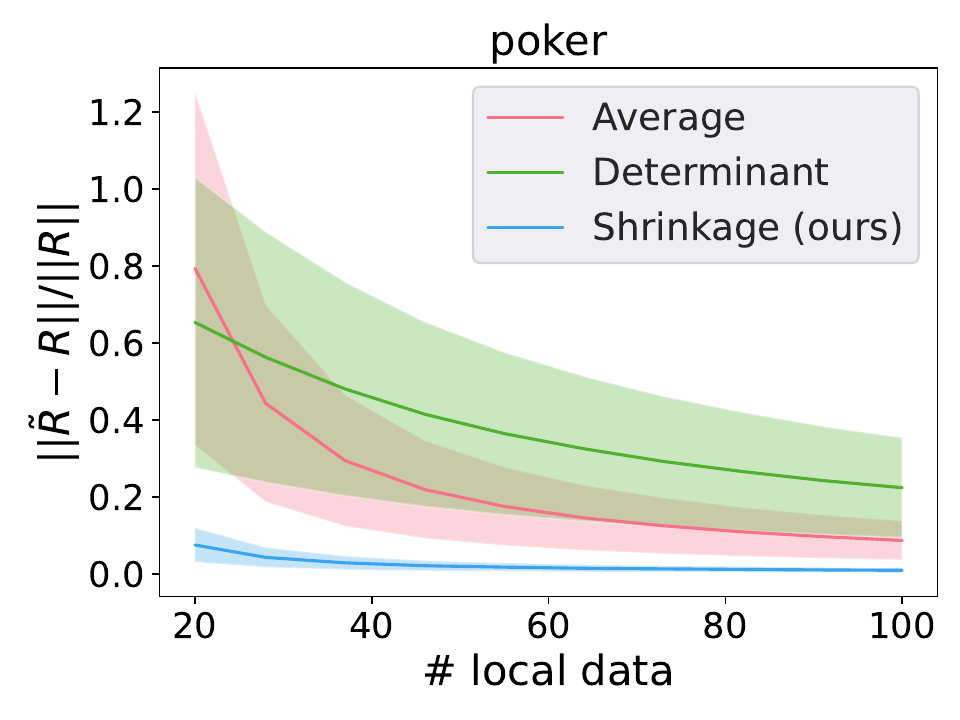}
\end{subfigure} 

\begin{subfigure}{0.23\linewidth}
\includegraphics[width=\linewidth]{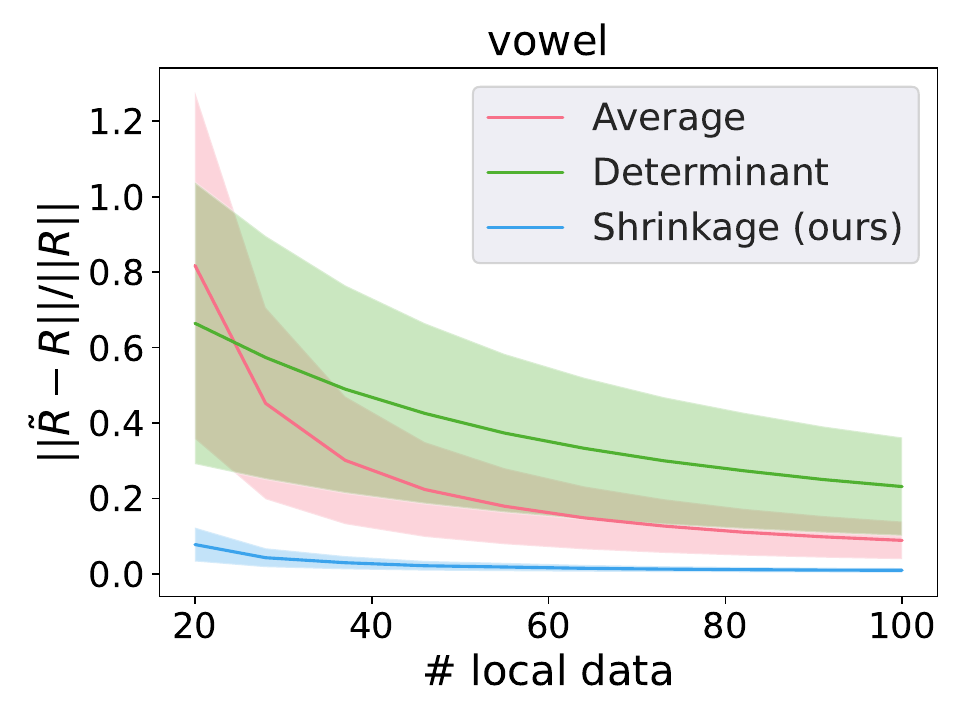}
\end{subfigure}
\hfill
\begin{subfigure}{0.23\linewidth}
\includegraphics[width=\linewidth]{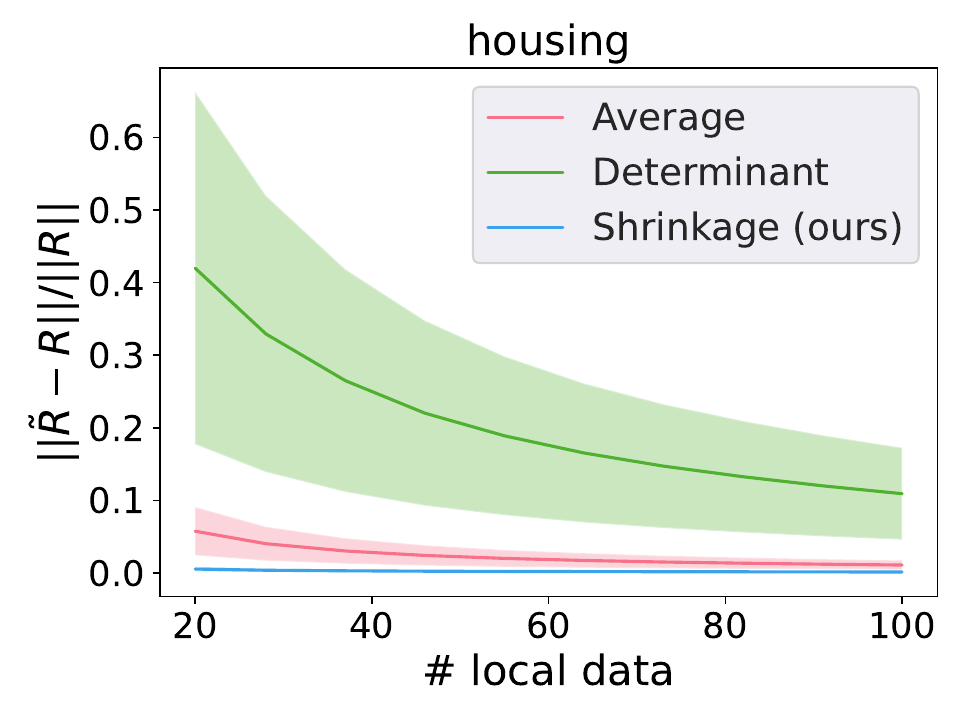}
\end{subfigure}
\hfill
\begin{subfigure}{0.23\linewidth}
\includegraphics[width=\linewidth]{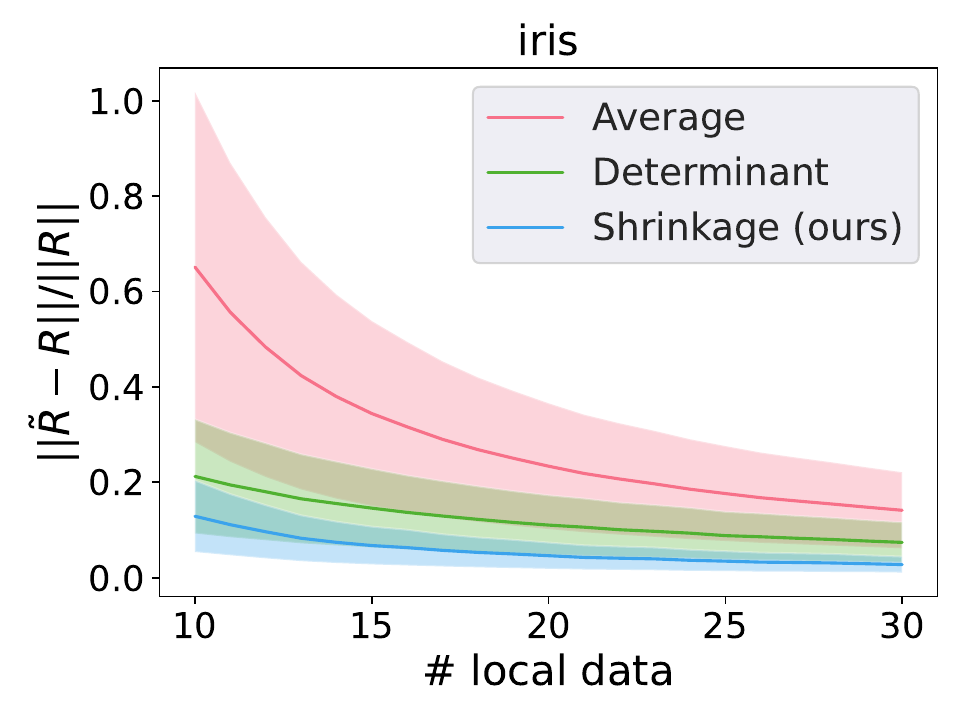}
\end{subfigure}
\hfill
\begin{subfigure}{0.23\linewidth}
\includegraphics[width=\linewidth]{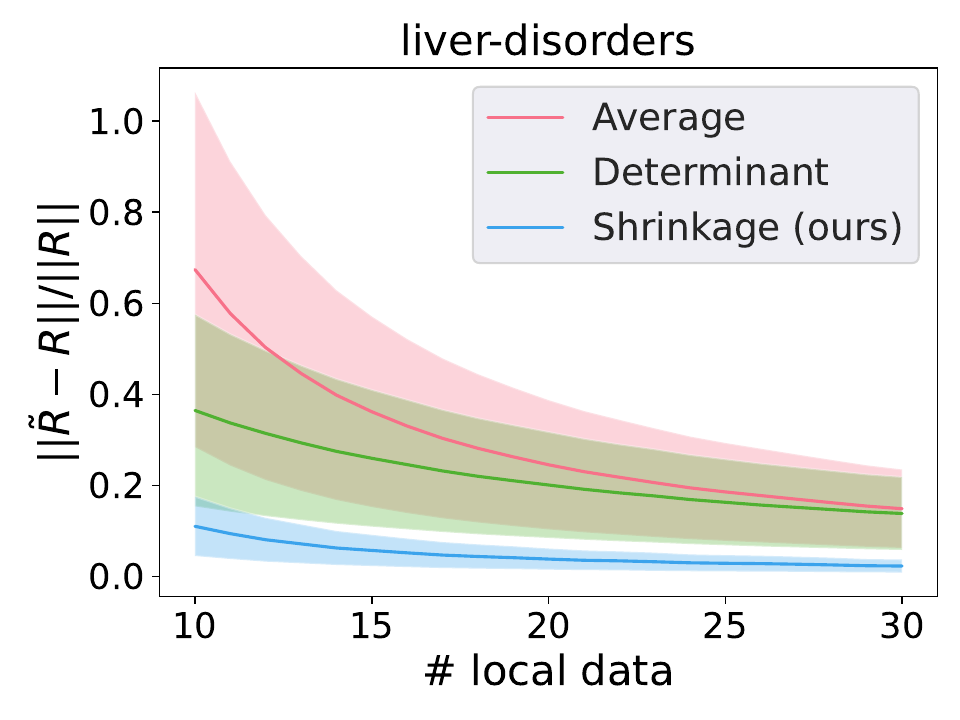}
\end{subfigure}
\caption{Sketched real data experiments on covariance resolvent estimation. We use sketching matrix with entry i.i.d. $\frac{1}{\sqrt{n}}U(-\frac{\sqrt{12}}{2},\frac{\sqrt{12}}{2})$, sketch size  $m=10000$, regularizer $\lambda=0.001$. We use $\Sigma=\frac{1}{n}A^TA$.}\label{fig8}
\end{figure}



\subsubsection{Experiments of small regularizer regime}\label{sec5.1.4append}
We test our method in the small regularizer regime discussed in Section \ref{small-regularizer}. Figure \ref{sr} shows the simulation results on synthetic data. The  results suggest that when the regularizer is large, using $\frac{md}{n}$ is much worse than using $\frac{md_\lambda}{n}$, while when the regularizer is small, $\frac{md}{n}$ can be used in place of $\frac{md_\lambda}{n}$ in computing $\tilde R_s$ and $\tilde R_s$ still approximates $R$ well, which confirms Theorem \ref{thm2.5} (see Section \ref{section5.1} for the definition of $\tilde R_s$ and $R$).
\begin{figure}[h]
\centering
\begin{subfigure}{0.23\linewidth}
\includegraphics[width=\linewidth]{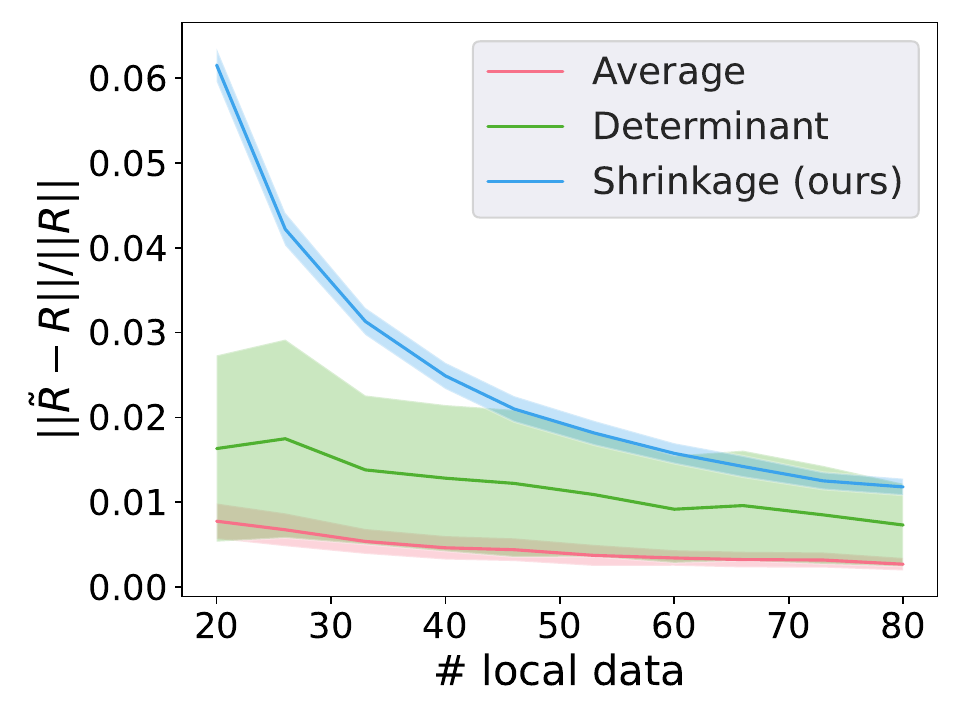}
\end{subfigure}
\begin{subfigure}{0.23\linewidth}
\includegraphics[width=\linewidth]{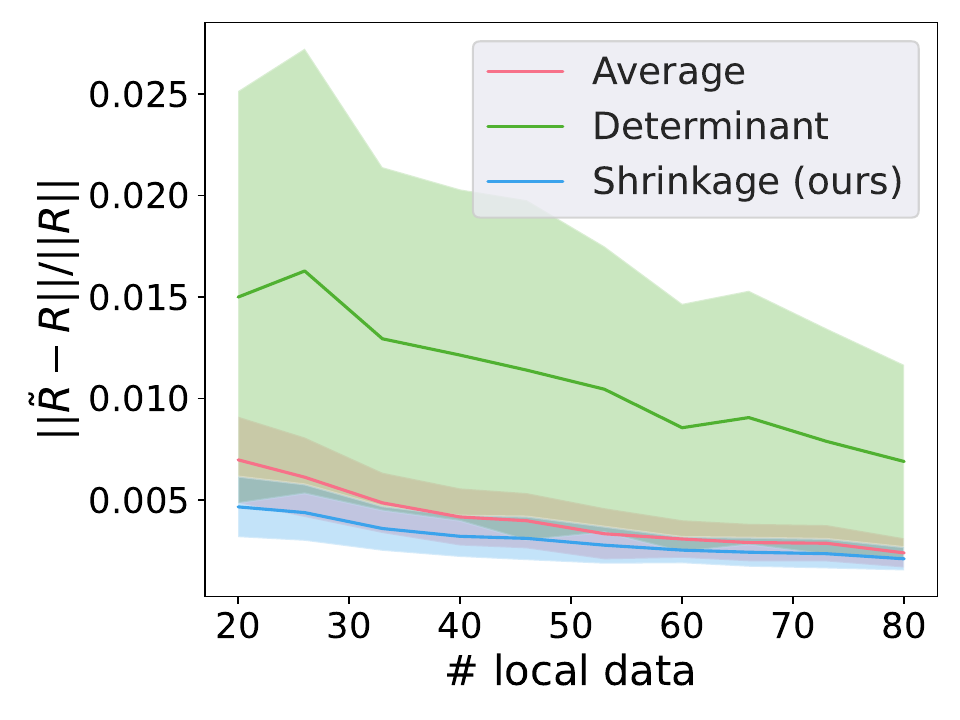}
\end{subfigure}
\begin{subfigure}{0.23\linewidth}
\centering
\includegraphics[width=\linewidth]{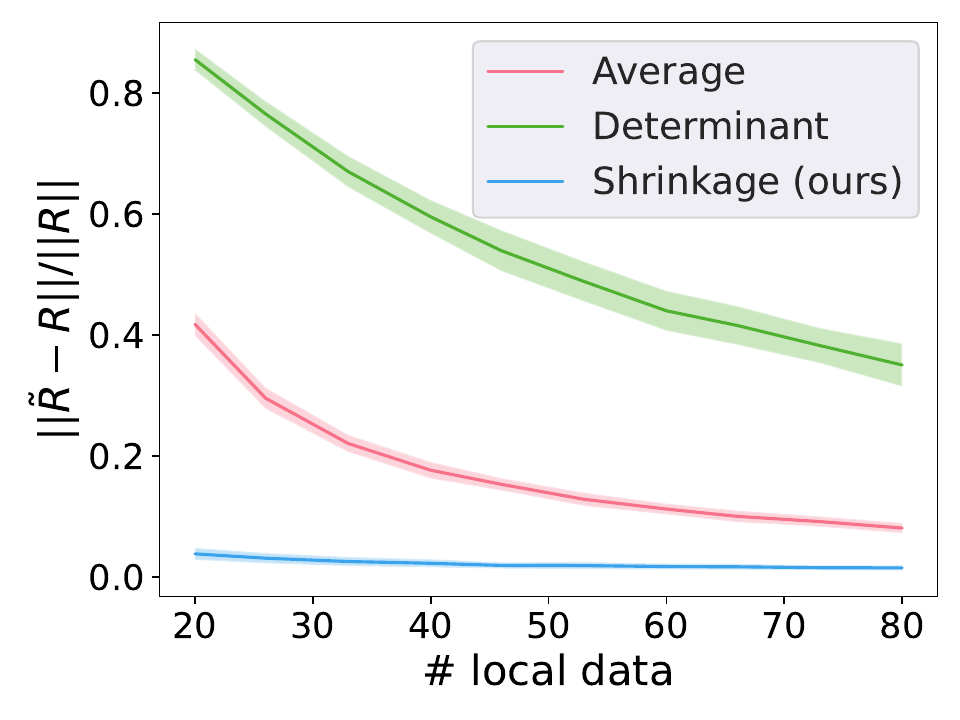}
\end{subfigure}
\caption{Synthetic data experiments for small regularizer regime. We experiment with $m=100$ agents and data dimension $d=10$. Data is i.i.d. $\mathcal{N}(0,\Sigma), \Sigma=100C^TC, C_{ij}\sim U(0,1).$ We take regularizer $\lambda=2000$ in the first two plots and $\lambda=1$ in the third plot.  $\frac{md}{n}$ is used in place of $\frac{md_\lambda}{n}$ in the first plot and  the third plot. $\frac{md_\lambda}{n}$ is used in the second plot.}\label{sr}
\end{figure}

We also test with sketched real data. Figure \ref{sr2} shows the result for sketched abalone dataset, which is similar to the synthetic data result.

\begin{figure}[H]
\centering
\begin{subfigure}{0.23\linewidth}
\includegraphics[width=\linewidth]{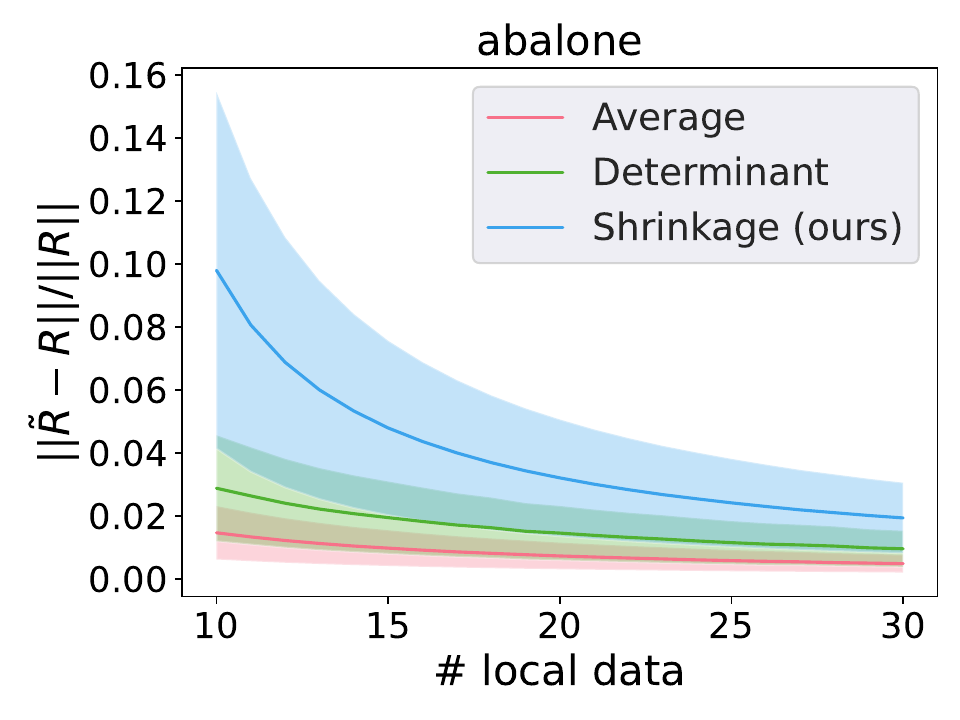}
\end{subfigure}
\begin{subfigure}{0.23\linewidth}
\includegraphics[width=\linewidth]{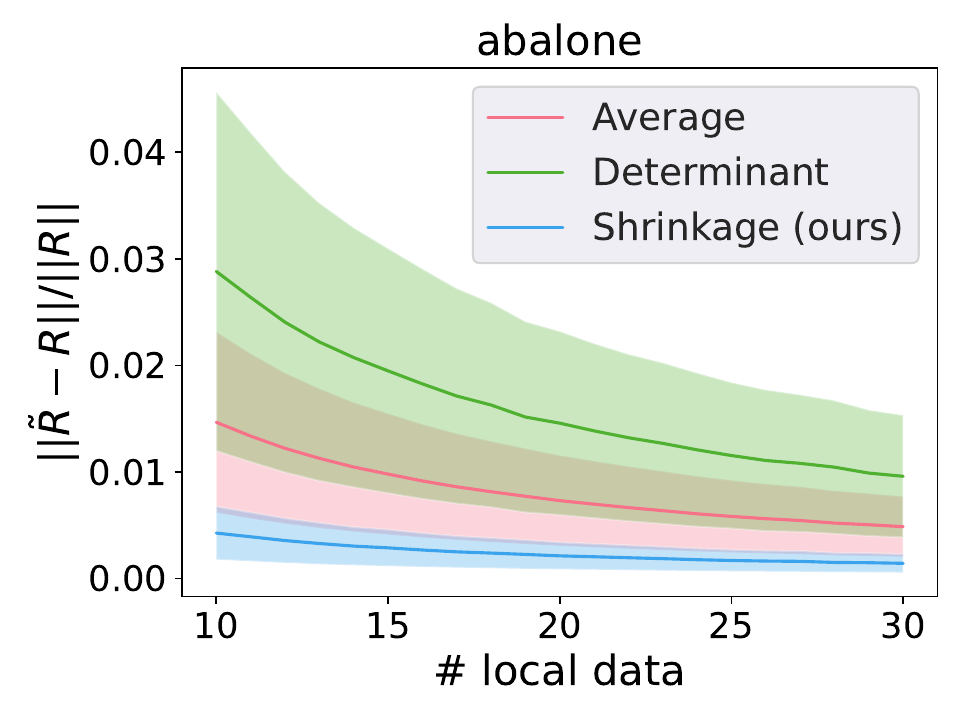}
\end{subfigure}
\begin{subfigure}{0.23\linewidth}
\includegraphics[width=\linewidth]{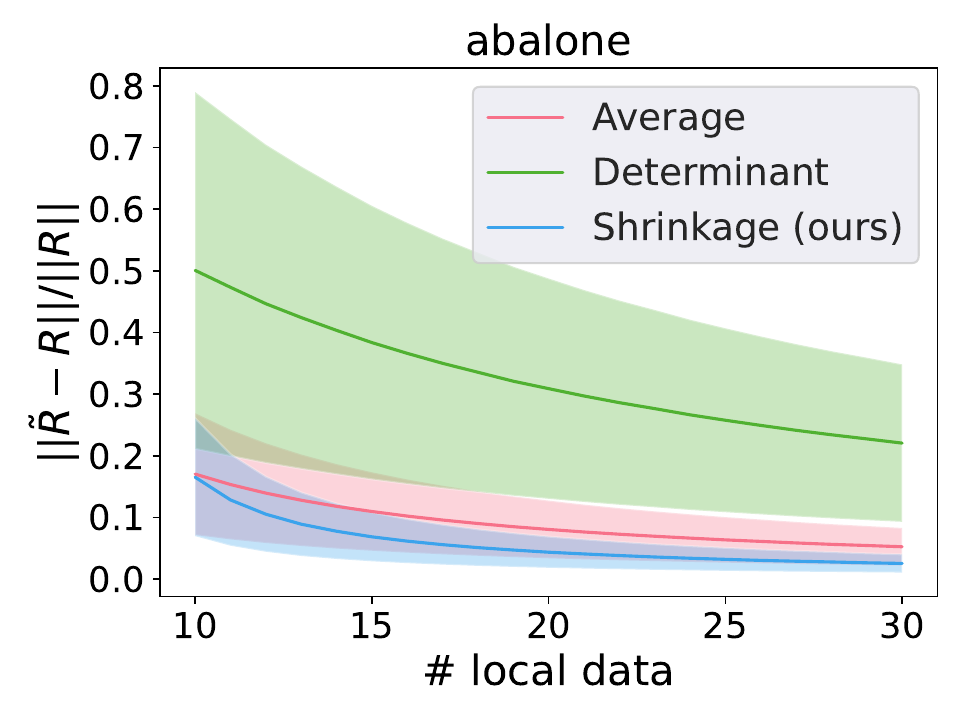}
\end{subfigure}

\caption{Sketched real data experiments for small regularizer regime. We use sketching matrix with entry i.i.d. $\mathcal{N}(0,\frac{1}{m})$, sketch size $m=10000$. We use $\Sigma=\frac{1}{n}A^TA$. We take regularizer $\lambda=1$ in the first two plots and $\lambda=0.001$ in the third plot.  $\frac{md}{n}$ is used in place of $\frac{md_\lambda}{n}$ in the first plot and  the third plot. $\frac{md_\lambda}{n}$ is used in the second plot.}\label{sr2}
\end{figure}

\subsection{Supplementary Simulation for Section \ref{section5.2}}
\subsubsection{Additional Experiments on Distributed newton's  method}\label{c.2.1append}
Here we give additional simulation results on distributed Newton's method for minimizing regularized quadratic loss with normalized real datasets. See Section \ref{section5.2} for a description of the setup and different methods being compared.  Figure \ref{fig11} shows the superiority of our method for saving communication rounds in  distributed Newton's method.
\begin{figure}[h]
\begin{subfigure}{0.23\linewidth}
\includegraphics[width=\linewidth]{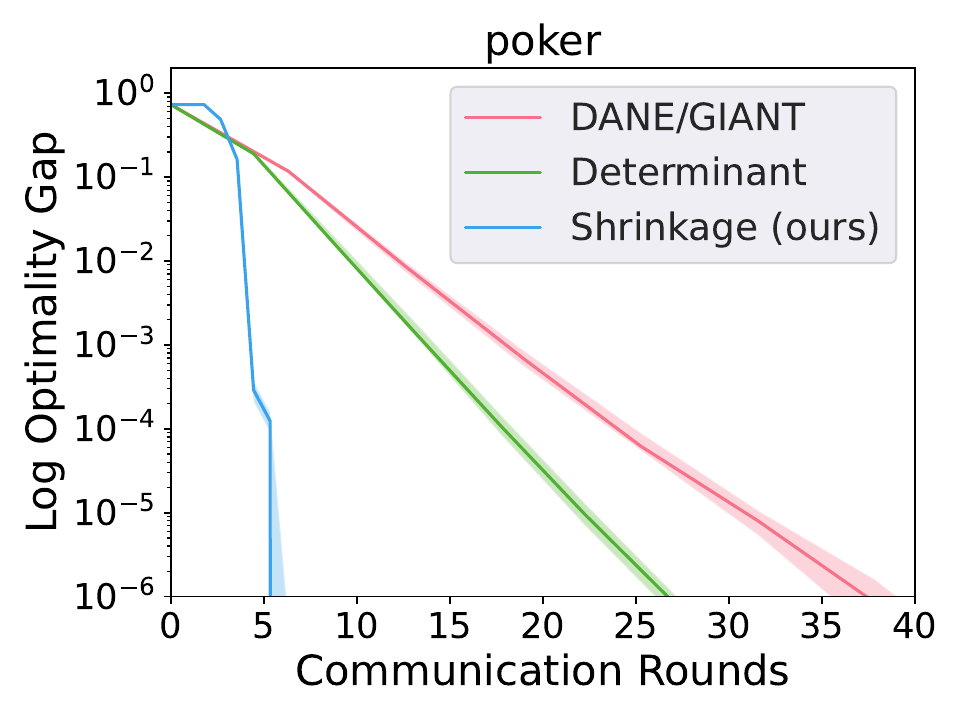}
\end{subfigure}
\hfill
\begin{subfigure}{0.23\linewidth}
\includegraphics[width=\linewidth]{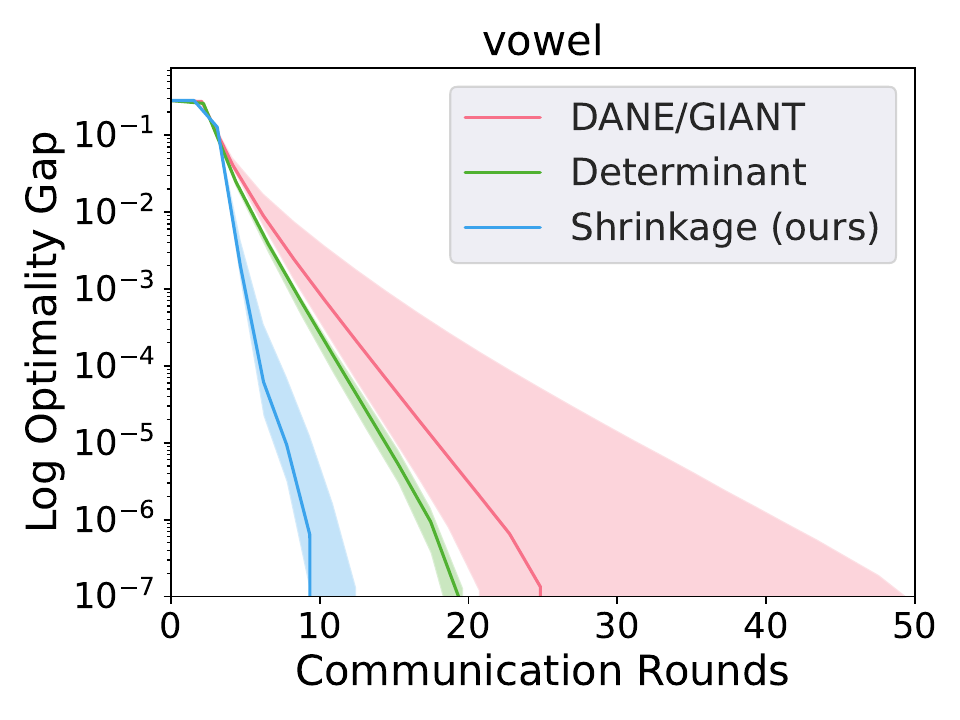}
\end{subfigure}
\hfill
\begin{subfigure}{0.23\linewidth}
\includegraphics[width=\linewidth]{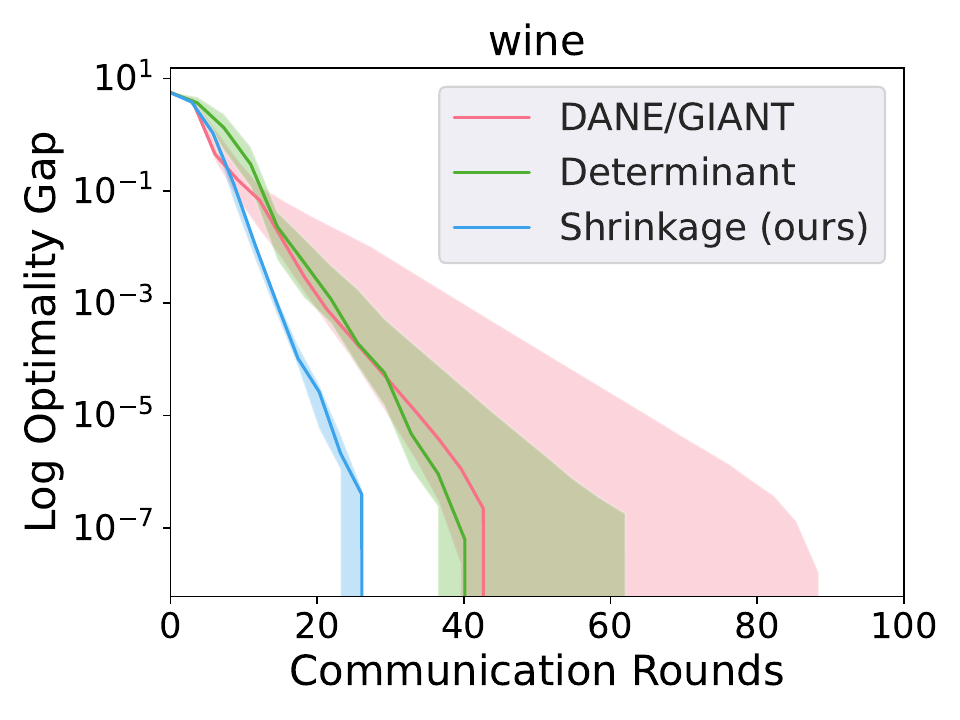}
\end{subfigure}
\hfill
\begin{subfigure}{0.23\linewidth}
\includegraphics[width=\linewidth]{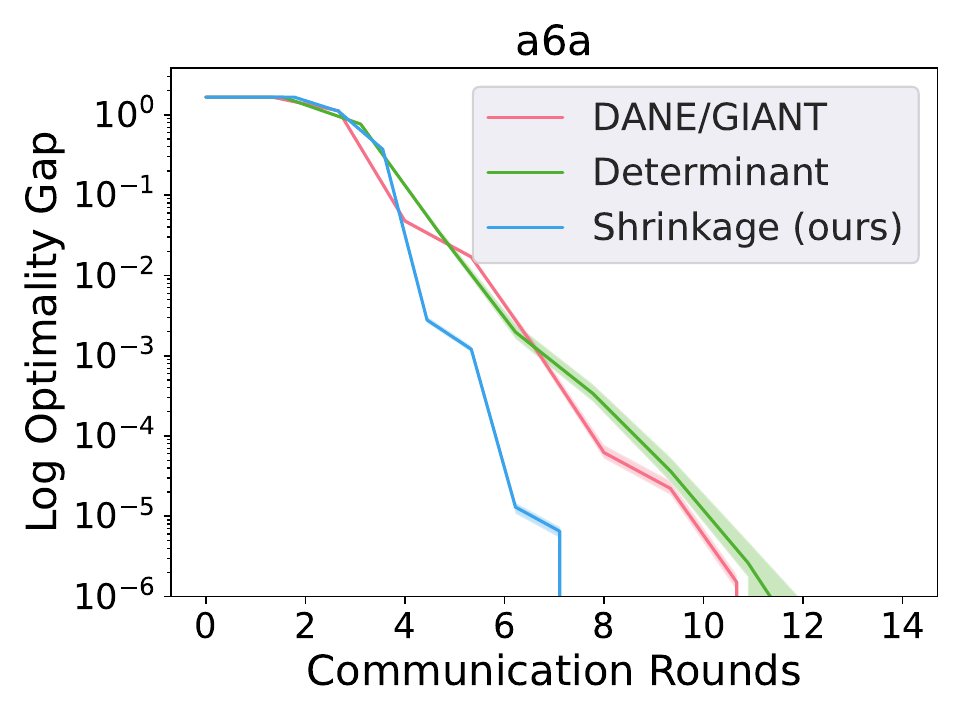}
\end{subfigure}
\caption{Normalized real data experiments on distributed Newton's method for ridge regression. Number of total data is rounded down as multiple of number of agents and number of local data is number of total data divided by number of agents. Let $\lambda$ denote the regularizer and $m$ denote the number of agents. We pick  $m=1000,\lambda=0.001$ for poker, $m=20,\lambda=0.05$ for vowel, $m=10,\lambda=0.01$ for wine, $m=50,\lambda=0.5$ for a6a.
}\label{fig11}
\end{figure}

\subsubsection{Experiments on Distributed Inexact newton's method}\label{c.2.2append}
Figure \ref{pcg-ridge-real} shows simulation results for distributed inexact Newton's method.  The algorithm of our method is given in Section \ref{section3}. We compare with DiSCO, where Hessian inverse of the first agent is used as the preconditioning matrix in distributed preconditioned conjugate gradient method\footnote{For the implementation of DiSCO, we only borrow its preconditioner and don't use its initialization step and step size choice since they are too specific.}. Average method is taking the average of local Hessian inverses as the preconditioning matrix. Determinant method is using exactly the same approximate Hessian inverse as in the  determinantal averaging method discussed in Section \ref{prior-work} as the preconditioning matrix. We do not compare with conjugate gradient method since it usually takes many more steps to converge. Determinant method is not plotted whenever we encounter numerical stability issues.

From the plots, we see that compared to averaging method, DiSCO, and determinantal averaging method, our  shrinkage method achieves smaller log optimality gap in fewer rounds of communication on the datasets we have tested, which suggests our shrinkage method is approximating Hessian inverse more accurately. Another takeaway is that determinantal averaging method is unstable when data  dimension is large and computing determinant becomes infeasible, while our shrinkage method is always valid as long as the local data size is larger than effective dimension of local Hessian matrix. Note for minimizing quadratic loss, inexact Newton's method usually converges in one Newton step, and thus the discrepancy for different methods are smaller in these plots compared to distributed Newton's method's simulation plots.

\begin{figure}[H]
\begin{subfigure}{0.23\linewidth}
\includegraphics[width=\linewidth]{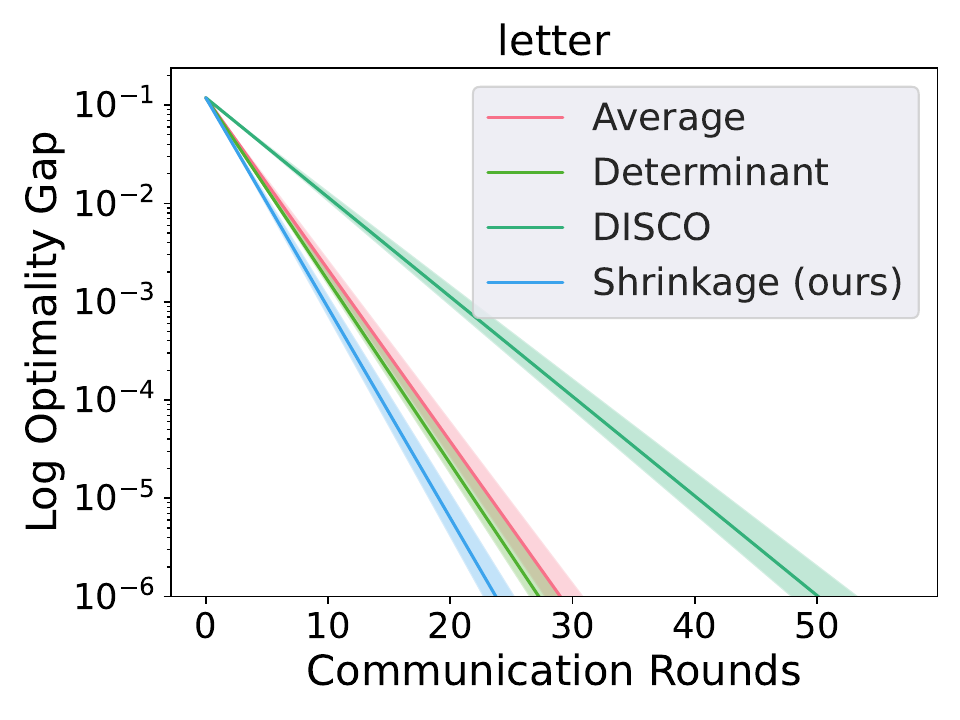}
\end{subfigure}
\hfill
\begin{subfigure}{0.23\linewidth}
\includegraphics[width=\linewidth]{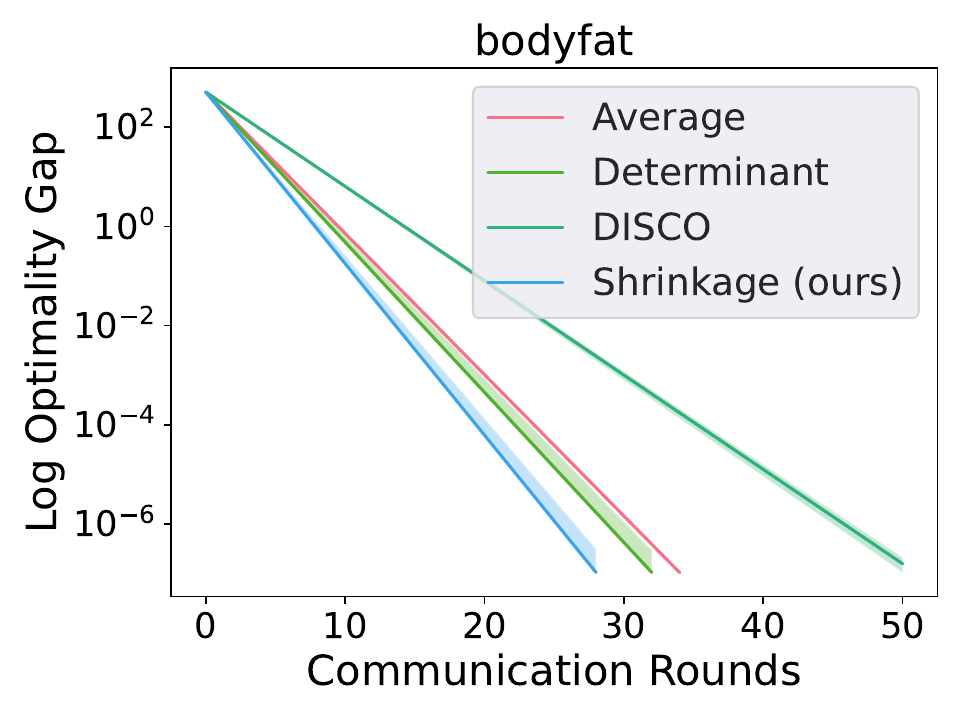}
\end{subfigure}
\hfill
\begin{subfigure}{0.23\linewidth}
\includegraphics[width=\linewidth]{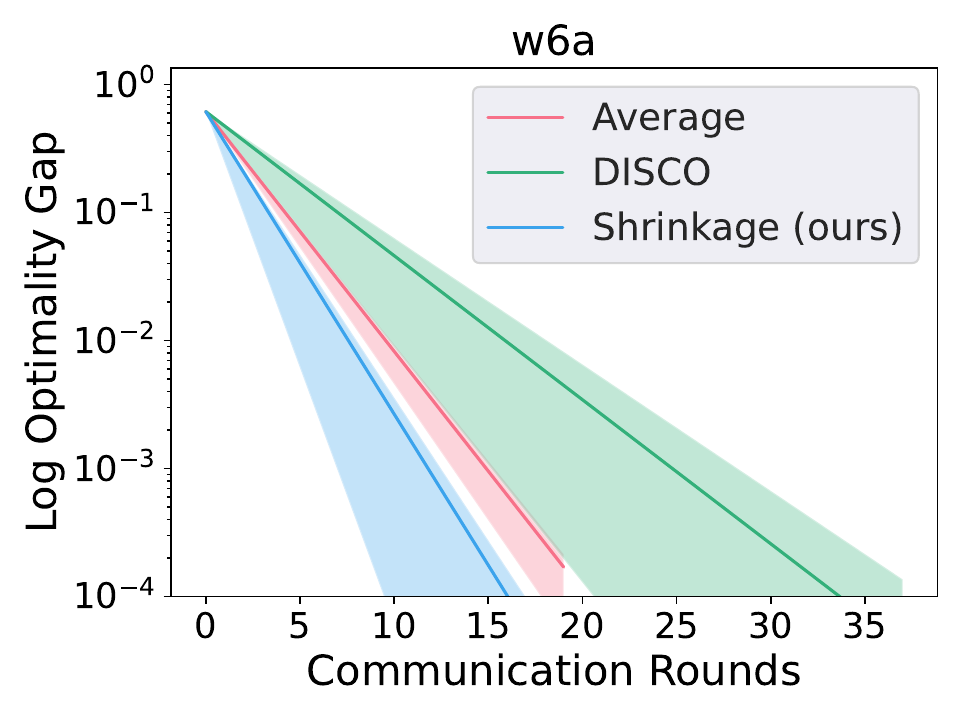}
\end{subfigure}
\hfill
\begin{subfigure}{0.23\linewidth}
\includegraphics[width=\linewidth]{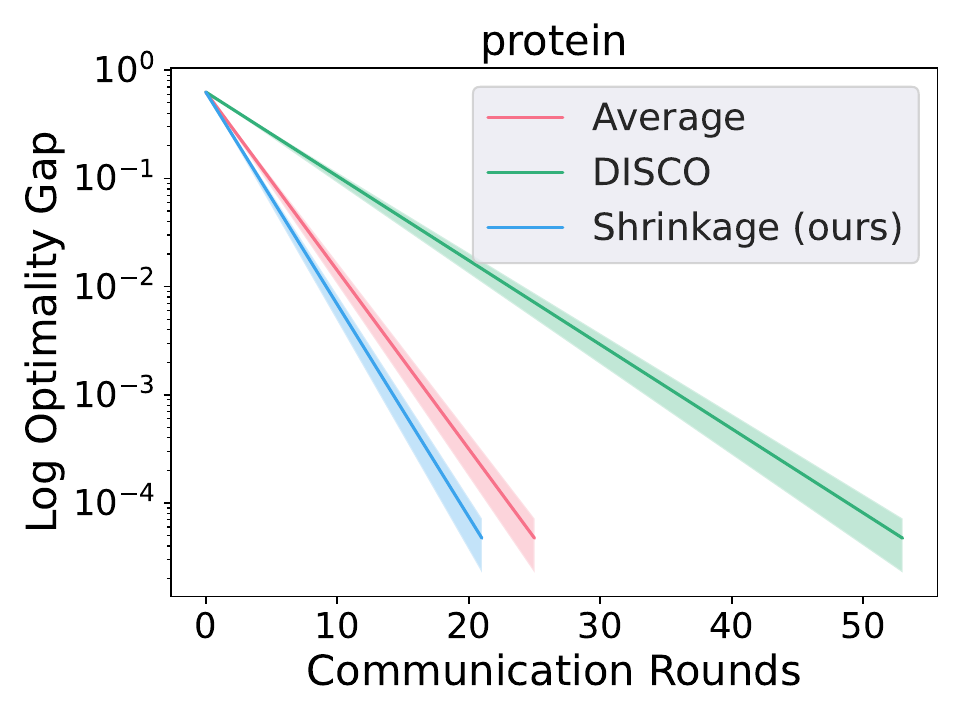}
\end{subfigure}

\begin{subfigure}{0.23\linewidth}
\includegraphics[width=\linewidth]{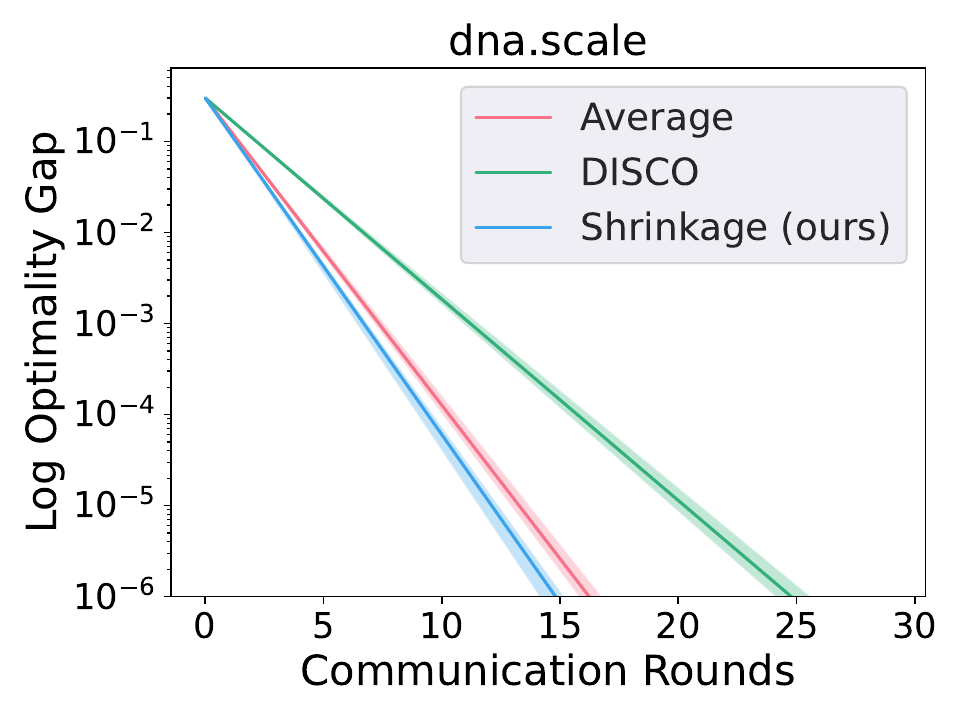}
\end{subfigure}
\hfill
\begin{subfigure}{0.23\linewidth}
\includegraphics[width=\linewidth]{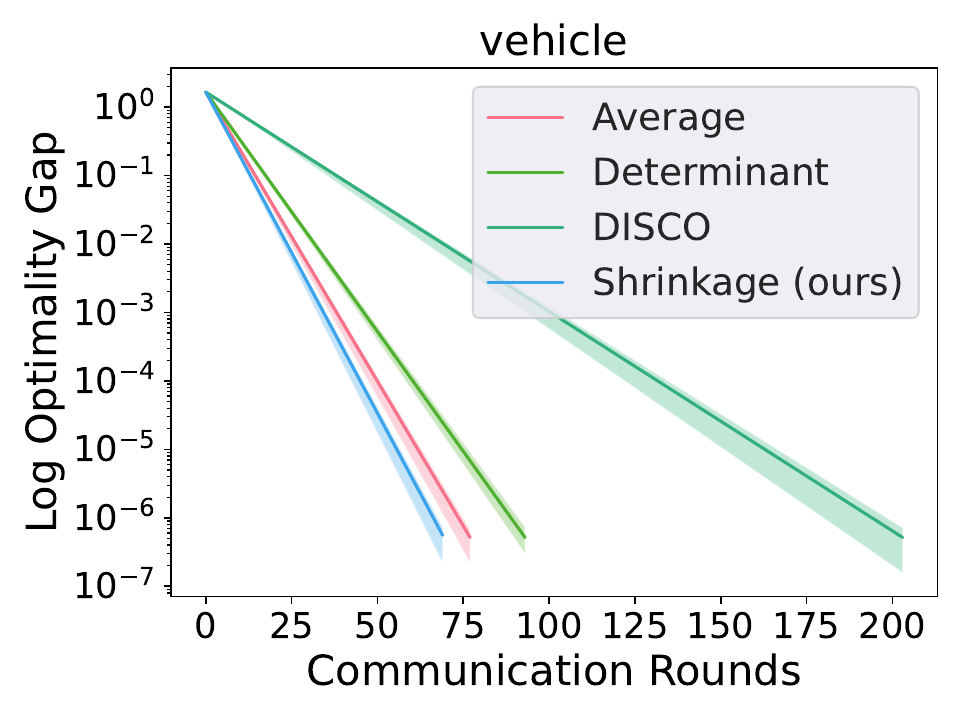}
\end{subfigure}
\hfill
\begin{subfigure}{0.23\linewidth}
\includegraphics[width=\linewidth]{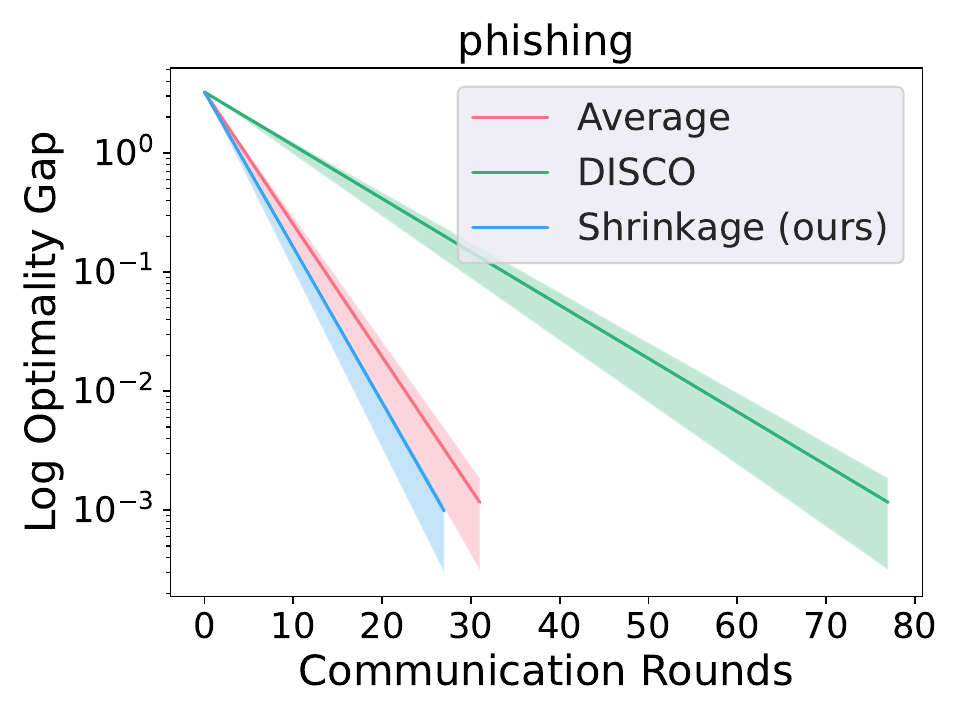}
\end{subfigure}
\hfill
\begin{subfigure}{0.23\linewidth}
\includegraphics[width=\linewidth]{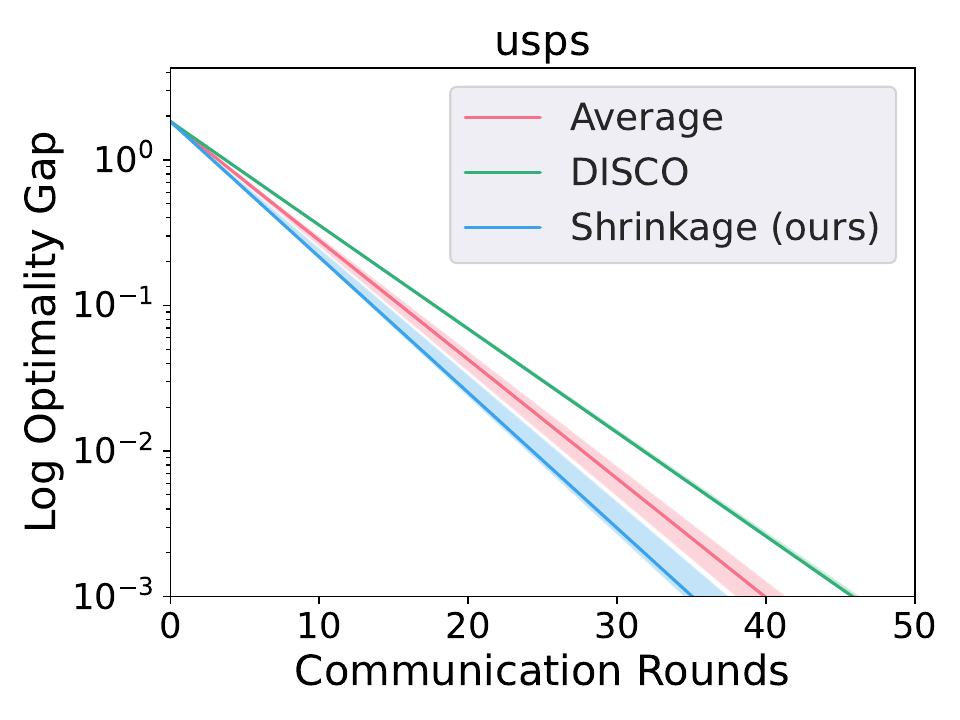}
\end{subfigure}
\hfill
\caption{Normalized real data experiments on distributed inexact Newton's method for ridge regression.  Number of total data is rounded down as an integer multiple of number of agents and number of local data equals the number of total data divided by the number of agents. Let $\lambda$ denote the regularizer and $m$ denote the number of agents. We pick $m=1000,\lambda=0.1$ for letter, $m=20,\lambda=0.01$ for bodyfat, $m=30,\lambda=1$ for  w6a, $m=50,\lambda=0.1$ for protein and usps,   $m=50,\lambda=10$ for dna.scale, $m=40,\lambda=1$ for vehicle, $m=200,\lambda=0.01$ for phishing.
}\label{pcg-ridge-real}
\end{figure}
\subsubsection{Experiments on Logistic Regression}\label{logit}
We give simulation results for distributed Newton's method and distributed inexact Newton's method for logistic regression on normalized real datasets.  See Section \ref{section5.2} for a description of the setup and different methods being compared for distributed Newton's method. See Appendix \ref{c.2.2append} for a description of methods being compared for distributed inexact Newton's method.
According to our convergence analysis for non-quadratic loss in Section \ref{cvx}, we need each Newton's step to operate  on data independent of previous Newton's steps. Therefore we are taking fresh data batches for computing each Newton's step. Determinantal averaging method does not appear in the inexact Newton's method plots since the method failes due to numerical issues.

According to Figure \ref{logit-fig}, our shrinkage method frequently saves communication rounds for both Newton's method and inexact Newton's method compared to other methods, though the discrepancy is not as significant for the case with quadratic loss. Moreover, a larger variance in the performance of Newton's method is observed. This might be due to a small batch size used in each worker. We believe that our shrinkage method can be optimized further for non-quadratic losses, which is left as future work.
\begin{figure}[H]
\begin{subfigure}{0.23\linewidth}
\includegraphics[width=\linewidth]{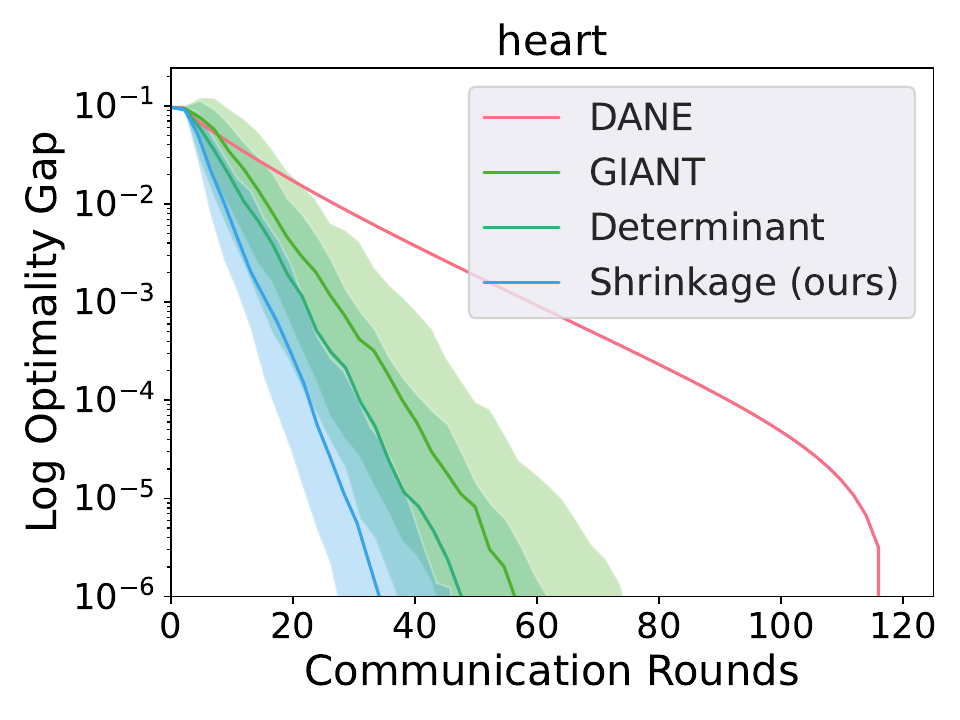}
\end{subfigure}
\hfill
\begin{subfigure}{0.23\linewidth}
\includegraphics[width=\linewidth]{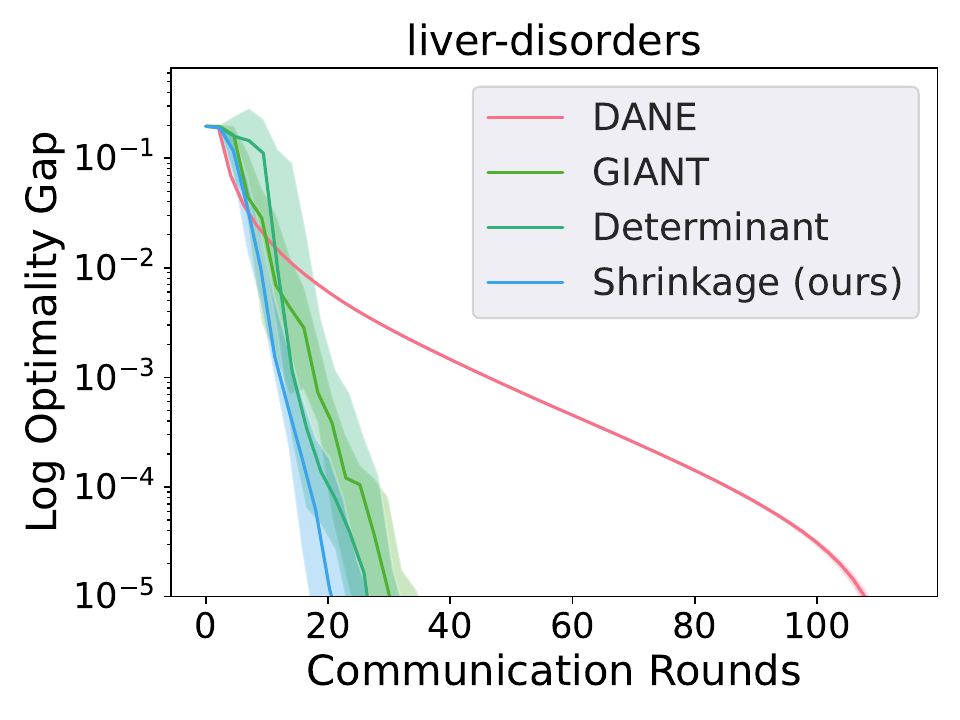}
\end{subfigure}
\hfill
\begin{subfigure}{0.23\linewidth}
\includegraphics[width=\linewidth]{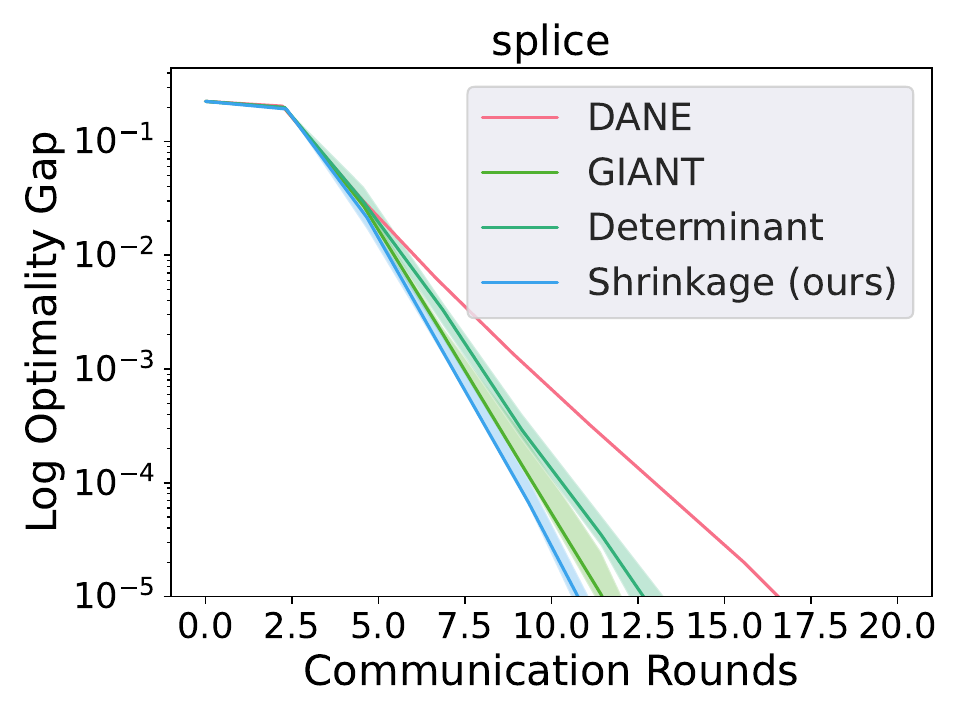}
\end{subfigure}
\hfill
\begin{subfigure}{0.23\linewidth}
\includegraphics[width=\linewidth]{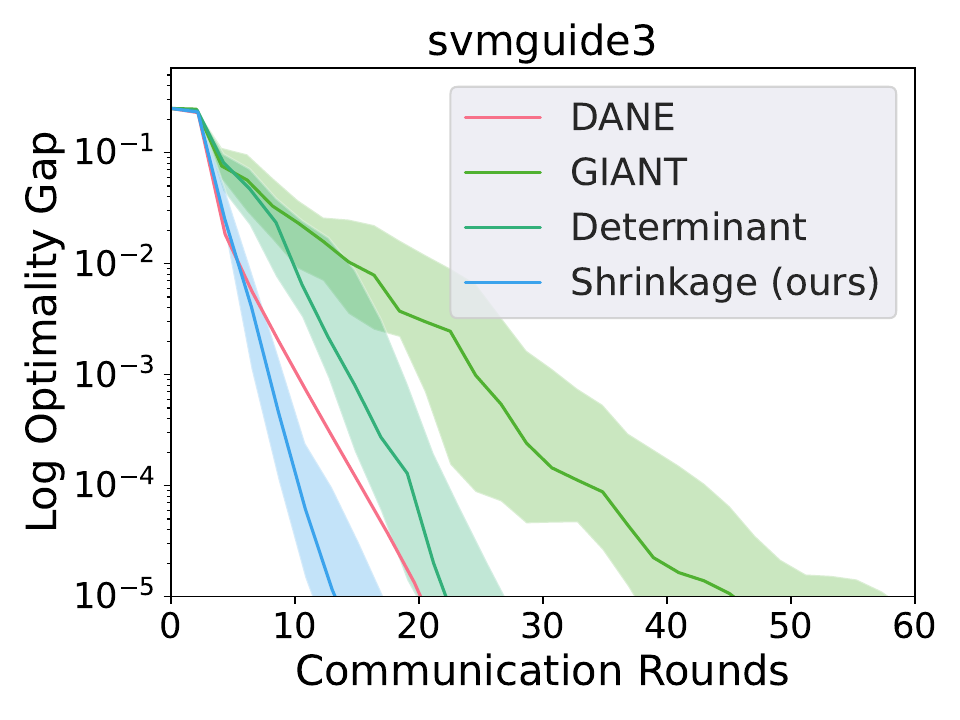}
\end{subfigure}

\begin{subfigure}{0.23\linewidth}
\includegraphics[width=\linewidth]{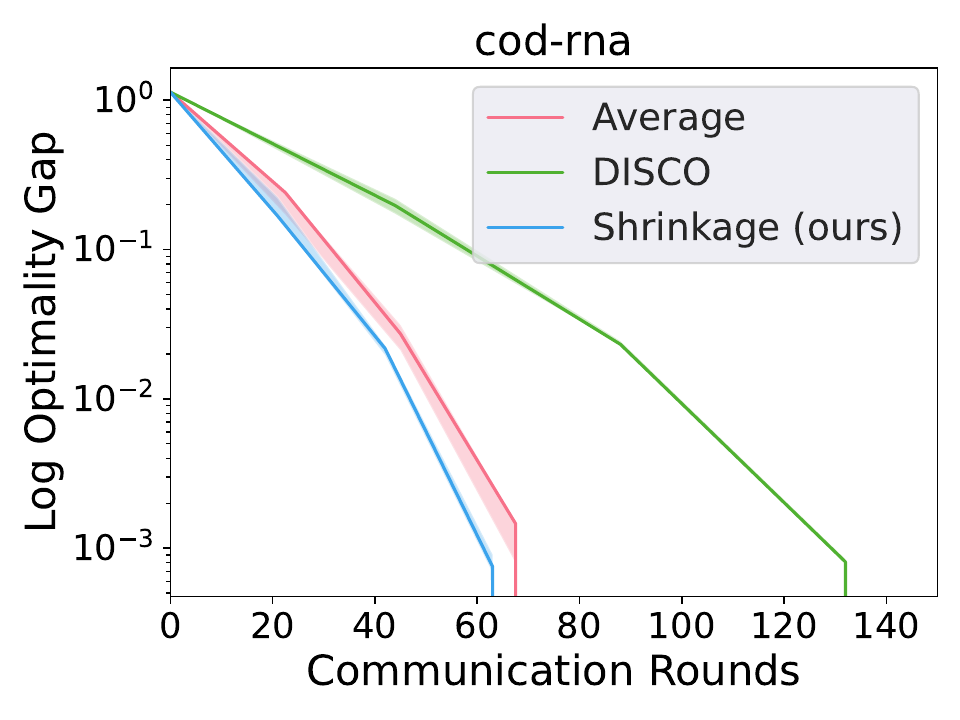}
\end{subfigure}
\hfill
\begin{subfigure}{0.23\linewidth}
\includegraphics[width=\linewidth]{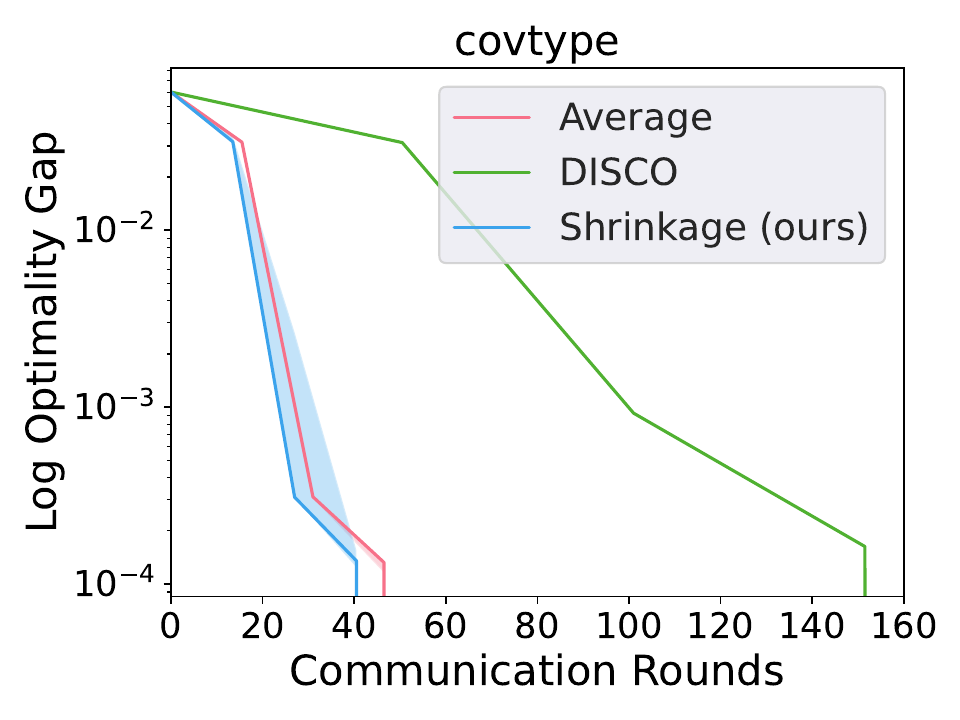}
\end{subfigure}
\hfill
\begin{subfigure}{0.23\linewidth}
\includegraphics[width=\linewidth]{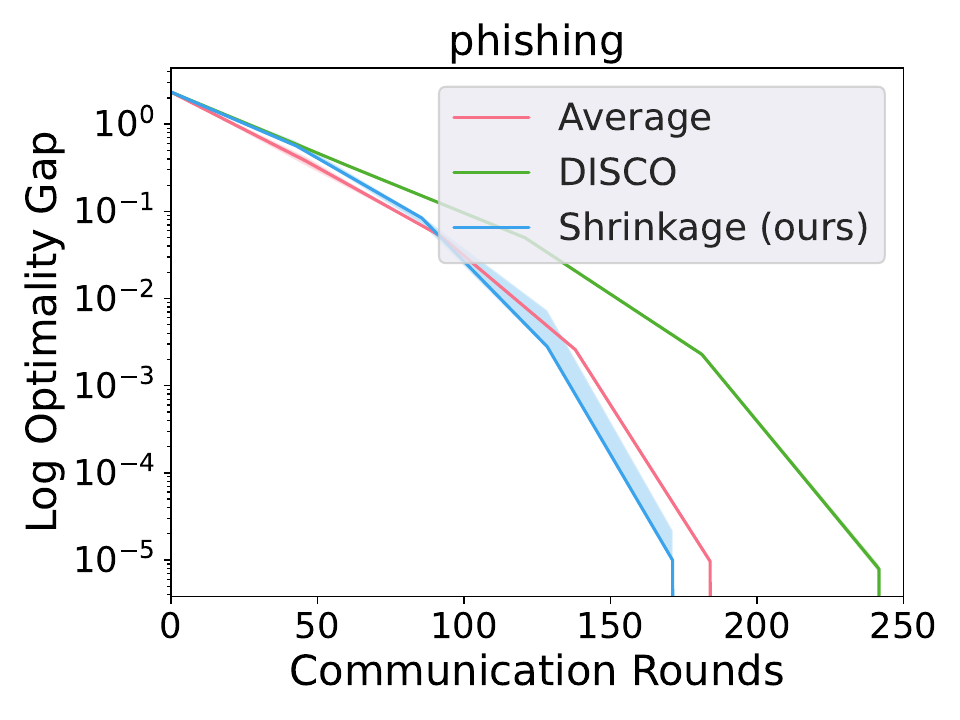}
\end{subfigure}
\hfill
\begin{subfigure}{0.23\linewidth}
\includegraphics[width=\linewidth]{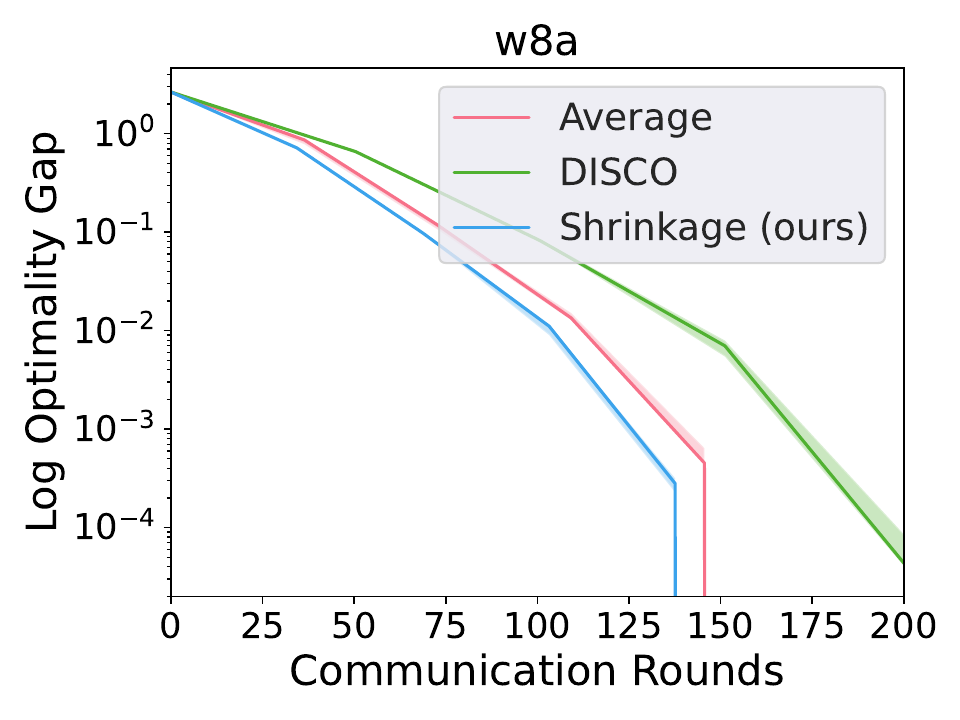}
\end{subfigure}
\caption{Normalized real data experiments on distributed second-order optimization algorithms for logistic regression. Top four plots are for distributed Newton's method and bottom four plots are for distributed inexact Newton's method. Number of total data is rounded down as multiple of number of agents and number of local data is number of total data divided by number of agents. For each Newton's step, a refreshed data batch containing the number of local data divided by max$\_$iters pieces of data is used (we limit number of Newton's step to not exceed max$\_$iters). Let $\lambda$ denote the regularizer and $m$ denote the number of agents. We pick $m=5,\lambda=0.01, $max$\_$iters$=20$ for heart, $m=2,\lambda=0.01, $max$\_$iters$=10$ for liver-disorders, $m=3,\lambda=0.1, $max$\_$iters$=5$ for splice, $m=10,\lambda=0.1, $max$\_$iters$=20$ for svmguide3, $m=100,\lambda=1e-5, $max$\_$iters$=50$ for cod-rna, $m=200,\lambda=1e-5, $max$\_$iters$=50$ for covtype, $m=40,\lambda=0.01, $max$\_$iters$=50$ for phishing, $m=50,\lambda=0.1, $max$\_$iters$=50$ for w8a.}\label{logit-fig}
\end{figure}
\subsubsection{Experiments on Iterative Hessian Sketch With Optimal Shrinkage}\label{ihs-append}
We only present the  simulation plots for IHS and IHS with  shrinkage where a heuristic shrinkage coefficient is used in the main text in Section \ref{section5.3}.  Here we provide more plots on these two methods and we also provide the exact version for IHS with optimal shrinkage. For sake of comparison, we include the plots in the main text here again.

Figure  \ref{ihs-figure1} presents our simulation results on IHS with shrinkage where the effective dimension of sketched data is used. From the  figure, we see that IHS equipped with  shrinkage method beats the classic IHS method in datasets we experimented with, though for datasets bodyfat, housing, mpg and triazines, the difference between these two methods are more obvious. Figure \ref{ihs-figure2} presents the plots for IHS with the exact optimal shrinkage coefficient, on the same datasets choice and with the same parameter choice. Still, IHS with shrinkage method beats the classic IHS method in 
 datasets we have tested on. But this time, the difference between the two methods are obvious for all datasets. If we look at Figure \ref{ihs-figure1} and Figure \ref{ihs-figure2} together, the performance of IHS with the exact shrinkage coefficient is at least as good as IHS with the heuristic shrinkage coefficient, which is as expected, though the heuristic version does not worsen the performance much such that IHS with heuristic shrinkage still remains superior compared to the classic IHS method.
\begin{figure}[H]
\begin{subfigure}{0.23\linewidth}
\includegraphics[width=\linewidth]{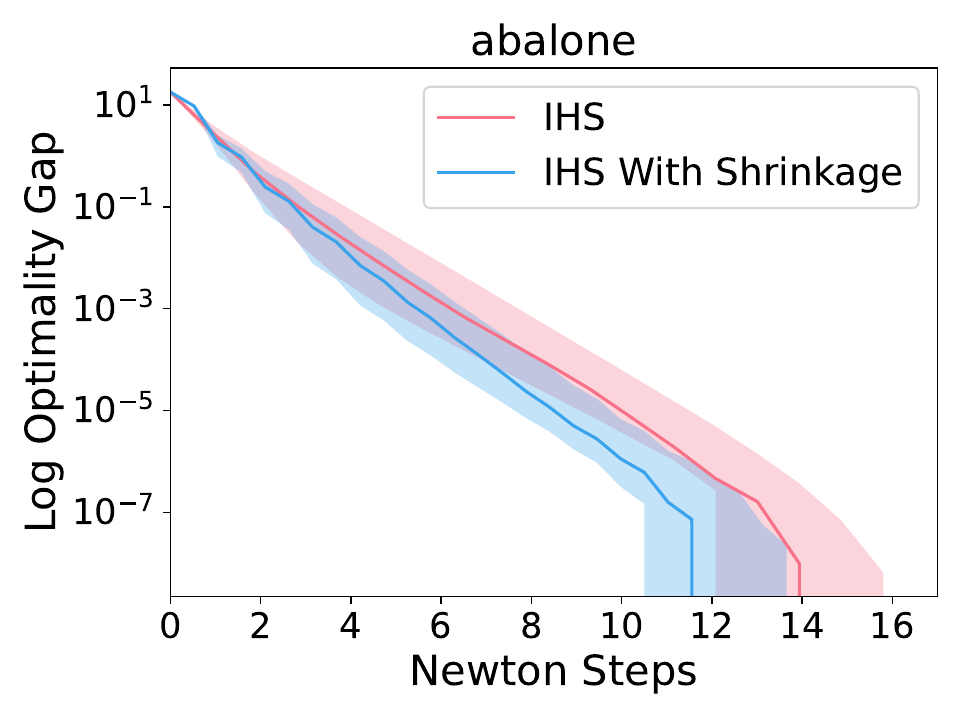}
\end{subfigure}
\hfill
\begin{subfigure}{0.23\linewidth}
\includegraphics[width=\linewidth]{ihs-bodyfat-16-www.pdf}
\end{subfigure}
\hfill
\begin{subfigure}{0.23\linewidth}
\includegraphics[width=\linewidth]{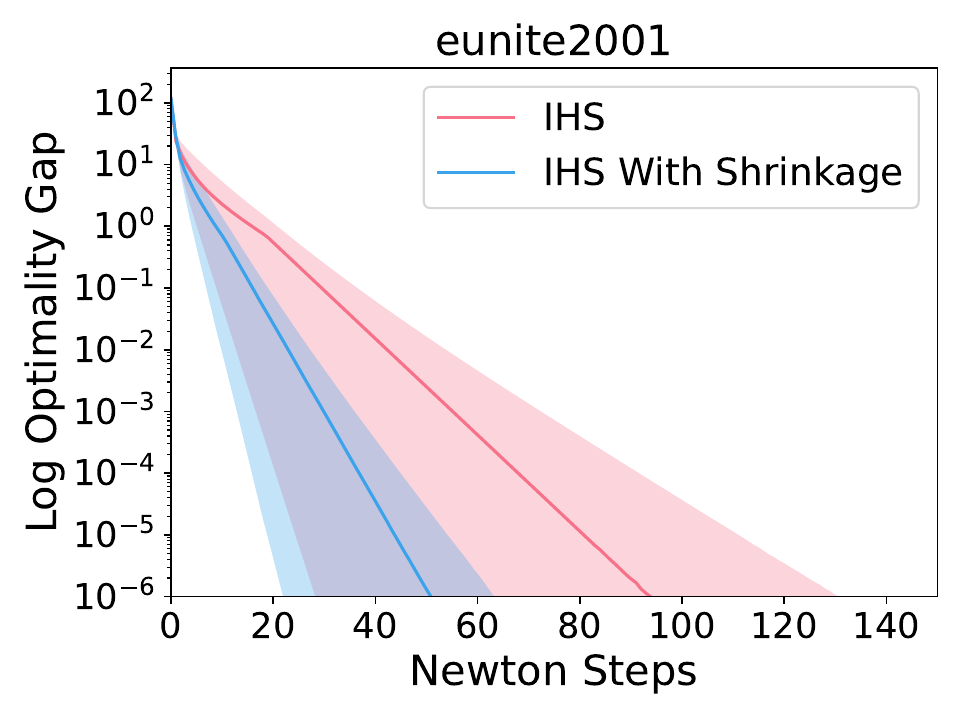}
\end{subfigure}
\hfill
\begin{subfigure}{0.23\linewidth}
\includegraphics[width=\linewidth]{ihs-housing-16-www.pdf}
\end{subfigure}

\begin{subfigure}{0.23\linewidth}
\includegraphics[width=\linewidth]{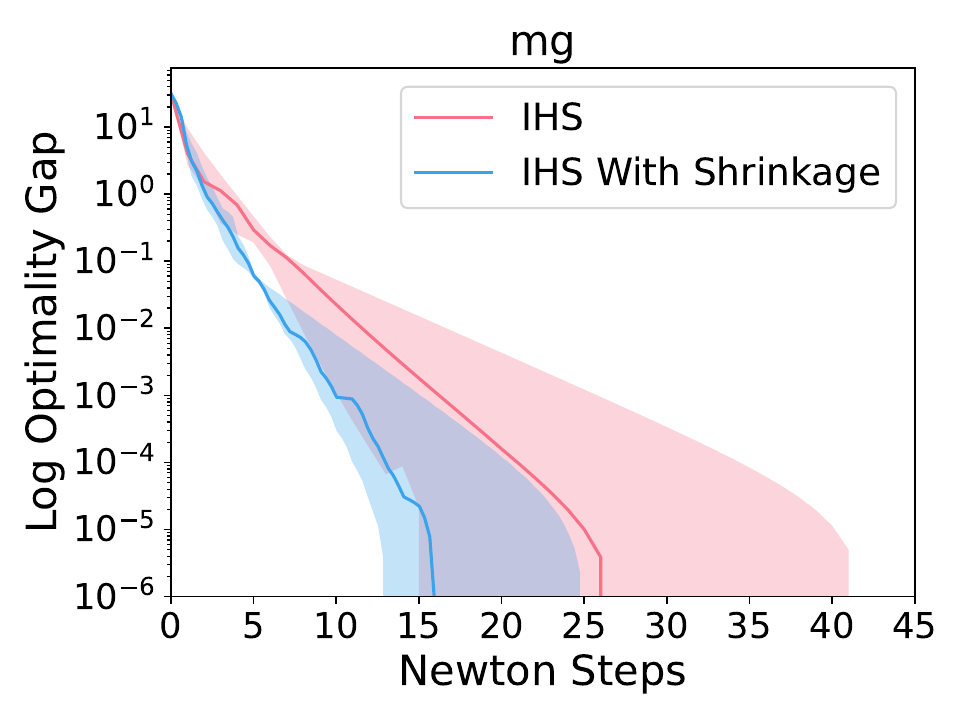}
\end{subfigure}
\hfill
\begin{subfigure}{0.23\linewidth}
\includegraphics[width=\linewidth]{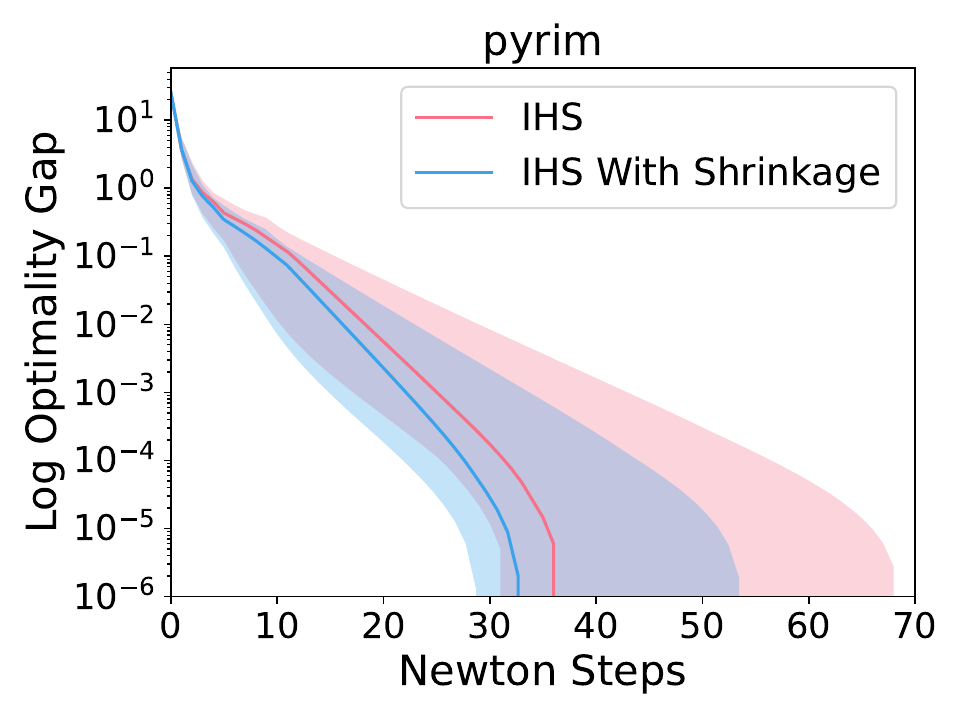}
\end{subfigure}
\hfill
\begin{subfigure}{0.23\linewidth}
\includegraphics[width=\linewidth]{ihs-mpg-16-www.pdf}
\end{subfigure}
\hfill
\begin{subfigure}{0.23\linewidth}
\includegraphics[width=\linewidth]{ihs-triazines-16-www.pdf}
\end{subfigure}
\caption{Real data experiments on Iterative Hessian Sketch method with heuristic shrinkage coefficient for ridge regression. Let $\lambda$ denote the regularizer and $m$ denote the sketch size. We pick $\lambda=0.001$ for abalone,triazines and $\lambda=0.01$ for bodyfat, eunite2001, housing, mg, pyrim, mpg . We pick $m=50$  for abalone, $m=100$ for bodyfat and pyrim, $m=300$ for eunite2001, housing, and triazines, $m=20$ for mg, and $m=30$ for  mpg.}\label{ihs-figure1}
\end{figure}

\begin{figure}[H]
\begin{subfigure}{0.23\linewidth}
\includegraphics[width=\linewidth]{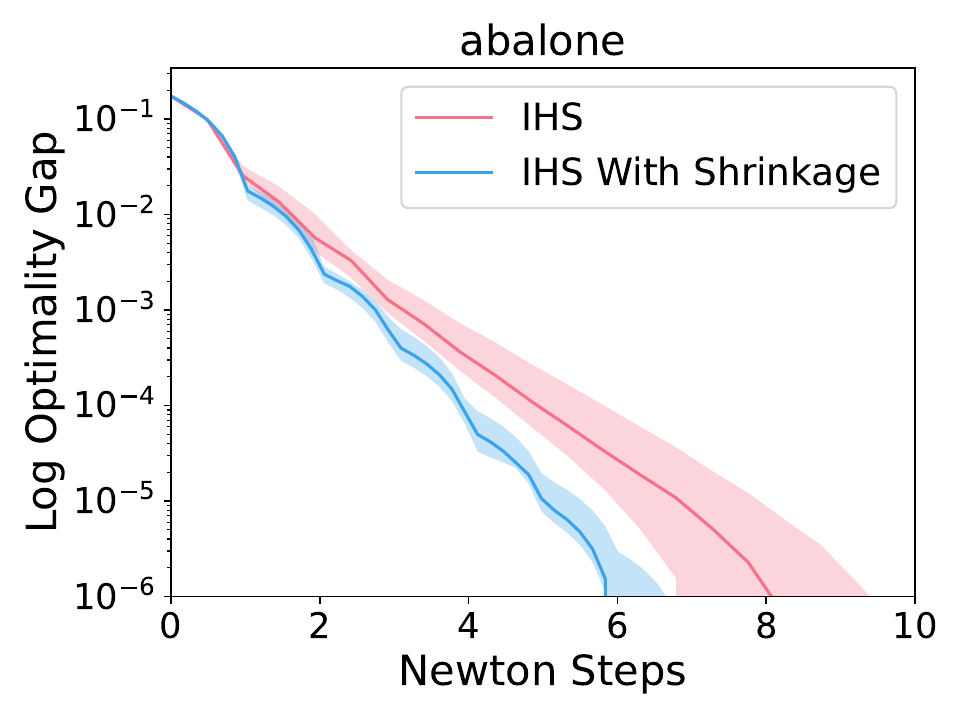}
\end{subfigure}
\hfill
\begin{subfigure}{0.23\linewidth}
\includegraphics[width=\linewidth]{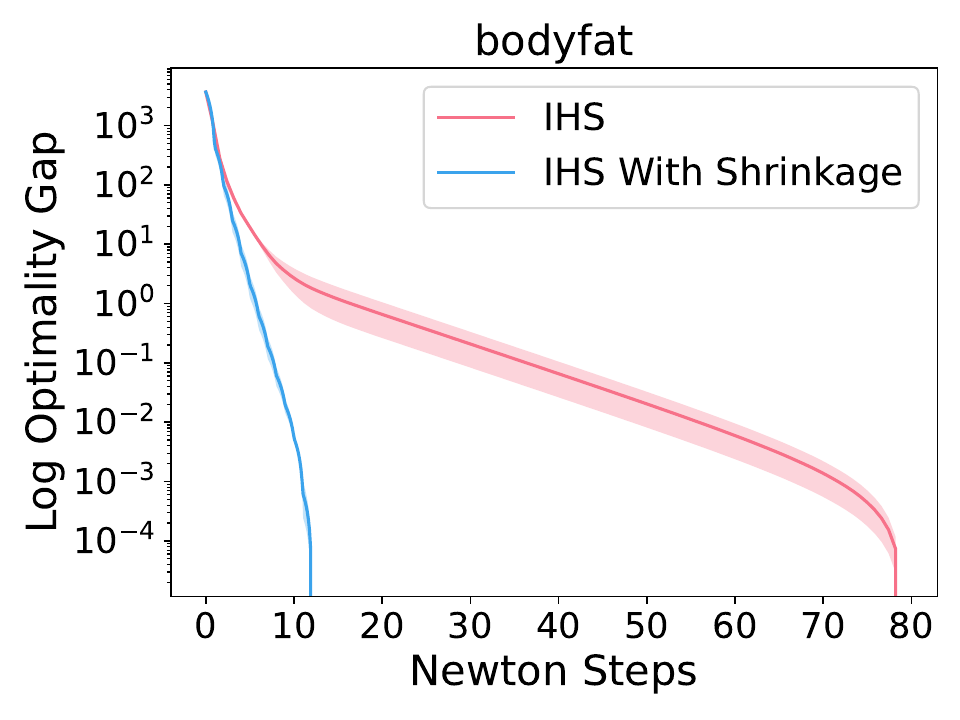}
\end{subfigure}
\hfill
\begin{subfigure}{0.23\linewidth}
\includegraphics[width=\linewidth]{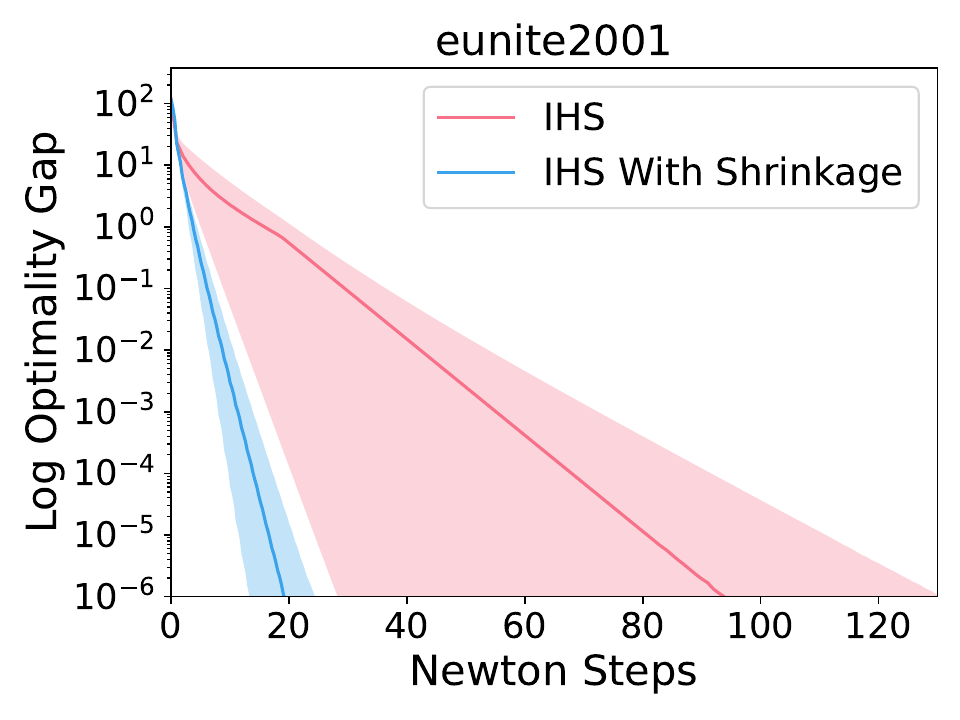}
\end{subfigure}
\hfill
\begin{subfigure}{0.23\linewidth}
\includegraphics[width=\linewidth]{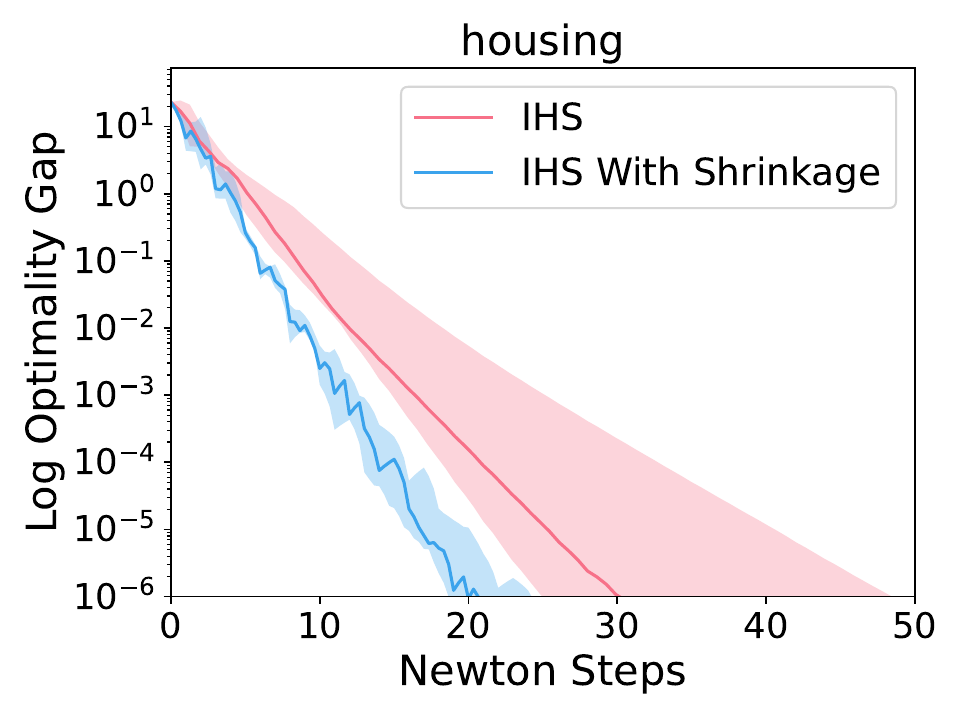}
\end{subfigure}

\begin{subfigure}{0.23\linewidth}
\includegraphics[width=\linewidth]{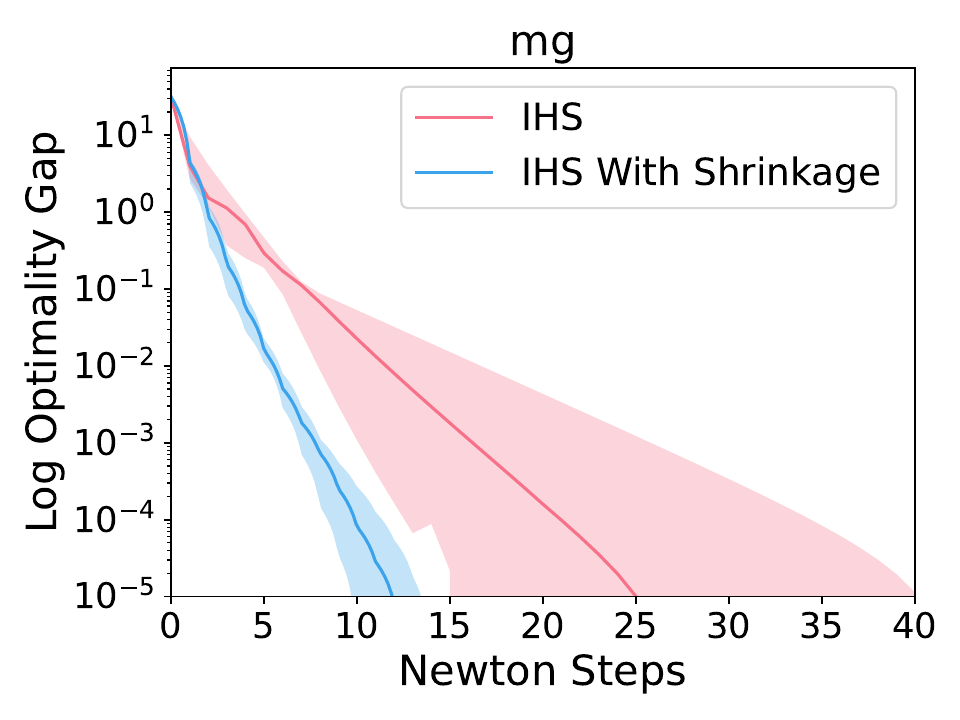}
\end{subfigure}
\hfill
\begin{subfigure}{0.23\linewidth}
\includegraphics[width=\linewidth]{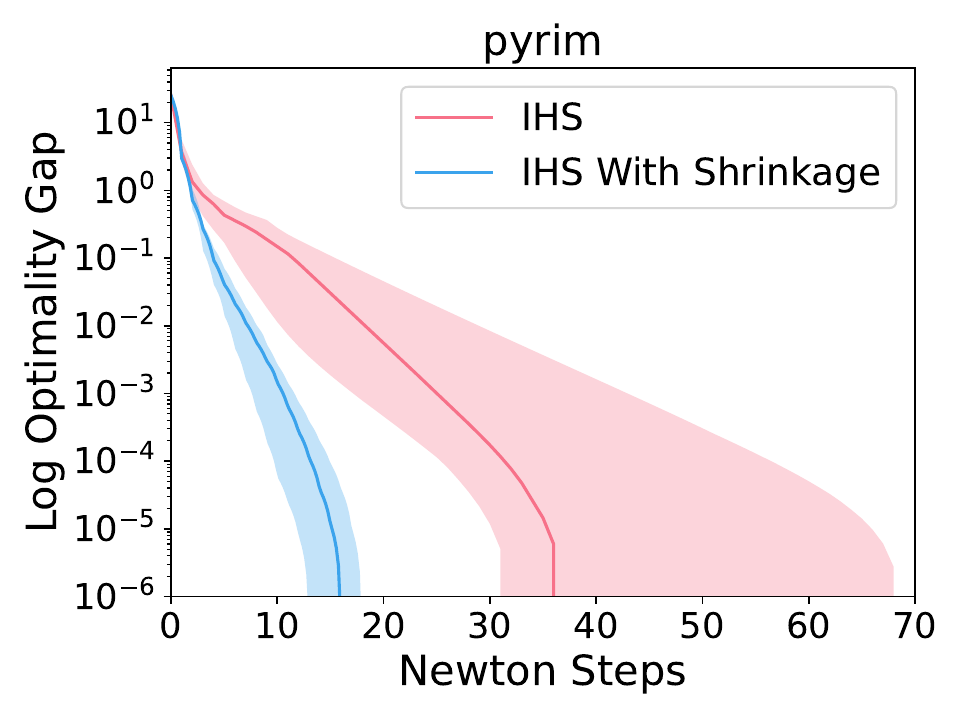}
\end{subfigure}
\hfill
\begin{subfigure}{0.23\linewidth}
\includegraphics[width=\linewidth]{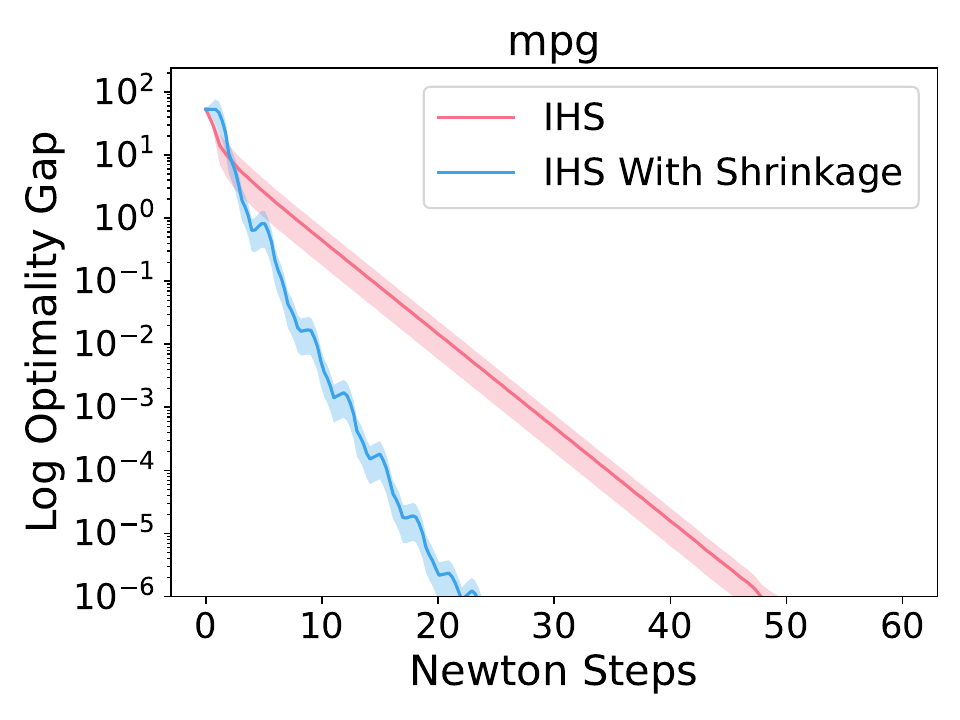}
\end{subfigure}
\hfill
\begin{subfigure}{0.23\linewidth}
\includegraphics[width=\linewidth]{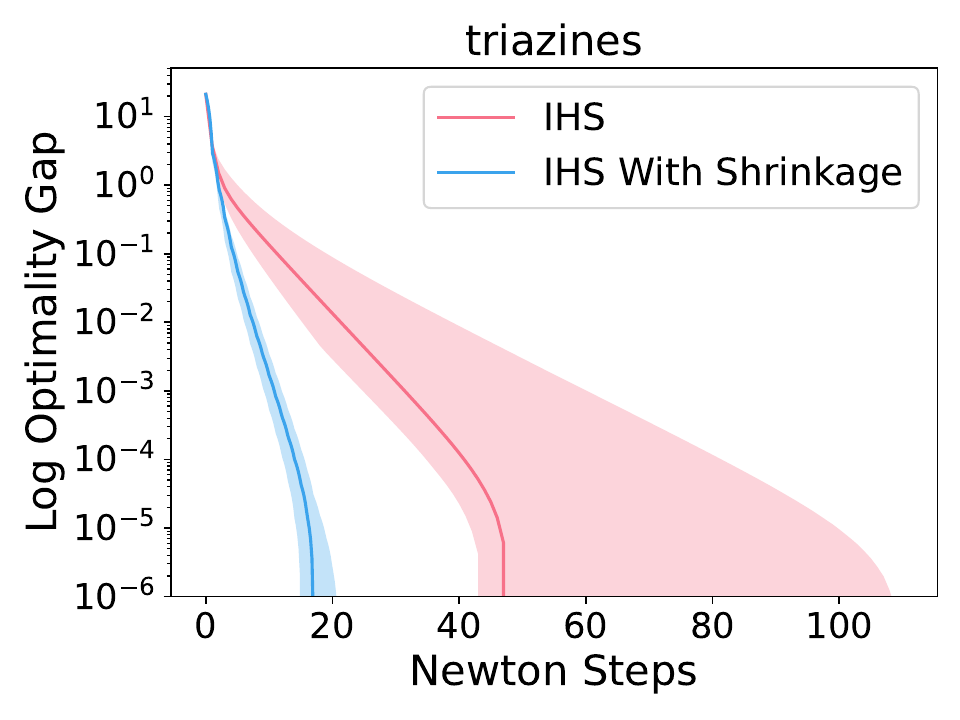}
\end{subfigure}
\caption{Real data experiments on Iterative Hessian Sketch method with exact shrinkage coefficient for ridge regression. Same parameter choice as in Figure \ref{ihs-figure1}.}\label{ihs-figure2}
\end{figure}





\end{document}